\documentclass[11pt,a4paper,leqno,twoside]{amsart}

\usepackage{amssymb}
\usepackage{amsmath}
\usepackage{amscd}
\usepackage[all]{xy}
\usepackage{amsthm}
\usepackage{enumerate}
\date{\today}
\author{Andreas Heider}
\title{Two results from Morita theory \\ of stable model categories}

\newcommand{\MMODD}{\ensuremath{\textnormal{-}\mathrm{Mod}\textnormal{-}}}
\newcommand{\MMOD}{\ensuremath{\textnormal{-}\mathrm{Mod}}}
\newcommand{\MODD}{\ensuremath{\mathrm{Mod}\textnormal{-}}}
\newcommand{\MODDalpha}{\ensuremath{\mathrm{Mod}_\alpha \textnormal{-}}}

\newcommand{\id}{\ensuremath{\mathrm{id}}}

\newcommand{\iso}{\cong}
\newcommand{\tensor}{\otimes}
\newcommand{\sm}{\wedge}
\newcommand{\HOM}{\ensuremath{\mathrm{Hom}}}

\newcommand{\loc}{\ensuremath{\textnormal{-}\mathrm{loc}}}

\newcommand{\Ab}{{\mathrm {Ab}}}
\newcommand{\Sp}{{\mathrm{Sp}}^\Sigma}

\newcommand{\mS}{{\mathbb S}}

\newcommand{\mZ}{{\mathbb Z}}

\newcommand{\mL}{{\mathrm L}}
\newcommand{\mF}{{\mathrm F}}

\newcommand{\mFF}{{\mathcal F}}
\newcommand{\mGG}{{\mathcal G}}
\newcommand{\mOO}{{\mathcal O}}
\newcommand{\mAA}{{\mathcal A}}

\newcommand{\mCC}{{\mathcal C}}
\newcommand{\mEE}{{\mathcal E}}

\newcommand{\mTT}{{\mathcal T}}
\newcommand{\mSS}{{\mathcal S}}
\newcommand{\mUU}{{\mathcal U}}
\newcommand{\mMM}{{\mathcal M}}
\newcommand{\mNN}{{\mathcal N}}
\newcommand{\mRR}{{\mathcal R}}
\newcommand{\mKK}{{\mathcal K}}
\newcommand{\mWW}{{\mathcal W}}
\newcommand{\mII}{{\mathcal I}}
\newcommand{\mJJ}{{\mathcal J}}
\newcommand{\mDD}{{\mathcal D}}

\newcommand{\DE}{{\mD(\mEE^{\op})}}
\newcommand{\DAE}{{\mD_\alpha(\mEE^{\op})}}

\newcommand{\tI}{\widetilde{I}}
\newcommand{\tJ}{\widetilde{J}}
\newcommand{\tF}{\widetilde{F}}
\newcommand{\tH}{\widetilde{H}_0}

\newcommand{\mD}{{\mathsf D}}
\renewcommand{\to}{\,\longrightarrow\,}
\newcommand{\ot}{\,\longleftarrow\,}
\newcommand{\righthook}{\text{\usefont{OML}{cmm}{m}{it} \symbol{44}}}

\newcommand{\incl}{\,\mbox{$\hspace{1mm}\righthook\hspace{-2mm}
\longrightarrow\hspace{1mm}$}\,}

\newcommand{\cof}{\,\hskip-18pt\xymatrix@C=1.2pc{&\ar@{>->}[r]&}\,}
\newcommand{\fib}{\,\hskip-18pt\xymatrix@C=1.2pc{&\ar@{->>}[r]&}\,}

\newcommand{\weq}{\,\stackrel{\sim}{\to}\,}

\DeclareMathOperator{\ev}{ev}

\DeclareMathOperator{\Hom}{Hom}
\DeclareMathOperator{\End}{End}
\DeclareMathOperator{\RHom}{RHom}
\DeclareMathOperator{\map}{map}
\DeclareMathOperator{\colim}{colim}

\DeclareMathOperator{\cofiber}{cofiber}
\DeclareMathOperator{\essim}{essim}
\DeclareMathOperator{\supp}{supp}

\DeclareMathOperator{\cofr}{\, cof\,}
\DeclareMathOperator{\fibr}{\, fib\,}
\DeclareMathOperator{\op}{\, op}
\DeclareMathOperator{\Ho}{Ho}
\DeclareMathOperator{\Ch}{Ch}

\DeclareMathOperator{\Sk}{Sk}

\def\SSS{$S\kern-1.4pt\lower1pt\hbox{\bf.}$}

\numberwithin{equation}{section}

\newtheoremstyle{1}% name
     {18pt}%      Space above
     {15pt}%      Space below
     {\it}%         Body font
     {}%         Indent amount (empty = no indent, \parindent = para indent)
     {\bfseries}% Thm head font
     {.}%        Punctuation after thm head
     {.5em}%     Space after thm head: " " = normal interword space;
     {}%         Thm head spec (can be left empty, meaning `normal')

\newtheoremstyle{1'}% name
     {18pt}%      Space above
     {18pt}%      Space below
     {\it}%         Body font
     {}%         Indent amount (empty = no indent, \parindent = para indent)
     {\bfseries}% Thm head font
     {.}%        Punctuation after thm head
     {.5em}%     Space after thm head: " " = normal interword space;
     {}%         Thm head spec (can be left empty, meaning `normal')

\newtheoremstyle{1''}% name
     {18pt}%      Space above
     {-15pt}%      Space below
     {\it}%         Body font
     {}%         Indent amount (empty = no indent, \parindent = para indent)
     {\bfseries}% Thm head font
     {.}%        Punctuation after thm head
     {.5em}%     Space after thm head: " " = normal interword space;
     {}%         Thm head spec (can be left empty, meaning `normal')

\newtheoremstyle{1'''}% name
     {18pt}%      Space above
     {-15pt}%      Space below
     {\it}%         Body font
     {}%         Indent amount (empty = no indent, \parindent = para indent)
     {\bfseries}% Thm head font
     {.}%        Punctuation after thm head
     {\newline}%     Space after thm head: " " = normal interword space;
     {}%         Thm head spec (can be left empty, meaning `normal')

\newtheoremstyle{2}% name
     {18pt}%      Space above
     {18pt}%      Space below
     {}%         Body font
     {}%         Indent amount (empty = no indent, \parindent = para indent)
     {\bfseries}% Thm head font
     {.}%        Punctuation after thm head
     {.5em}%     Space after thm head: " " = normal interword space;
     {}%         Thm head spec (can be left empty, meaning `normal')

\theoremstyle{1}
  \newtheorem*{thm*}{Theorem}
  \newtheorem{thm}[equation]{Theorem}
  \newtheorem{lemma}[equation]{Lemma}
  \newtheorem{cor}[equation]{Corollary}
  \newtheorem{prop}[equation]{Proposition}
\theoremstyle{1'}

\theoremstyle{1''}
  
\theoremstyle{1'''}
  
\theoremstyle{2}
  \newtheorem{defn}[equation]{Definition}
  \newtheorem{ex}[equation]{Example}
  \newtheorem{exs}[equation]{Examples}
  \newtheorem{rem}[equation]{Remark}
  \newtheorem{notation}[equation]{Notation}
 
\SelectTips{cm}{}

\usepackage{hyperref}
\usepackage{isotope}
\begin{document}

\begin{abstract}
We prove two results from Morita theory of stable
model categories. Both can be regarded as topological versions
of recent algebraic theorems. One is on
re\-colle\-ments of triangulated categories, which have been studied in 
the algebraic case
by J{\o}rgensen. We
give a criterion which answers the following question: 
When is there a re\-colle\-ment
for the derived
category of a given symmetric ring spectrum in terms of two other
symmetric ring spectra?

The other result is on well generated triangulated categories in the
sense of Neeman.
Porta characterizes the algebraic well generated
categories as localizations of derived categories of
DG categories. 
We prove a topological analogon: a topological
triangulated category is well generated if and only if it is
triangulated equivalent to a 
localization of the derived category of a symmetric ring spectrum
with several objects. Here
`topological' means triangulated equivalent to the homotopy
category of a spectral model category.
Moreover, we show that every well generated spectral model category
is Quillen equivalent to a Bousfield localization of a category of modules
via a single Quillen functor.

%We give an analoguous characterization of the topological
%well generated categories,  where
%`topological' means triangulated equivalent to the homotopy
%category of a spectral model category.

\end{abstract}

\maketitle
\pagestyle{myheadings}
\markboth{{\small\textsc{Andreas Heider}}}
{{\small\textsc{Two results from Morita theory of stable model categories}}}

\vspace{1cm}

%\newpage
\tableofcontents
\newpage
\section*{Introduction}

In classical Morita theory \cite{Morita} questions like these are studied:
When are two rings Morita equivalent, that is, when do they
have equivalent module categories? When is an abelian category 
equivalent to the category of modules over some ring?
One result is the following.
An abelian category $\mAA$ with (arbitrary) coproducts is equivalent
to a category of modules if and only if it has a compact projective
generator $P$ \cite[Chapter~II, Theorem~1.3]{Bass}. 
In this case, an equivalence is given by the Hom-functor
\[
\Hom_\mAA(P,-):\mAA\to\MODD\End_\mAA(P) .
\]

A weaker notion than that of classical Morita equivalence is that of 
derived equivalence first considered by Happel: 
two rings are
derived equivalent if their derived categories are equivalent as triangulated
categories. 
Natural questions are: When are two rings derived equivalent? When
is a triangulated category equivalent to the derived category of a ring?
Here, ordinary rings can more generally be replaced by differential
graded rings (DG rings) or DG algebras over some fixed commutative ring 
-- or `several objects versions' of such (DG categories).
These questions about derived Morita equivalence
have been studied among others by Rickard \cite{MR1002456}
and Keller \cite{MR1258406}. As in the
classical case, compact generators and certain Hom-functors play an important
role.

Using the setting of model categories due to Quillen
(cf.~\cite{Quillen} or \cite{Hovey}), 
one can also consider
derived categories of other appropriate ring objects (with possibly
several objects), such as symmetric
ring spectra, and then study similar questions \cite{SS03}.

%We will prove two results from (generalized) Morita theory of stable
%model categories. Both can be regarded as topological versions
%of algebraic theorems, which have recently been proved. One is on
%re\-colle\-ments of triangulated categories, which have been studied in 
%the algebraic case
%by J{\o}rgensen in his paper \cite{Jorgensen}. The other is a characterization
%of algebraic well generated triangulated categories \cite[Theorem~3.9]{DGCats}
%due to Porta.
%In both cases, the paper of Schwede and Shipley \cite{SS03} will
%be of great relevance since it provides the essential tools one 
%needs for studying Morita theory
%of stable model categories.

{\vskip1cm}

\textbf{Recollements.}
A re\-colle\-ment of triangulated categories is a diagram of triangulated
categories
\[
\xymatrix@M=.7pc{\mTT'\ar[rr]^{i_{\ast}}
  &&\mTT \ar@/_7mm/[ll]_{i^{\ast}}
  \ar@/^7mm/[ll]^{i^!}\ar[rr]^{j^{\ast}}
  && \mTT'' \ar@/_7mm/[ll]_{j_!}
  \ar@/^7mm/[ll]^{j_{\ast}}
}
\]
where $(i^{\ast},i_{\ast})$, $(i_{\ast},i^!)$, $(j_!,j^{\ast})$, 
and $(j^{\ast},j_{\ast})$ are adjoint pairs of triangulated functors 
satisfying some more conditions (see Definition~\ref{recollement}). 
This generalizes the notion of triangulated 
equivalence in so far as a re\-colle\-ment
with $\mTT'=0$ (resp.~$\mTT''=0$) is the same as a triangulated
equivalence between $\mTT$ and $\mTT''$ (resp.~$\mTT'$).  
In a re\-colle\-ment, the category $\mTT$ can be
viewed as glued together by $\mTT'$ and $\mTT''$.
The notion has its origins in the theory of perverse sheaves 
in algebraic geometry and appeared first in 
\cite{MR751966}, where the authors show among other things that
a re\-colle\-ment as above together with t-structures on $\mTT'$ and $\mTT''$ 
induces a t-structure on $\mTT$.

J{\o}rgensen \cite{Jorgensen} studies re\-colle\-ments in the
case where the involved triangulated categories are derived categories
of DG algebras
over some fixed commutative ground ring. 
He gives a criterion for the existence of DG algebras $S$ and $T$
and a re\-colle\-ment
\begin{align}\label{Troubardix}
\xymatrix@M=.7pc{\mathsf D(S)\ar[rr]^{i_{\ast}}
  &&\mathsf D(R) \ar@/_7mm/[ll]_{i^{\ast}}
  \ar@/^7mm/[ll]^{i^!}\ar[rr]^{j^{\ast}}
  && \mathsf D(T) \ar@/_7mm/[ll]_{j_!}
  \ar@/^7mm/[ll]^{j_{\ast}}
}\tag{$\ast$}
\end{align}  
of derived categories 
for a given DG algebra $R$ \cite[Theorem~3.4]{Jorgensen}.

The derived category of a DG algebra $R$ can be regarded
as the homotopy category of the model category of
differential graded $R$-modules. More generally, the homotopy category of
every stable model category is a triangulated category in a natural way
\cite[Chapter~7]{Hovey}. This holds in particular for the category of symmetric
spectra in the sense of \cite{HSS} and for the category of modules over a 
(symmetric) ring spectrum.
%which have the model structure as defined in \cite{SS00} 
For 
a ring spectrum $R$ let $\mD (R)$ denote the homotopy category 
of modules over $R$.
Given a ring spectrum $R$ we ask, similar to the differential graded case,
for a criterion for the existence of ring spectra
$S$ and $T$ and a re\-colle\-ment as (\ref{Troubardix}).

One can also study the case where the category of
symmetric spectra is more generally replaced by any `reasonable' 
monoidal stable model
category, including both the case of symmetric spectra and the case of
chain complexes ($\mZ$-graded and unbounded, 
over some fixed commutative ground ring) -- here a monoid
is the same as a DG algebra.
The main theorem of Part~1 is Theorem~\ref{mainthm}, 
which states that a re\-colle\-ment (over a
reasonable monoidal stable model category)
of the form (\ref{Troubardix}) exists if and only if 
there are two objects in $\mD(R)$ which 
satisfy certain finiteness and generating conditions.
%there are $R$-modules
%$B$ and $C$ such that $B$ is self-compact (this is weaker than compact), 
%$C$ is compact and both modules satisfy a certain generating condition for
%$\Ho (R)$
We will proceed in a way similar to J\o rgensen's \cite{Jorgensen}.
However, 
the proofs will sometimes be different and involve the model 
structure.

{\vskip1cm}

\textbf{Well generated categories.}
In his book \cite{Neeman}, Neeman introduces the notion of  well generated
(triangulated) categories, which generalize 
compactly generated categories. They 
satisfy, like the compactly generated categories, Brown representability.
One advantage over the compactly generated ones is that 
the class of well generated categories is
stable under passing to appropriate localizing subcategories and
localizations (cf.~Proposition~\ref{Neeman's}).
A classical example of a compactly generated triangulated category
occurring in algebra is the derived category $\mD(\mAA)$ 
of a DG algebra, or more generally, of a DG category
$\mAA$, which is 
just a `several objects version' 
of a DG algebra. By Proposition~\ref{Neeman's},
all (appropriate) localizations of $\mD(\mAA)$ are well generated again.
One could ask whether the converse is also true, that is, whether
every well generated triangulated category $\mTT$ is, up to triangulated
equivalence, a localization of the
derived category $\mD(\mAA)$ for an appropriate DG category $\mAA$.
Porta gives a positive answer if $\mTT$ is algebraic 
\cite[Theorem~5.2]{Porta}. 
This characterization of algebraic well generated
categories can be regarded as a refinement of \cite[Theorem~4.3]{MR1258406},
where Keller characterizes the algebraic compactly generated categories
with arbitrary coproducts,
up to triangulated equivalence, as the derived categories of DG categories.

A topological version of Keller's theorem 
has been proved in 
\cite[Theorem~3.9.3(iii)]{SS03}: the compactly generated topological categories
are characterized, up to triangulated equivalence, as the `derived categories
of ring spectra with several objects'. This needs some explanation. 
A spectral category is a ring spectrum with several objects, 
i.e., 
a small category enriched over the symmetric monoidal model category
of symmetric spectra in the sense of \cite{HSS}. Generalizing the
correspondence between ring spectra and DG algebras, spectral 
categories are the topological versions of DG categories.
The derived category of a spectral category $\mEE$ is the homotopy category
of the model category of $\mEE$-modules.
By a topological triangulated category we mean any triangulated category
equivalent to the homotopy category of a spectral model category.
This is not the same as (but closely related with)
a topological triangulated category in the sense 
of \cite{alg/top}, where any triangulated category equivalent to a full
triangulated subcategory of the homotopy category of a stable model
category is called topological.
By \cite[Theorem~3.8.2]{SS03}, the homotopy category of
any simplicial, cofibrantly
generated and proper stable model category is topological. 
%Another class of topological triangulated categories can be 
%found in \cite{Dugger}.

The aim of Part~2 of this paper is to give
a characterization of the topological well generated categories.
We will prove that 
every topological well generated triangulated category is 
triangulated equivalent to a localization of the
derived category of a small spectral category
such that the acyclics of the localization are
generated by a set.
On the other hand, 
the derived category of a small spectral category is compactly 
generated  by the free modules \cite[Theorem~A.1.1(ii)]{SS03}
and the class of well generated categories is stable under localizations
(as long as the acyclics are generated by a set), 
cf.~Proposition~\ref{Neeman's}.
Hence we get the following characterization (Theorem~\ref{characterization}): 
The topological 
well generated categories
are, up to triangulated equivalence, exactly the localizations 
(with acyclics generated by a set) of derived categories
of spectral categories.

Finally, we use Hirschhorn's existence theorem for Bousfield localizations
\cite[Theorem~4.1.1]{Hirschhorn} 
to give a lift to the level of model categories
in the following sense (Theorem~\ref{lift}): Every 
spectral model category which has a well generated homotopy category
admits a Quillen
equivalence to a Bousfield localization of a model category of modules
(over some endomorphism spectral category).
While a rough slogan of a main result in \cite{SS03} is, 
`Compactly
generated stable model categories are categories of modules',
the corresponding slogan of our result is, 
`Well generated 
stable model categories
are localizations of categories of modules'.
\vspace{1cm}

\textbf{Terminology and conventions.}
Our main reference for triangulated category theory is Neeman's
book \cite{Neeman} and thus we use basically his terminology. One
exception concerns the definition of a triangulated category:
since we are interested in triangulated categories arising from topology
we allow the suspension functor $\Sigma :\mTT \to \mTT$ of a triangulated
category $\mTT$ to be a self-equivalence of $\mTT$ and do not require it
to be an automorphism.
%(i.e., $\Sigma$
%is not necessarily an automorphism but it has only an inverse $\Sigma ^{-1}$ 
%up to natural isomorphism). 
In other words, we take the definition of 
a triangulated category that 
can be found, for example, in \cite[Appendix~2]{Margolis}.

Another point of difference is that all our categories are supposed to have
Hom-\emph{sets}, not only Hom-\emph{classes}. 
(In the terminology of \cite{Neeman},
the morphisms between
two objects are allowed to form a class. If, between
any two objects, 
they actually form a set, 
then the category is said to have `small Hom-sets' in
\cite{Neeman}.)
Such triangulated `meta'-categories with Hom-classes arise in the context
of Verdier quotients (cf.~Remark~\ref{Keller/Neeman}(2)). But it
turns out that all Verdier quotients we need to consider are in fact `honest'
categories, that is, the morphisms between any two objects form a set.

When we say that a category has (co-)products, we always mean
arbitrary set-indexed (co-)products. Adjoint pairs of functors will arise
throughout the paper. We use the convention according to
which in diagrams the left
adjoint functor is drawn above the right adjoint. If we have three functors
\[
\xymatrix@M=.7pc{\mCC \ar[rr]^G
  && \mDD \ar@/_7mm/[ll]_{F}
  \ar@/^7mm/[ll]^H
}
\]
such that $(F,G)$ and $(G,H)$ are adjoint pairs we will call $(F,G,H)$
an adjoint triple.

\vspace{1cm}

\textbf{Acknowledgements.}
First of all, I would like to thank my advisor Stefan Schwede 
for suggesting this project to me and for always motivating and
supporting me in carrying it out. 
I am deeply indebted to Bernhard Keller 
for several helpful conversations concerning in particular
the second part of this paper and to
Phil Hirschhorn for helpful discussions on cellular model categories.
I am grateful to Henning Krause and to Marco
Porta for their interest in the
subject of this paper and discussions about it.
Furthermore, it is a pleasure to thank
Steffen Sagave and Arne Weiner
for many
comments on an pre-version
of this paper.
Thanks for non-mathematical support go to my family in Str\"ohen and to
the Posaunenchor der Lutherkirche in Bonn.

\newpage
\part{Stable model categories and recollements}
%\input{Section1}
%%%In Section~1 we will state some 
%facts on module categories over closed symmetric
%monoidal model categories and prove some basic (probably well-known) lemmas.
We start in Section~1 with a recollection of some notions and lemmas from
triangulated category theory which will also be important in Part~2 of this
paper. We will then discuss the definition of re\-colle\-ments and some 
of their properties. 
Re\-colle\-ments are closely related to localizations
and colocalizations. We consider this relation in Section~1.2. 
An example of a re\-colle\-ment coming from stable
homotopy theory is described in Section~1.3.

In Section~2, we introduce `reasonable' stable model categories, that is,
closed symmetric monoidal model categories which are stable and have
some other nice properties that allow us to study Morita theory over
such categories. Both symmetric spectra and chain complexes are examples
of reasonable stable model categories. In \cite[Theorem~3.9.3]{SS03}, 
Schwede and Shipley 
relate spectral model categories to certain categories of modules via
a Quillen pair. We consider a version thereof over reasonable stable
model categories in Section~2.3. In Section~2.4, we prove our main result, 
Theorem~\ref{mainthm}, which 
gives a criterion for the existence of a re\-colle\-ment for the derived category
$\mD(R)$, where $R$ is a monoid in a reasonable stable model category. 
 
\section{Recollements}

\subsection{Definition and formal properties}

Let us recall some general notions from triangulated category theory. 

By a \emph{triangulated subcategory} $\mUU$ of $\mTT$ we mean a non-empty full 
subcategory which is 
closed under (de-)suspensions and triangles (if two out of
three objects in a triangle are in $\mUU$ then so is the third).
Note that $\mUU$ is then automatically 
closed under finite coproducts and it contains the whole isomorphism
class of an object (i.e., $\mUU$ is `replete').
One says $\mUU$ is
\emph{thick} if it is closed under direct summands. If $\mTT$ has (arbitrary)
coproducts, $\mUU$ 
is called \emph{localizing} whenever it is closed under coproducts.
If $\mUU$ is localizing it is automatically thick (since in this case,
idempotents split in $\mUU$ \cite[Proposition~1.6.8]{Neeman}).

If $\mTT$ and $\mTT'$ are triangulated categories with suspension functors 
$\Sigma$ and $\Sigma '$ a \emph{triangulated}
(or \emph{exact}) functor is a functor 
$F:\mTT\to\mTT'$ together with a natural isomorphism 
$F\circ \Sigma\stackrel{\iso}{\to}\Sigma '\circ F$ such that for every
exact triangle 
\[X\to Y \to Z \to \Sigma X\]
in $\mTT$ we get an exact triangle
\[F(X)\to F(Y)\to F(Z) \to \Sigma 'F(X)\] 
in $\mTT'$, whose last arrow involves the
natural isomorphism. Unless stated otherwise, by a functor between
triangulated
categories we always mean a triangulated one.
The \emph{kernel} of $F$ is the thick triangulated
subcategory of $\mTT$ containing the
objects which are mapped to zero in $\mTT'$, 
\[\ker F=\{X\in \mTT\, |\, F(X)\iso 0\}.\]
If $\mTT$ and $\mTT'$ have coproducts and $F$ preserves them, then $\ker F$ is 
localizing. One cannot expect the image of $F$ to be a triangulated 
subcategory of $\mTT'$. Even if $F$ is full the image need not be replete.
But the \emph{essential image} of $F$,
\[
\essim F=\{X'\in\mTT'\, |\,X'\iso F(X) \textnormal{ for some } X\in\mTT\},
\]
is a triangulated subcategory if $F$ is a full (!) triangulated functor.
It is localizing if $\mTT$ and $\mTT'$ contain coproducts 
and $F$ preserves them.

If $\mSS$ is a set of objects of a triangulated category $\mTT$ with coproducts
then $\langle\mSS\rangle$ 
denotes 
%the localizing triangulated subcategory generated by $\mSS$, that is,
the smallest localizing triangulated subcategory of $\mTT$ containing $\mSS$.
(It does exist, it is just the intersection of all localizing triangulated
subcategories containing $\mSS$.) 
%One says $\mSS$ is a set of \emph{generators}
%for $\mTT$ if $\langle S\rangle = \mTT$. 

\begin{ex}
If $R$ is a DG algebra, 
that is, a monoid in the symmetric monoidal model category
of chain complexes, then $R$ considered as a module over itself is a generator
for $\mD(R)$, the derived category of $R$. 
This is a special case of \cite[Section~4.2]{MR1258406}. 

Similarly, if $R$ is a symmetric ring spectrum, that is a monoid in the 
symmetric monoidal model
category of symmetric spectra, then $R$ is a generator for the derived category
$\mD(R)$, which
is defined as the homotopy category of the stable
model category of $R$-modules \cite[Theorem~A.1.1]{SS03}.
\end{ex}

For $F: \mTT\to\mTT'$ let $F(\mSS)$ be the
set of all $F(X)$ with $X\in \mSS$. We have the following (probably
well-known)

\begin{lemma}\label{essim}
Let $F:\mTT\to\mTT'$ be a coproduct preserving triangulated functor between
triangulated categories with coproducts and $\mSS$ a set of objects
in $\mTT$. 
\begin{enumerate} 
\item[\textnormal{(i)}] There is an inclusion of (not necessarily triangulated)
full subcategories 
\[
\essim\left( F\!\mid _{\langle\mSS\rangle}\right) 
\subset \langle F(\mSS)\rangle. 
\]
\item[\textnormal{(ii)}] 
If $F$ is full then
\[
\essim \left(F\!\mid _{\langle\mSS\rangle}\right) = \langle F(\mSS)\rangle
\]
as triangulated categories
\end{enumerate}
\end{lemma}

\begin{proof}
Those $X$ in $\mTT$ for which $F(X)$ is in $\langle F(\mSS)\rangle$
form a localizing triangulated subcategory containing $\mSS$ and hence 
containing $\langle\mSS\rangle$. So the image (and, as a consequence,
the essential image) of 
$F\!\mid _{\langle\mSS\rangle}$ is contained in $\langle F(\mSS)\rangle$,
as was claimed in (i).
For the other inclusion note that since $F$ is full, 
$\essim F\!\mid _{\langle\mSS\rangle}$ is a localizing triangulated subcategory
of $\mTT'$ which contains $F(\mSS)$. This shows (ii).
\end{proof}

The following lemma is often useful, too.
\begin{lemma}\label{useful}
Let $F$, $G:\mTT\to\mTT'$ be coproduct preserving triangulated functors
between triangulated categories with coproducts and $\eta:F\to G$
a natural transformation of triangulated functors. Then those objects $X$
for which $\eta_X$ is an isomorphism form a localizing triangulated
subcategory of $\mTT$. \hfill $\square$
\end{lemma}

As a definition for re\-colle\-ments we take J\o rgensen's 
\cite[Definition~3.1]{Jorgensen}.
\begin{defn}\label{recollement}
A \emph{recollement} of triangulated categories is a diagram of
triangulated categories
\[
\xymatrix@M=.7pc{\mTT'\ar[rr]^{i_{\ast}}
  &&\mTT \ar@/_7mm/[ll]_{i^{\ast}}
  \ar@/^7mm/[ll]^{i^!}\ar[rr]^{j^{\ast}}
  && \mTT'' \ar@/_7mm/[ll]_{j_!}
  \ar@/^7mm/[ll]^{j_{\ast}}
}
\]
such that
\begin{enumerate}[(i)]
\item both $(i^\ast,i_\ast,i^!)$ and 
$(j_!, j^\ast, j_\ast)$ are adjoint triples, 
that is, $(i^{\ast},i_{\ast})$, $(i_{\ast},i^!)$, $(j_!,j^{\ast})$, 
and $(j^{\ast},j_{\ast})$ are adjoint pairs of triangulated functors,
\item $j^\ast i_\ast =0$,
\item\label{fullyfaithful}
the functors $i_\ast$, $j_!$, and $j_\ast$ are 
fully faithful,
\item for each object $X$ in $\mTT$ there are exact triangles
\begin{enumerate}[(a)]
\item $j_!j^\ast X \to X \to i_\ast i^\ast X \to \Sigma j_!j^\ast X$,
\item $i_\ast i^!X \to X \to j_\ast j^\ast X \to \Sigma i_\ast i^!X$,
\end{enumerate}
where the maps to $X$ are counit maps, the maps out of $X$ are unit maps,
and $\Sigma$ denotes the suspension.
\end{enumerate}
\end{defn}

Sometimes we will drop the structure functors $i^\ast$, $i_\ast$, $i^!$, 
$j_!$, $j^\ast$, and $j_\ast$ from the notation and simply write
$(\mTT' ,\mTT , \mTT'')$ for a re\-colle\-ment.

\begin{rem}\label{rem}
Here are some formal properties.
\begin{enumerate}[(1)]
\item Being a left (resp.~right) adjoint of $j^\ast i_\ast =0$, the composition
of the upper (resp.~lower) functors in a re\-colle\-ment is zero:
\[i^\ast j_!=0 \quad\textnormal{and}\quad i^!j_\ast =0.\]
\item Provided condition (i) in Definition~\ref{recollement} holds, 
condition~(\ref{fullyfaithful}) is equivalent to
the following. For $X'$ in $\mTT'$ and $X''$ in $\mTT''$ the counit and unit
maps 
\[i^\ast i_\ast X'\to X',\quad j^\ast j_\ast X'' \to X'', \quad
   X'\to i^! i_\ast X', \quad X''\to j^\ast j_!X''
\]
are natural isomorphisms.
\item Composing the natural isomorphism $i^\ast i_\ast X'\to X'$ in (2)
with $i_\ast$
we get that the restriction of $i_\ast i^\ast$ to the essential image 
of 
$i_\ast$ is naturally isomorphic to the identity functor.
\item The third arrow in the exact triangles (a) and (b) of Definition 
\ref{recollement}(iv) is natural in $X$ and uniquely determined. 
To see the naturality consider a diagram
\begin{equation*}
\xymatrix@M=.7pc{j_! j^\ast X \ar[r] \ar[d]_{j_! j^\ast (f)}
& X \ar[r]^-{\eta_X} \ar[d]_f & i_\ast i^\ast X \ar[r]^-{\psi_X}
 \ar@{.>}[d]_{\bar f} 
& j_! j^\ast \Sigma X\ar[d]^{j_! j^\ast \Sigma (f)}\\
j_! j^\ast Y \ar[r]
& Y \ar[r]^-{\eta_Y} & i_\ast i^\ast Y \ar[r]^-{\psi_Y}&j_! j^\ast \Sigma Y
}
\end{equation*}
where the rows are exact triangles as in Definition~\ref{recollement}(iv)(a)
and solid arrows are given such that the left square commutes. The
axioms of a triangulated category guarantee the existence of a dotted
arrow $\bar f$ such that the whole diagram commutes. 

We claim that
there is only one arrow $\bar f$ such that the square in the middle commutes,
that is, $\bar f \eta_X = \eta_Y f$. It is enough to consider the case $f=0$
and to show that $\bar f$ is necessarily zero, too. 
But $f=0$ implies $\bar f \eta_X =0$ and since the representing
functor
$\mTT (-,i_\ast i^\ast Y)$ is cohomological there exists an arrow 
$g:j_! j^\ast \Sigma X\to i_\ast i^\ast Y$ such that $g \psi_X = \bar f$.
Now the adjoint map of $g$ with respect to the adjoint pair $(j_!,j^\ast)$
is a map into $j^\ast i_\ast i^\ast Y$ which is zero by Definition 
\ref{recollement}(ii). Hence $g$ itself is zero and so is $\bar f$, proving
our claim.

As the unit $\eta$ is a natural transformation, the map $i_\ast i^\ast (f)$
satisfies $i_\ast i^\ast (f) \eta_X = \eta_Y f$ and consequently
$\bar f=i_\ast i^\ast (f)$. Since the right square in the diagram
is commutative, this shows
the naturality of $\psi$. Taking $f$ to be the identity arrow on $X$ shows
the uniqueness of the third arrow $\psi_X$.
\item Replacing any of $\mTT$, $\mTT'$ or $\mTT''$ in a re\-colle\-ment
by an equivalent triangulated 
category still gives a re\-colle\-ment.
\item A re\-colle\-ment with $\mTT' =0$ is the same as an equivalence 
$\mTT\simeq\mTT''$ of triangulated categories. Namely $i_\ast =0$ implies by
Definition~\ref{recollement}(iv)(b) that $X\iso j_\ast  j^\ast X$, 
so $j_\ast$ is essentially 
surjective on objects. Since $j_\ast$ is also fully faithful
by Definition~\ref{recollement}(iii) it is an equivalence of categories 
with inverses $j_\ast$ and $j_!$ (which are hence isomorphic).
Similarly, a re\-colle\-ment with 
$\mTT''=0$ is the same as a triangulated equivalence
$\mTT'\simeq\mTT$.
\item A map of re\-colle\-ments from $(\mTT' ,\mTT , \mTT'')$
to $(\mUU' , \mUU , \mUU'')$ consists of three triangulated functors
$F':\mTT' \to \mUU'$, $F:\mTT \to \mUU$, $F'':\mTT'' \to \mUU''$ 
which commute (up to natural
isomorphism)
with the structure functors. It is a theorem of Parshall and Scott 
\cite[Theorem~2.5]{Parshall} that a
map of re\-colle\-ments is determined (up to natural isomorphism) by $F'$ and
$F$ (resp.~$F$ and $F''$).  Furthermore, if two of $F'$, $F$ and $F''$ are
equivalences then so is the third. This is not true for re\-colle\-ments
of abelian categories, see \cite[Section 2.2]{Pirashvili}.
\item For every re\-colle\-ment one has  
\[\essim i_\ast = \ker j^\ast,\quad\essim j_!=\ker i^\ast,\quad
\essim j_\ast = \ker i^!.
\] 
Consider, for example, the first equality. 
The inclusion $\essim i_\ast\subset\ker j^\ast$
follows immediately from $j^\ast i_\ast =0$. If, on the other hand,
$j^\ast X=0$, then the third term in the exact triangle
\[
i_\ast i^!X \to X \to j_\ast j^\ast X \to \Sigma i_\ast i^!X
\]
of Definition~\ref{recollement}(iv)(b) vanishes so that the first map
is an isomorphism and thus $X\in\essim i_\ast$.

Since $i_\ast$ is fully faithful we have an equivalence of
triangulated categories $\mTT'\simeq\essim i_\ast$ and hence
$\mTT' \simeq\ker j^\ast$. Hence, due to Remark~\ref{rem}(5), 
every re\-colle\-ment is `equivalent'
to the re\-colle\-ment
\[
\xymatrix@M=.7pc{\ker j^\ast\ar[rr]^{\iota_\ast}
  &&\mTT \ar@/_7mm/[ll]_{\iota^{\ast}}
  \ar@/^7mm/[ll]^{\iota^!}\ar[rr]^{j^{\ast}}
  && \mTT''\, , \ar@/_7mm/[ll]_{j_!}
  \ar@/^7mm/[ll]^{j_{\ast}}
}
\]
where $\iota_\ast$ is the inclusion with left (resp.~right) adjoint
$\iota^\ast$ (resp.~$\iota^!$).
\end{enumerate}
\end{rem}

\begin{ex} %SIEHE HARTSHORNE KAP. II ÜBG. 1.19!
The following is the classical example of a re\-colle\-ment arising in 
algebraic geometry \cite[Section~1.4.1]{MR751966}. 
Let $X$ be a topological space, 
$U$ an open subspace and $F$ the complement of $U$ in $X$. Given a
sheaf $\mOO _X$ of commutative rings on $X$, we denote the restricted
sheaves of rings on $U$, resp.~$F$, by $\mOO _U$, 
resp.~$\mOO_F$, 
and the three categories of sheaves of left modules by $\mOO _X \MMOD$,
$\mOO _U \MMOD$, and $\mOO _F \MMOD$. We have six functors
\[
\xymatrix@M=.7pc{\mOO _F \MMOD\ar[rr]^{i_{\ast}}
  &&\mOO _X \MMOD\ar@/_7mm/[ll]_{i^{\ast}}
  \ar@/^7mm/[ll]^{i^!}\ar[rr]^{j^{\ast}}
  && \mOO _U \MMOD \ar@/_7mm/[ll]_{j_!}
  \ar@/^7mm/[ll]^{j_{\ast}}
}
\]
where $i^\ast$ and $j^\ast$ are restriction functors, $i_\ast$ and $j_\ast$ are
direct image functors, and $j_!$ is the functor which extends a sheaf on
$U$ by $0$ outside $U$ to the whole of $X$, i.e., for every
$\mOO _U$-module $\mFF$ and every open subset $V$ of $X$ we have
$j_!\mFF (V)=\mFF(V)$ 
if $V\subset U$ and $j_!\mFF (V)=0$ else. Finally, $i^!$ is defined by
\[(i^!\mGG)(V\cap F) = \{ s\in \mGG(V)\, |\, \supp (s) \subset F \}\] for every 
$\mOO _X$-module $\mGG$ and every open subset $V$ of $X$.

Let $\mD^+ (\mOO_F)$, $\mD^+ (\mOO_X)$, and $\mD^+(\mOO_U)$ 
be the corresponding derived
categories of left bounded complexes. The derived functors of  
$i^\ast$, $i_\ast$, $i^!$, 
$j_!$, $j^\ast$, and $j_\ast$ exist and yield a re\-colle\-ment
\[
\xymatrix@M=.7pc{\mD^+(\mOO _F) \ar[rr]^{i_{\ast}}
  &&\mD^+(\mOO _X) \ar@/_7mm/[ll]_{i^{\ast}}
  \ar@/^7mm/[ll]^{i^!}\ar[rr]^{j^{\ast}}
  &&\mD^+(\mOO _U)\, . \ar@/_7mm/[ll]_{j_!}
  \ar@/^7mm/[ll]^{j_{\ast}}
}
\] 
\end{ex}

\subsection{Localization and colocalization}

It turns out that the data of a re\-colle\-ment is essentially 
the same as a triangulated functor $j^\ast$ which admits 
both a \emph{localization functor} $j_!$
and a \emph{colocalization functor} $j_\ast$. 
These two notions are defined as follows.

\begin{defn}\label{loc}
If a triangulated functor $F:\mTT\to\mUU$ admits a fully faithful right 
adjoint $G:\mUU\to\mTT$ we call $G$ a \emph{localization functor} and $\mUU$
a \emph{localization} of $\mTT$. The objects in the kernel of $F$ are 
called \emph{($F$-)acyclic} and those objects $X\in\mTT$ for which the
unit of the adjunction $X\to GF(X)$ is an isomorphism (or, equivalently,
which are in the essential image of $G$) are called 
\emph{($F$-)local}.

Dually, 
if $F:\mTT\to\mUU$ admits a fully faithful left 
adjoint $H:\mUU\to\mTT$ we call $H$ a \emph{colocalization functor} and $\mUU$
a \emph{colocalization} of $\mTT$. The objects in the kernel of $F$ are called
\emph{($F$-)acyclic} and those objects $X\in\mTT$ for which the
counit of the adjunction $HF(X)\to X$ is an isomorphism (or, equivalently,
which are in the essential image of $H$) are called 
\emph{($F$-)colocal}.
\end{defn}

Since by \cite[Appendix~2, Proposition~11]{Margolis} the adjoint 
of a triangulated functor is itself triangulated, localization and 
colocalization functors are always triangulated. 

\begin{rem}\label{Minerva}
If $F:\mTT\to\mUU$ 
admits a localization functor $G:\mUU\to\mTT$, then $\mUU$ is triangulated 
equivalent to $\essim G$. The composition $GF:\mTT\to \essim G$ has the
inclusion $\essim G\incl \mTT$ as a right adjoint. In other words, 
the localization $\mUU$ of $\mTT$ is equivalent to the triangulated
subcategory of local objects, which can then be regarded as a localization
of $\mTT$ with exactly the same acyclics as the original localization
of $\mTT$.
\end{rem}

\begin{rem}\label{Keller/Neeman}
Let us compare our definition of localization with others occurring
in the literature.
\begin{enumerate}[(1)]
\item Keller's definition is slightly different from
ours, see \cite[Section~3.7]{DGCats}: in addition to our definition,
the kernel of 
$F:\mTT \to\mUU$ is supposed to be generated by a \emph{set} of objects.
(The reason for this is that under this additional technical assumption
a localization of a \emph{well generated} triangulated category is
again well generated, cf.~Proposition~\ref{Neeman's}.)

\item The definition given in Neeman's book \cite[Definition~9.1.1]{Neeman} 
is the following.
Given a thick triangulated subcategory $\mSS$ of $\mTT$,
there always exists a \emph{Verdier quotient} $\mTT / \mSS$ together with
a universal functor $\mTT \to \mTT / \mSS$ with kernel $\mSS$  
\cite[Theorem~2.1.8 and Remark~2.1.10]{Neeman}. 
In Neeman's terminology, the Hom-`sets' of this triangulated
category $\mTT / \mSS$ are not necessarily small, 
that is, they do not form sets but only classes, and 
hence $\mTT / \mSS$ is not an honest category in general.
If the Verdier quotient functor
$F:\mTT \to \mTT / \mSS$ admits a fully faithful right adjoint $G:\mTT/\mSS
\to \mTT$ then $\mTT /\mSS$ is called a \emph{Bousfield localization} and 
the functor $G$ is called
a \emph{Bousfield localization functor}.
It is a consequence of \cite[Theorem~9.1.16]{Neeman} that, if $\mTT /\mSS$
is a Bousfield localization, $\mTT /\mSS$ is an honest category (i.e.,
has small Hom-sets).

A Bousfield localization in Neeman's sense is in particular a localization
as in Definition~\ref{loc}. Namely the right adjoint $G$, if it exists,
is automatically fully faithful. (To see this, it is enough to show
that the counit $\varepsilon$ 
of the adjunction $(F,G)$ is an isomorphism. Since $F$
is the identity on objects one has only to check that $\varepsilon F$ is
an isomorphism. But this follows from \cite[Lemma~9.1.7]{Neeman}.)
On the other hand, by part (iii) of Lemma~\ref{mylemma}(b) below,
a localization in our sense is always a Bousfield localization
up to triangulated equivalence.

Hence Neeman's notion of Bousfield localization is essentially
equivalent to our notion of localization
as in Definition~\ref{loc}.

\item In \cite{HPS} the authors consider 
stable
homotopy categories, i.e., triangulated categories endowed with
a closed symmetric monoidal product $\sm$
 and with a certain set of generators --
for the complete definition see \cite[Definition~1.1.4]{HPS}. 
They define a localization functor \cite[Definiton~3.1.1]{HPS} 
on a stable homotopy category $\mCC$
to be a pair $(L,i)$, where $L:\mCC\to\mCC$ is a 
triangulated functor and $i:\id _\mCC \to L$ is a natural transformation
such that
\begin{enumerate}[(i)]
\item the natural transformation $Li:L\to L^2$ is an isomorphism,
\item for all objects $X$, $Y$ in $\mCC$ the map 
$i^\ast_X : \mCC (LX, LY)\to\mCC (X,LY)$
given by precomposition with $i_X$ is an isomorphism,
\item if $LX=0$ then $L(X\sm Y)=0$ for all $Y$.
\end{enumerate}
The $L$-local objects in $\mCC$
are (by definition in \cite{HPS})
the objects $Y$ for which $i_Y$ is an isomorphism or, equivalently,
which are isomorphic to some $LX$.
If $\mCC_L$ is the full subcategory of $L$-local objects then
$L:\mCC\to \mCC_L$ is left adjoint to the inclusion $\mCC_L \incl \mCC$. In
other words: we have a localization of triangulated categories as
in Definition~\ref{loc}, and the two notions of $L$-local objects
(ours and that
of \cite{HPS}) coincide.
Note that we did not use condition (iii), which involves the monoidal 
structure.

On the other hand, if we are given a functor $F:\mTT\to\mUU$ admitting 
a localization functor $G:\mUU\to\mTT$, the composite $GF$ together with
the unit of the adjunction $\id_\mTT\to GF$ satisfies the first and the second
of the above conditions.
In so far, if we ignore the monoidal structure, our definition and the one
in \cite{HPS} are equivalent.

Dualizing this definition of localization leads 
to the notion of colocalization of stable homotopy categories, see 
\cite[Definition~3.1.1]{HPS}. Each localization $L$ on $\mCC$ 
determines a colocalization
$C$ on $\mCC$ and vice versa \cite[Lemma~3.1.6]{HPS}.
Two such correspond if and only if
there is an exact triangle
\[
CX \to X \to LX \to \Sigma (CX)
\]
where first map comes from the natural transformation of the colocalization
$C$
and the second from natural transformation of the localization $L$.
For each such localization-colocalization pair $(L,C)$ we have
$\essim L = \ker C$ and $\essim C = \ker L$. Hence the $L$-local objects
are exactly the $C$-acyclics and the $C$-colocals are exactly
the $L$-acyclics.

\end{enumerate}
\end{rem}

\begin{defn}\label{perp}
For a class $\mAA$ of objects in a triangulated category $\mTT$, the category
$\mAA ^\perp$ is defined as the full subcategory 
of $\mTT$ containing those objects which do not receive non-zero graded maps 
from $\mAA$,
that is,
\[
\mAA^\perp =\{ X\in\mTT\, |\,\mTT (\Sigma^nA,X)\iso 0 \textnormal{ for each } 
n\in\mZ \textnormal{ and each } A\in \mAA\}.
\]
In the case where $\mAA$ consists
only of one object $A$, we simply write $A^\perp$ for $\mAA^\perp$.
Dually, we define 
\[
\isotope{\perp}{}{\mAA}=\{ X\in\mTT\, |\,\mTT (X,\Sigma^nA)
\iso 0 \textnormal{ for each } 
n\in\mZ \textnormal{ and each } A\in \mAA\}.
\] 
\end{defn}

Note that 
$\mAA^\perp$ is a thick triangulated 
subcategory of $\mTT$, which is colocalizing 
(i.e., closed under products) if $\mTT$ has products. 
It is localizing if $\mTT$ has coproducts
and all objects in $\mAA$ are compact, whereas
$\isotope{\perp}{}{\mAA}$ is always
a localizing triangulated subcategory if $\mTT$ has coproducts.

The reader should be warned that there is not a standardized use of
$\mAA^\perp$ and $\isotope{\perp}{}{\mAA}$ in the literature.
Neeman \cite[Definitions~9.1.10 and 9.1.11]{Neeman}
writes $\mAA^\perp$ where J{\o}rgensen \cite[Section~3]{Jorgensen}, 
for example, uses 
$\isotope{\perp}{}{\mAA}$
and vice versa.
Our definition is the same as J{\o}rgensen's.

In the next lemma, some facts on colocalizations
and localizations are summarized. 
I expect them to be well-known but
I do not know a reference for the lemma in the form that will be needed.
Hence a complete proof will be given.
\pagebreak
\begin{lemma}\label{mylemma}
Let $j^\ast :\mTT\to \mTT''$ be a triangulated functor and $\mTT'= \ker j^\ast$.
\begin{enumerate}
\item[\textnormal{(a)}] Suppose $j^\ast$ admits a colocalization functor, i.e.,
a fully faithful 
left adjoint $j_!$,
\[
\xymatrix@M=.7pc{\mTT \ar@<-.5ex>[r]_-{j^{\ast}}
  &\mTT'' \ar@<-.5ex>[l]_-{j_!}.}
\]
Then the following statements hold.
\begin{enumerate}
\item[\textnormal{(i)}] The 
inclusion $i_{\ast} : \mTT' \to \mTT$ has also a left
adjoint $i^{\ast}$.
\item[\textnormal{(ii)}] For $X$ in $\mTT$ there are natural exact triangles
\[
j_!j^\ast X \stackrel{\varepsilon _X}{\to} X \stackrel{\eta ' _X}{\to} 
i_\ast i^\ast X \to \Sigma j_! j^{\ast} X
\]
where $\varepsilon$ is the counit of $(j_! , j^{\ast})$ and $\eta '$ is the
unit of $(i^{\ast} , i_{\ast})$.
\item[\textnormal{(iii)}] The Verdier quotient $\mTT / \mTT'$ is 
triangulated equivalent to $\mTT''$. In particular, $\mTT / \mTT'$
is an honest category (i.e., the Hom-`sets' form actual sets).
\item[\textnormal{(iv)}] For the subcategory of colocal objects, one has
\[
\essim j_!
=\ker i^\ast =\isotope{\perp}{}{}(\ker j^\ast).
\]
\end{enumerate}
\item[\textnormal{(b)}] Dually, 
suppose $j^\ast$ admits a localization functor, i.e.,
a fully faithful 
right adjoint~$j_\ast$,
\[
\xymatrix@M=.7pc{\mTT \ar@<.5ex>[r]^-{j^{\ast}}
  &\mTT'' \ar@<.5ex>[l]^-{j_{\ast}}.}
\]
Then the following statements hold.
\begin{enumerate}
\item[\textnormal{(i)}] The inclusion $i_{\ast} : \mTT' \to \mTT$ has also a right
adjoint $i^!$.
\item[\textnormal{(ii)}] For $X$ in $\mTT$ there are natural exact triangles
\[
i_{\ast}i^! X \stackrel{\varepsilon ' _X}{\to} X \stackrel{\eta _X}{\to} 
j_{\ast} j^\ast X \to \Sigma i_{\ast}i^! X
\]
where $\varepsilon '$ is the counit of $(i_\ast , i^!)$ and $\eta $ is the
unit of $(j^{\ast} , j_{\ast})$.
\item[\textnormal{(iii)}] The Verdier quotient $\mTT / \mTT'$ is 
triangulated equivalent to $\mTT''$. In particular, $\mTT / \mTT'$
is an honest category (i.e., the Hom-`sets' form actual sets).
\item[\textnormal{(iv)}] 
For the subcategory of local objects, one has
\[
\essim j_\ast =\ker i^!=(\ker j^\ast)^\perp .
\]
\end{enumerate}

\end{enumerate}
\end{lemma}

\begin{proof}
Part (b) follows from (a) by considering opposite categories (then left 
adjoints become right adjoints and vice versa). 

Let us
consider part (a) and first prove the statements (i) and (ii) together.
Since $i_\ast$ is triangulated its left adjoint $i^\ast$ will automatically
be triangulated. 
Let us first define
$i^\ast$ on objects. Statement (ii) tells us what to do. For $X$ in $\mTT$ 
take the counit of the adjunction $(j_!,j^\ast)$ and complete this to an
exact triangle
\begin{equation}\label{trian}
j_! j^\ast X\stackrel{\varepsilon _X}{\to} X\stackrel{\eta'_X}{\to}
 \varphi X\to 
   j_! j^\ast\Sigma X
\end{equation}
in $\mTT$. 
%Note that $\varphi X$ is unique up to isomorphism.
By assumption, $j_!$ is fully faithful, hence the unit $\eta$ of the
adjunction $(j_!,j^\ast)$ is an isomorphism. As for all adjoint pairs,
the diagram 
\[
\xymatrix@M=.7pc{j^\ast X\ar[r]^-{\eta_{j^\ast X}}\ar[dr]_= 
& j^\ast j_! j^\ast X\ar[d]_\iso^{j^\ast(\varepsilon _X)}\\
& j^\ast X  }
\]
commutes \cite[Chapter~IV.1, Theorem~1]{MacLane}, 
so that $j^\ast (\varepsilon _X)$ is an isomorphism. Applying
$j^\ast$ to the triangle (\ref{trian}) shows that 
$\varphi X\in\ker j^\ast$.
Therefore we can define $i^\ast X$ by
$i_\ast i^\ast X =\varphi X$.

Given a map $f:X\to Y$ in $\mTT$, the axioms of a triangulated category 
guarantee the existence of a map $\bar f$
such that we get a map of exact triangles
\begin{equation}\label{map of exact triangles}
\xymatrix@M=.7pc{j_! j^\ast X \ar[r]^{\varepsilon_X} \ar[d]_{j_! j^\ast (f)}
& X \ar[r]^-{\eta'_X} \ar[d]_f & i_\ast i^\ast X \ar[r] \ar@{.>}[d]_{\bar f} 
& j_! j^\ast \Sigma X\ar[d]^{j_! j^\ast \Sigma f}\\
j_! j^\ast Y \ar[r]^{\varepsilon_Y} 
& Y \ar[r]^-{\eta'_Y} & i_\ast i^\ast Y \ar[r]  & j_! j^\ast \Sigma Y.
}
\end{equation}
Using exactly the same arguments as in Remark~\ref{rem}(4) one can
show that there is only one  map $\bar f$ such that the square in the middle
commutes, i.e.,
$\bar f \eta'_X = \eta'_Y f$.
%, or equivalently, if $f$ is zero 
%then so is the induced map $\bar f$.
%Since $\bar f \eta'_X=0$ and the representing functor
%$\mTT (-, i_\ast i^\ast Y)$ is cohomological
%there exists a map $\bar g$ such that
%\[
%\xymatrix@M=.7pc{j_! j^\ast X \ar[r]^{\varepsilon_X} \ar[d]_0
%& X \ar[r]^-{\eta'_X} \ar[d]_0 & i_\ast i^\ast X \ar[r] \ar[d]_{\bar f} 
%& j_! j^\ast \Sigma X\ar[d]^0 \ar@{.>}[dl]_-{\bar g}\\
%j_! j^\ast Y \ar[r]^{\varepsilon_Y} 
%& Y \ar[r]^-{\eta'_Y} & i_\ast i^\ast Y \ar[r]  & j_! j^\ast \Sigma Y
%}
%\]
%commutes. The adjoint map of $\bar g$ with respect to the adjoint pair
%$(j_!, j^\ast)$ is a map into $j^\ast i_\ast i^\ast Y$, which is zero
%since we already checked that $j^\ast i_\ast=0$. Consequently 
%$\bar g$ is itself
%zero. This implies $\bar f=0$. Thus for each $f$ there exists a unique
%$\bar f$ giving a map of exact triangles (\ref{map of exact triangles}).
Consequently, the assignment $f\mapsto \bar f$ is additive
and compatible 
with identities and
composition. Since $i_\ast$ is fully faithful we get a functor 
$i^\ast : \mTT\to\ker j_\ast$. To see that $(i^\ast , i_\ast )$ is an adjunction
it suffices to have a natural transformation (which is then the unit
of the adjunction) $X\to i_\ast i^\ast X$ for 
$X$ in $\mTT$ which is universal from $X$ to the functor
$i_\ast :\ker j_\ast \to \mTT$.
Our candidate is the map $\eta'_X$ defined by the triangle~(\ref{trian}).
It is natural by (\ref{map of exact triangles}). To check that $\eta'_X$ 
is universal
from $X$ to $i_\ast :\ker j_\ast \to \mTT$ 
let $X'\in\mTT$ and a map $X\to i_\ast X'$
be given. The composition
$j_! j^\ast X\to X\to i_\ast X'$ has zero as an adjoint map with respect
to the adjoint pair $(j_!,j^\ast)$ so it is itself zero. This gives us
a commutative diagram (of solid arrows)
\[
\xymatrix@M=.7pc{j_! j^\ast X \ar[r] \ar[d]
& X \ar[r]^-{\eta'_X} \ar[d] & i_\ast i^\ast X \ar[r] \ar@{.>}[d]_{i_\ast (h)} 
& j_! j^\ast \Sigma X\ar[d]\\
0 \ar[r] 
& i_\ast X' \ar[r]^-{=} & i_\ast X' \ar[r]  & 0
}
\]
which can be completed into a map of exact triangles via a map $i_\ast (h)$ 
for some map $h:i^\ast X \to X'$. As above it follows that $h$ is unique.
This shows that $\eta'$ is in fact the unit of an adjoint pair 
$(i^\ast ,i_\ast)$. The exactness of the triangle~(\ref{trian}) ensures
that the statement in (ii) is satisfied.

For part (iii) let $F:\mTT\to\mTT / \mTT'$ be the the canonical functor into
the Verdier  quotient and $\varphi = F j_!$. By the universal 
property of $F$ there exists a functor $\psi$ such that $\psi F = j^\ast$.
\[
\xymatrix@=15mm@M=.7pc{\mTT'\ar@<-.5ex>[r]_{i_\ast}&\mTT \ar@<-.5ex>[r]_{j^\ast}
\ar@<-.5ex>[l]_{i^\ast}\ar[rd]_F & \mTT'' \ar@<-.5ex>[l]_{j_!}
\ar@{.>}@<-.5ex>[d]_\varphi \\  && \mTT / \mTT'\ar@{.>}@<-.5ex>[u]_\psi }
\]
As $j_!$ is fully faithful the unit of the adjoint pair $(j_!,j^\ast)$ 
is an isomorphism and we can conclude that $\psi$ is a left inverse of 
$\varphi$:
\[
\psi\varphi = \psi F j_! = j^\ast j_! \iso \id _{\mTT''}
\]
Let us now apply $F$ to the exact triangle in statement (ii) of 
part (a) of the lemma so that we get an exact triangle
\[
F j_!j^\ast X \to F X \to 
F i_\ast i^\ast X \to F \Sigma j_! j^{\ast} X.
\]
Since $Fi_\ast i^\ast X \iso 0$ we have isomorphisms
\[
F\iso Fj_! j^\ast = \varphi j^\ast = \varphi \psi F
\]
and thus by the universal property of $F$ an isomorphism $\id _{\mTT / \mTT'}\iso
\varphi \psi$. 
This shows that $\varphi$ and $\psi$ are inverse triangulated equivalences.

For part (iv) note that $\ker i^\ast=\essim j_!$ can be proved
in exactly the same way as in Remark~\ref{rem}(8). To see 
$\essim j_! \subset \isotope{\perp}{}{}(\ker j^\ast)$ note that a map
$j_!X\to Y$ with $Y\in\ker j^\ast$ corresponds via the adjunction
$(j_!,j^\ast)$ to a map $X\to j^\ast Y =0$, which has to be the zero map.
Hence the map $j_!X\to Y$ is itself zero.

It now suffices to prove $\isotope{\perp}{}{}(\ker j^\ast)\subset\ker i^\ast$.
For $X\in\isotope{\perp}{}{}(\ker j^\ast)$ the unit $\eta_X:X\to i_\ast i^\ast X$
is zero because $i_\ast i^\ast X$ is in the essential image of $i_\ast$,
which is the same as the kernel of $j^\ast$ (this is again
proved as in Remark~\ref{rem}(8)).
Consider the commutative diagram
\[
\xymatrix@l@M=.7pc{i^\ast X 
& i^\ast i_\ast i^\ast X\ar[l]_-{\varepsilon_{i^\ast X}}\\
& i^\ast X \ar[u]_{i^\ast(\eta _X)}\ar[ul]^= }
\]
involving the unit and counit of the adjunction $(i^\ast, i_\ast)$.
As we have just seen, $i^\ast(\eta_X)$ is zero. Since the right adjoint
$i_\ast$ is fully faithful, the counit $\varepsilon_{i^\ast X}$ is an 
isomorphism. This implies $i^\ast X\iso 0$.
\end{proof}

The next proposition helps us to construct re\-colle\-ments when `the right
part' of a re\-colle\-ment is already given. Together with Remark~\ref{rem}(8)
it implies that, up to equivalence of triangulated categories,
the data of a re\-colle\-ment as in Definition~\ref{recollement} is equivalent
to the data of Proposition~\ref{myprop}.
\begin{prop}\label{myprop}
Let there be given a diagram 
\[
\xymatrix@M=.7pc{\mTT \ar[rr]^{j^{\ast}}
  && \mTT'' \ar@/_7mm/[ll]_{j_!}
  \ar@/^7mm/[ll]^{j_{\ast}}
}
%\[
%\xymatrix@M=.7pc{\mTT'\ar[rr]^{i_{\ast}}
%  &&\mTT \ar@{.>}@/_7mm/[ll]_{i^{\ast}}
%  \ar@{.>}@/^7mm/[ll]^{i^!}\ar[rr]^{j^{\ast}}
%  && \mTT'' \ar@/_7mm/[ll]_{j_!}
%  \ar@/^7mm/[ll]^{j_{\ast}}
%}
\]
of triangulated categories such that
\begin{enumerate}
\item[\textnormal{(i)}] 
$(j_!,j^\ast,j_\ast)$ is an adjoint triple of 
triangulated functors,
\item[\textnormal{(ii)}] at least one of the 
functors $j_!$ and $j_\ast$ is fully faithful,
%\item[\textnormal{(iii)}] $\essim i_\ast = \ker j^\ast$.
\end{enumerate}
and let $i_\ast:\ker j^\ast \to \mTT$ denote the full inclusion.
Then the diagram can be completed into a re\-colle\-ment 
\[
\xymatrix@M=.7pc{ \ker j^\ast\ar[rr]^{i_{\ast}}
  &&\mTT \ar@{.>}@/_7mm/[ll]_{i^{\ast}}
  \ar@{.>}@/^7mm/[ll]^{i^!}\ar[rr]^{j^{\ast}}
  && \mTT'' \ar@/_7mm/[ll]_{j_!}
  \ar@/^7mm/[ll]^{j_{\ast}}
}
\]
by functors $i^\ast$ and
$i^!$ which are unique up to isomorphism.
\end{prop}

\begin{proof}
As left resp.~right adjoints of $i_\ast$, the functors $i^\ast$ and
$i^!$ have to be unique, and we clearly have $j^\ast i_\ast =0$. Let us assume
$j_!$ is fully faithful (in case $j_\ast$ is fully faithful we could consider
opposite categories). By Lemma~\ref{mylemma}(a), parts (i) and (ii), we 
get the upper half of the re\-colle\-ment,
\[
\xymatrix@M=.7pc{ \ker j^\ast\ar[rr]^{i_{\ast}}
  &&\mTT \ar@{.>}@/_7mm/[ll]_{i^{\ast}}
  \ar[rr]^{j^{\ast}}
  && \mTT''.\ar@/_7mm/[ll]_{j_!}
  }
\]

Then, by part (iii) of the same Lemma,
$\mTT''$ is triangulated equivalent to $\mTT / \ker j^\ast$ and hence, by
Remark~\ref{Keller/Neeman}(2), $j_\ast$ is automatically fully faithful.
Now Lemma~\ref{mylemma}(b), parts (i) and (ii), 
applies and gives us also the lower part of the re\-colle\-ment,
\[
\xymatrix@M=.7pc{ \ker j^\ast\ar[rr]^{i_{\ast}}
  &&\mTT 
  \ar@{.>}@/^7mm/[ll]^{i^!}\ar[rr]^{j^{\ast}}
  && \mTT''.\ar@/^7mm/[ll]^{j_{\ast}}
}
\]
\end{proof}

\subsection{An example}
We will now give an example of a re\-colle\-ment arising from
finite localization in stable homotopy theory:

\begin{ex}
Throughout this example we will use the notions of stable homotopy category
and localization as in \cite{HPS}, see also Remark~\ref{Keller/Neeman}(3).
Let $\mCC$ be a stable homotopy category with smash product
$\sm$, internal Hom-functor $\Hom$, and unit $\mS$. 
Recall that a generating set $\mGG$
is part of the data of $\mCC$. Suppose that $\mAA$ is an essentially small
$\mGG$-ideal in $\mCC$, that is, $\mAA$ is a thick subcategory such that
$G\sm A\in \mAA$ whenever $G\in\mGG$ and $A\in\mAA$. Let $\mDD$ denote the
localizing ideal (i.e., localizing subcategory with $C\sm D\in \mDD$ whenever
$C\in\mCC$ and $D\in\mDD$) generated by $\mAA$. If all objects of $\mAA$ are
compact, then there exists a localization functor $L^f_\mAA$ on $\mCC$ whose
acyclics are precisely the objects of $\mDD$ \cite[Theorem~3.3.3]{HPS}.
This functor $L^f_\mAA$ 
is referred to as \emph{finite localization away from}~$\mAA$. 
Theorem~3.3.3 in \cite{HPS} also tells us that finite localization
is always smashing, that is, 
the natural transformation $L^f_\mAA \mS \sm - \to L^f_\mAA$ (which exists
for every localization, cf. \cite[Lemma~3.3.1]{HPS}) 
is an isomorphism. For the complementary
colocalization $C^f_\mAA$ one then has an isomorphism
$C^f_\mAA \iso C^f_\mAA\mS \sm -$.
In particular, $L^f_\mAA$, resp. $C^f_\mAA$, has a right adjoint
$C_\mAA=\Hom (C^f_\mAA \mS,-)$, resp. $L_A=\Hom (C^f_\mAA \mS,-)$.

Now suppose, in addition, all objects of $\mAA$ are strongly dualizable.
This means, the natural map $\HOM(A,\mS)\sm C \to \HOM (A,C)$ is an isomorphism
for all $A\in\mAA$ and all $C\in\mCC$. Roughly speaking,
an object $A$ is strongly dualizable if mapping out of $A$ is the
same as smashing with the (Spanier-Whitehead) dual of $A$.
Under these assumptions, by \cite[Theorem~3.3.5]{HPS},
the right adjoint functors $L_\mAA$ and $C_\mAA$ form also a 
localization-colocalization pair such that
\[
\ker L_\mAA = \essim C_\mAA = \ker C^f_\mAA = \essim L^f_\mAA=
 \mAA ^\perp ,
\]
%where $\mAA^\perp$ is defined as the full subcategory consisting of those
%objects which do not receive non-zero graded maps from $\mAA$,
%\[
%\mAA^\perp = \{ Z\in\mCC\, |\,\mCC (\Sigma^nA,Z)\iso 0 \textnormal{ for each } 
%n\in\mZ \textnormal{ and each } A\in \mAA\}
%\]
(Note that the notation, which
we have adopted from \cite{HPS}, might be misleading:
the acyclics of $L_\mAA$ are not the objects of $\mAA$ 
but those of $\mAA^\perp$.)
We hence get a diagram
\[
\xymatrix@M=.7pc{\ker L_\mAA\ar@{^(->}[rr]
  &&\mCC \ar@/_7mm/[ll]_{L^f_\mAA}
  \ar@/^7mm/[ll]^{C_\mAA}\ar[rr]^{L_\mAA}
  && \essim L_\mAA \ar@/_7mm/[ll]_{C^f_\mAA}
  \ar@/^7mm/@{_(->}[ll]
}
\]
consisting of two adjoint triples because 
a localization functor can be regarded as a left adjoint for the inclusion of
the locals whereas a colocalization can be regarded as a right adjoint
for the inclusion of the colocals. Using Proposition~\ref{myprop} we can
conclude that this diagram is in fact a re\-colle\-ment.
\end{ex}

\section{Recollements of stable model categories}

In this section, we will use some facts on model categories of modules. These
are summarized in Section~A.1 of the Appendix.

\subsection{Reasonable stable model categories}

Every pointed model category $\mCC$ supports a suspension functor 
$\Sigma :\Ho\mCC\to\Ho\mCC$. 
This can, for example, be defined on objects by choosing a cofibrant
replacement $X^{\cofr}$ for $X$ in $\mCC$ and a cone 
of $X^{\cofr}$, that is, a factorization
\[
\xymatrix@M=.7pc{X^{\cofr}\ar[rr]\ar@{>->}[dr]&&\ast\\&C_{X^{\cofr}}
\ar[ru]_-\sim}.
\]
The suspension $\Sigma X$ is then defined as the cofiber of the cofibration
$X^{\cofr}\cof C_{X^{\cofr}}$,
that is, the pushout of the following diagram of solid arrows
\[
\xymatrix@M=.7pc{X^{\cofr}\ar[d]\ar@{>->}[r]&C_{X^{\cofr}}\ar@{.>}[d]\\
\ast\ar@{.>}[r] &\Sigma X .
}
\]
On the level of homotopy categories, this
construction 
becomes a well-defined functor. Also note that $\Sigma X$ is cofibrant.
This is because cofibrations are preserved by pushouts.
% be seen as follows: $X^{\cofr}\cof C_{X^{\cofr}}$ is a cofibrant
%object in the model category of arrows in $\mCC$, carrying the model
%structure of \cite[Theorem~5.1.3]{Hovey}, and taking the cofiber
%is a left Quillen functor from arrows in $\mCC$ to $\mCC$).
The model category is called \emph{stable} if
$\Sigma$ is an equivalence. In this case, 
$\Ho\mCC$ is a triangulated category
with coproducts where the suspension functor is just $\Sigma$. 
Instead of $\Ho\mCC(X,Y)$, we will usually
write $[X,Y]^{\Ho \mCC}$ or simply $[X,Y]$ for the abelian group
of all morphisms from $X$ to $Y$.
A Quillen functor between stable model categories induces a triangulated
functor \cite[Proposition~6.4.1]{Hovey} on the level of homotopy
categories, which is a triangulated equivalence if the Quillen functor is
a Quillen equivalence. 

Recall the following definition, see \cite[Definition~1.2]{Jorgensen}.

\begin{defn}
An object
$X$ of a triangulated category $\mTT$ is \emph{compact} if
\[
\mTT (X,-):\mTT\to\Ab
\]
preserves coproducts and \emph{self-compact}
if the restricted functor $\mTT (X,-)\!\mid _{\langle X\rangle }$ preserves 
coproducts.
\end{defn}

\begin{exs}\label{compactness}
\hspace{0mm}
\begin{enumerate}[(1)]
\item Using a result of Neeman \cite[Lemma~2.2]{Neeman-TheConnection}, 
one can show
that for a ring $R$, the compact objects in $\mD (R)$ are 
the perfect complexes, that is, the chain complexes which are 
quasi-isomorphic to a bounded complex of finitely generated projective
$R$-modules \cite[Theorem~3.8]{glasgow}.
\item Let $F:\mTT\to \mTT'$ be a  
functor between triangulated categories with coproducts.
Suppose $F$ preserves coproducts and is fully faithful.
If $C$ is a compact object in $\mTT$
then $F(C)$ is self-compact in $\mTT'$. To see this note that by Lemma 
\ref{essim}(ii) any family $(X_i)_{i\in I}$ in $\langle F(C)\rangle$ is
up to isomorphism of the form $(F(Y_i))_{i\in I}$ for 
$Y_i \in \langle C\rangle$.
This helps to construct self-compact objects which are not necessarily 
compact. For example $\mZ[\frac{1}{2}]$, the integers with $2$ inverted,
viewed as an object in $\mD(\mZ)$ is self-compact but not compact 
\cite[Example~1.8]{Jorgensen}.
\end{enumerate}
\end{exs}

In the following, we will consider `reasonable' 
stable model categories. 
\begin{defn}\label{reasonable}
By a \emph{reasonable} stable model category we mean a stable 
closed symmetric monoidal
model category $(\mCC, \sm, \mS)$ which satisfies the following
conditions:
\begin{enumerate}[(i)]
\item As a model category, 
$\mCC$ is cofibrantly generated \cite[Definition~2.1.17]{Hovey}.
\item All objects of $\mCC$ are small in the sense of 
\cite{SS00}, that is, every object is $\kappa$-small 
with respect to some cardinal $\kappa$.
\item The monoid axiom holds 
for $(\mCC, \sm, \mS)$ \cite[Definition~3.3]{SS00}.
\item The unit $\mS$ 
is cofibrant in $\mCC$ and a compact generator for $\Ho\mCC$.
\item The smashing condition holds \cite[Section~4]{SS00}, 
that is, for every monoid $R$ 
in $\mCC$ and every cofibrant $R$-module $X$ the
functor $-\sm_R X:\MODD R \to \mCC$ preserves weak equivalences.
\end{enumerate}
%is cofibrantly generated, has only small
%objects (in the sense of \cite{SS00}, that is, every object is $\kappa$-small
%with respect to some cardinal $\kappa$), 
%satisfies the monoid axiom \cite[Definition~3.3]{SS00}, and for which 
%$\mS$ is a cofibrant compact generator (compact in $\Ho\mCC$; 
%furthermore, we require
%that the following smashing condition is satisfied: 
%for every monoid $R$ in $\mCC$ and every cofibrant $R$-module $X$ the
%functor $-\sm_R X:\MODD R \to \mCC$ preserves weak equivalences.
\end{defn}

In particular, all statements from 
Section~A.1 (in the Appendix) hold for the case
of reasonable stable model categories.

\begin{exs}\label{Verleihnix}
We are mainly interested in symmetric
spectra and chain complexes. Both form reasonable
stable model categories:
\begin{enumerate}[(1)]
\item Hovey, Shipley and Smith \cite{HSS} have shown that the category 
$\Sp$ of symmetric spectra of simplicial sets with the stable model structure
has a smash product $\sm$ with unit the sphere
spectrum $\mS$ such that $(\Sp , \sm ,\mS)$ is a closed symmetric monoidal
model category which is cofibrantly generated, has only small objects and
satisfies the monoid axiom and the smashing condition. 
The sphere spectrum $\mS$ is cofibrant and a compact generator.
Hence symmetric spectra form a reasonable stable model category.
Monoids
in $(\Sp , \sm ,\mS)$ are called \emph{symmetric ring spectra}. 
%Smashing with a cofibrant $R$-module over a symmetric ring spectrum $R$ 
%preserves weak equivalences.
\item The category $\Ch (k)$ of unbounded chain complexes over some
commutative ground ring $k$ form a model category
with weak equivalences the quasi-isomorphisms and fibrations
the level-wise surjections \cite[Section~2.3]{Hovey}.
Together with the tensor product and the chain
complex $k[0]$ which is $k$ concentrated in dimension $0$ 
this is
a reasonable stable model category $(\Ch (k),\tensor, k[0])$.
\end{enumerate}
\end{exs}

\subsection{Model categories enriched over a reasonable stable model category}

Let from now on $(\mCC, \sm, \mS)$ be a fixed reasonable stable model
category.
The goal of this section is to prove Theorem~\ref{mainthm},
which gives a necessary and sufficient criterion
for the existence of a recollement with middle term $\mD(R)$, where
$R$ is a given monoid in $\mCC$.
A $\mCC$\emph{-model category} in 
the sense of \cite[Definition~4.2.18]{Hovey} is a model category
$\mMM$ together with a Quillen bifunctor 
$\tensor :\mCC \times \mMM \to\mMM$ which is associative and unital up to
natural and coherent isomorphism (to be precise, the natural 
coherent isomorphisms are part of the data of the
$\mCC$-model category).
In other words, $\mMM$ is enriched, tensored, and cotensored over $\mCC$
such that the tensor functor  $\tensor$ satisfies the pushout product axiom
\cite[Definition~4.2.1]{Hovey}. We will denote the enriched Hom-functor by
$\Hom_\mMM$. Since $\mCC$ is stable, the tensor functor is usually
denoted by $\sm$. But to distinguish it from the monoidal
product $\sm$ in $\mCC$, we will here use $\tensor$.
A $\Sp$-model category is usually called \emph{spectral} model category.
\begin{lemma}\label{Gutemine}
Every $\mCC$-model category is stable.
\end{lemma}

\begin{proof}
Let $\mMM$ be a $\mCC$-model category. Note first that $\mMM$ is pointed
since $\mCC$ is pointed. Namely if $0$ denotes 
the initial and $1$ the terminal 
object of $\mMM$ apply the left adjoint $-\tensor 1 :\mCC\to\mMM$ to
the map $\mS\to\ast$ in $\mCC$ and get a map $1\to 0$ in $\mMM$ which has
to be an isomorphism. 

We define the 1-sphere in $\mCC$ by $S^1=\Sigma\, \mS$ and 
claim that the suspension functor \linebreak
$\Sigma :\Ho\mMM\to \Ho\mMM$ is isomorphic to $S^1\tensor^L -:\Ho\mMM
\to\Ho\mMM$. (This left derived functor exists since $S^1=\Sigma\, \mS$
is cofibrant.)
Consider the diagram
\[
\xymatrix@M=.7pc{\mS\ar[rr]\ar@{>->}[dr]&&\ast \\&C_\mS \ar[ur]_-\sim}
\]
in $\mCC$ and apply $-\tensor X^{\cofr}$, where $X\in\mMM$. 
This is a left Quillen functor, so
it preserves the cofibration $\mS\cof C_\mS$ and the weak equivalence between
the cofibrant objects $C_\mS$ and $\ast$. Hence we get a diagram
\[
\xymatrix@M=.7pc{X^{\cofr}\iso\mS\tensor X^{\cofr}\ar[rr]\ar@{>->}[dr]&&\ast \\&C_\mS\tensor X^{\cofr} \ar[ur]_-\sim}
\]
from which we deduce that $\Sigma X\iso \cofiber \,(\mS\tensor X^{\cofr}\cof
C_\mS\tensor X^{\cofr})$ in $\Ho\mMM$. 
Now $-\tensor X^{\cofr}$ preserves cofibers. Thus
we have natural isomorphisms 
%\begin{eqnarray*}
\[\Sigma X  \iso  \cofiber\, (\mS\cof C_\mS)\tensor X^{\cofr} 
  \iso  \Sigma\,\mS\tensor X^{\cofr}\\
  \iso S^1 \tensor^LX
\]
%\end{eqnarray*}
in $\Ho\mMM$ proving our claim.
Since $\mCC$ is stable we can choose a cofibrant object $S^{-1}$ in $\mCC$
such that $S^1\sm S^{-1}\iso\mS$ in $\Ho\mCC$. Then $S^{-1}\tensor^L -:
\Ho\mMM \to \Ho\mMM$ is a quasi-inverse for $\Sigma\iso S^1\tensor^L -:
\Ho\mMM \to \Ho\mMM$.
\end{proof}

As in Section~A.1 of this paper we consider for a monoid $R$ in a 
reasonable stable 
model category the model structure on $R \MMOD$ 
where the fibrations, resp.~weak equivalences, are exactly the fibrations,
resp.~weak equivalences, of the underlying objects in $\mCC$, 
\cite[Section~4]{SS00}. We denote the homotopy category of $R \MMOD$ 
by $\mD(R)$ and call it the \emph{derived} category of $R$.
Just as in the category $\mCC$ itself, one has the notion of
modules in a $\mCC$-category $\mMM$ over a monoid $R$ in $\mCC$.

\begin{ex}
If $T$ is a monoid in $\mCC$ then the category of right $T$-modules (in $\mCC$)
is a $\mCC$-model category
(and hence stable).
The Quillen bifunctor is given by the three functors in (\ref{bifunctors})
with $R=S=\mS$ (the first functor is tensor, the second cotensor, and
the third enrichment). Replacing $T$ by $T^{\op}$ shows that left $T$-modules
also form a $\mCC$-model category. A left module in the  
$\mCC$-model category $\MODD T$ over another monoid $R$ is the same
as an $R$-$T$-bimodule.
\end{ex}

\begin{lemma}\label{cpgen}
Let $R$ be a monoid in a reasonable stable model category $(\mCC, \sm, \mS)$.
Then $R$ is a compact generator for $\mD(R)$ and for $\mD(R^{\op})$.
\end{lemma}

\begin{proof}
It suffices to consider the case of $\mD(R)$
since $R$ and $R^{\op}$ are the same as modules.
As $\mS$ is cofibrant in $\mCC$ we have an isomorphism $R\sm^L\mS\iso R$
in $\mD(R)$. The Quillen pair induced by extension 
and restriction of scalars gives us
then  an isomorphism 
\begin{equation}\label{trick}
[R,X]^{\mD(R)}\iso [\mS , X]^{\Ho \mCC}\, ,
\end{equation}
which is natural in $X\in\mD(R)$.
Using this we get
\[
\bigoplus_{i\in I} \left[R,X_i\right]^{\mD(R)}
\iso \bigoplus_{i \in I} [\mS,X_i]^{\Ho \mCC}
\iso \Bigl[\mS, \coprod_{i \in I}X_i\Bigr]^{\Ho \mCC}
\iso \Bigl[R,\coprod_{i\in I}X_i\Bigr]^{\mD(R)}
\]
for any family $(X_i)_{i \in I}$ of objects in $\mD(R)$, which shows the 
compactness of $R$.
 
For compact objects, one has the following characterization of being a 
generator 
(see \cite[Lemma~2.2.1]{SS03}). A compact object 
$P$ is a generator for a triangulated
category $\mTT$ with coproducts 
if and only if $P$ detects if objects are trivial, 
that is, $X\iso 0$ in $\mTT$ if and only
if $\mTT (P, \Sigma^n X) = 0$ for all $n\in \mZ$.
Let $[R, \Sigma ^n X]^{\mD(R)} = 0$ for all $n\in \mZ$. 
Using again the isomorphism (\ref{trick}),
\[
0=[R,\Sigma ^n X]^{\mD(R)} \iso [\mS,\Sigma ^n X]^{\Ho \mCC}
\]
which implies $X\iso 0$ because $\mS$ is a generator for $\mCC$. Thus
$R$ is a generator. 

\end{proof}

%If we regard $R$ as a module over itself, it is a generator for $\mD(R)$,
%just as in the case of chain complexes. This is a consequence of 
%\cite[Lemma 2.2.1 and Lemma 3.5.2]{SS03} and the fact that stable homotopy
%equivalences of spectra are stable equivalences \cite[Theorem 3.1.11]{HSS}.
%

%Due to \cite[Theorem 3.3.3]{SS03} every stable
%model category having a compact generator and 
%satisfying some extra conditions is, up to a chain of 
%Quillen equivalences, the model category of modules over a symmetric ring
%spectrum.
%So we restrict our attention to the case of 
%modules over symmetric ring spectra.

\subsection{A Quillen pair}

Let $\mMM$ be a $\mCC$-model category and $B$ a cofibrant and fibrant
object in $\mMM$. Then $E:=\Hom _\mMM (B,B)$ is a monoid in $\mCC$
and there
is an action $E\tensor B\to B$ of $E$ on $B$ given by the adjoint map
of the identity $E\stackrel{=}{\to} \Hom _\mMM(B,B)$ giving $B$ a left
$E$-module structure.

%We call an 
%object $X$ of a triangulated category $\mTT$ \emph{compact} if
%$\mTT (X,-)$ preserves coproducts and \emph{self-compact}
%if the restricted functor $\mTT (X,-)\!\mid _{\langle X\rangle }$ preserves 
%coproducts.

\begin{thm}\label{thm1}
Suppose that $\mMM$ is a 
$\mCC$-model category and $B$ a cofibrant and fibrant 
object in $\mMM$.
\begin{enumerate}
\item[\textnormal{(i)}] There is a Quillen pair 
\[
\xymatrix@M=.7pc{\MODD E\ar@<.5ex>[rr]^-{-\tensor_EB}
  &&\mMM.\ar@<.5ex>[ll]^-{\Hom_\mMM(B,-)}}
\]
\item[\textnormal{(ii)}] If $B$ is self-compact in $\Ho\mMM$
the restriction $i^!\!\mid _{\langle B\rangle}$
of the triangulated functor 
\[i^! =\RHom _\mMM (B,-) : \Ho\mMM\to\mD(E^{\op})\]
preserves coproducts. 
\item[\textnormal{(iii)}] If $B$ is self-compact in $\Ho\mMM$ 
the triangulated functor 
\[
i_\ast =-\tensor^L_EB :\mD(E^{\op})\to\Ho\mMM
\]
is fully faithful and has essential image
\[\essim i_\ast = \langle B\rangle.\]
\end{enumerate}
\end{thm}

\begin{proof}
This theorem is a variant of \cite[Theorem~3.9.3]{SS03}, in which spectral
categories, i.e., $\Sp$-categories are considered, but in the proof 
only those properties of $\Sp$ are required which every reasonable 
stable model category possesses. Moreover, self-compact objects have not
been considered in \cite{SS03}. That is why we have to modify the proof,
especially for part~(iii).

Part~(i) is
simply a `one object version' of \cite[Theorem~3.9.3(i)]{SS03}. 
For $A\in\mMM$, the object $\Hom_\mMM(B,A)$ of $\mCC$ has a canonical
right action of $E=\Hom_\mMM(B,B)$ so that the functor $\Hom_\mMM(-,B)$
takes values in $\MODD E$.
If
$X$ is a right $E$-module then $X\tensor_EB$ is defined as
the coequalizer in $\mMM$ of
\[
\xymatrix@M=.7pc{(X\sm E)\tensor B \ar@<.5ex>[r]\ar@<-.5ex>[r] & X\tensor B,}
\]
where one map is induced by the right action of $E$ on $X$ and the other
by the associativity isomorphism $(X\sm E)\tensor B\iso X\tensor (E\tensor B)$
and the left action of $E$ on $B$.

For part~(ii) one has to check that for any family $(A_j)_{j\in J}$ of objects 
in $\langle B\rangle$ the canonical map
$\coprod_j\RHom _\mMM (B, A_j)\to\RHom _\mMM (B,\coprod_j A_j)$
is an isomorphism, or equivalently,
the induced map
\begin{equation}\label{Asterix}
\Bigl[X,\coprod _{j\in J}\RHom _\mMM (B, A_{j})\Bigr]^{\mD(\mEE^{\op})}\to
\Bigl[X,\RHom _\mMM (B,\coprod _{j\in J} A_{j})\Bigr]^{\mD(\mEE^{\op})}
\end{equation}
is a natural isomorphism for every 
$X$ in $\mD(E^{\op})$. But those $X$ for which
the map~(\ref{Asterix}) is an isomorphism for all~(!) 
families $(A_j)_{j\in J}$ of objects 
in $\langle B\rangle$ form a localizing 
triangulated subcategory of $\mD(E^{\op})$.
The right $E$-module $E$ is contained in this subcategory -- to see this, 
use 
the compactness of $E$ in $\mD(E^{\op})$ (Lemma~\ref{cpgen}), 
the derived adjunction
of the Quillen pair from part (i), and the 
self-compactness of $B$ in $\Ho\mMM$
% and the fact 
%$E$ is compact in $\mD(E^{\op})$ by Lemma~\ref{cpgen}. 
Since $E$ is  a generator for $\mD(E^{\op})$ 
it now follows that the map (\ref{Asterix}) is always
an isomorphism.

For part~(iii) the proof of \cite[Theorem~3.9.3(ii)]{SS03} must be rearranged.
The point is that our $B$ is only self-compact, not necessarily compact.
We will give the details of the proof, the order is as follows.
\begin{enumerate}[(a)]
\item $\essim i_\ast \subset \langle B\rangle$
\item $i_\ast$ is fully faithful.
\item $\essim i_\ast \supset \langle B\rangle$
\end{enumerate}
Note that $i_\ast$, as a left adjoint, preserves coproducts.
Part (a) follows from Lemma~\ref{essim}(i) since $E$ is a generator
for $\mD(E^{\op})$ and $E\tensor^L_EB\iso B$ in $\Ho\mMM$. Part (a) 
implies that $i_\ast$
maps coproducts in $\mD(E^{\op})$ to coproducts in $\langle B\rangle$.
By part (ii) of this theorem, 
$i^!\!\mid _{\langle B\rangle}$ preserves coproducts as well
and so does the composition $i^! i_\ast$. A left adjoint is fully faithful
if and only if the unit of the adjunction is an isomorphism. Hence consider
for $X$ in $\MODD E$ the unit
\begin{equation}\label{unit}
X\to i^! i_\ast X,
\end{equation}
which is a natural transformation between coproduct preserving triangulated 
functors. By Lemma~\ref{useful} the unit will be an isomorphism if it is so
for $X=E$. Since $B$ is fibrant by assumption the unit is in this case the
isomorphism
\[E=\Hom _\mMM (B,B)\iso \RHom _\mMM (B,B)\iso\RHom _\mMM (B, E\tensor^L_E B)
=i^! i_\ast E.
\]
This shows (b). Now part (c) follows 
from Lemma~\ref{essim} (ii) because $i_\ast$ is full
by part (b). 
\end{proof}

Let $R$ be a monoid in $\mCC$ and $C$ a cofibrant left $R$-module
which is compact in $\mD(R)$. Set $F=\Hom _R(C,C)$ so that $C$ is
a left $(F\sm R)$-module. By Lemma~\ref{lemma2} we get a 
Quillen pair
\[
\xymatrix@M=.7pc{R \MMOD\ar@<-.5ex>[rr]_-{\Hom_R(C,-)}
  &&\MODD F ,\ar@<-.5ex>[ll]_-{-\sm_F C}}
\]
whose derived adjunction we will denote by
\[
\xymatrix@M=.7pc{\mD(R) \ar@<-.5ex>[rr]_-{j^\ast}
  &&\mD(F^{\op}). \ar@<-.5ex>[ll]_-{j_!}}
\]
Our goal is now to show that the right adjoint $j^\ast$ has itself
a right adjoint $j_\ast$. The idea is to imitate \cite[Setup 2.1]{Jorgensen}
where a `dual' module $C^\ast$ of $C$ is defined such that $j^\ast$,
the right derived of mapping out of $C$, is the same as the left derived
of tensoring with the dual $C^\ast$. To carry this out we must 
make the technical
assumption that $R$ is cofibrant in~$\mCC$. 
Note that $C$ being a cofibrant $R$-module also implies that
\[
\xymatrix@M=.7pc{R\MMODD R\ar@<-.5ex>[rr]_-{\Hom_R(C,-)}
  &&\MODD (F\sm R) ,\ar@<-.5ex>[ll]_-{-\sm_F C}}
\]
is a Quillen pair (Lemma~\ref{lemma2}). Hence we can define
\[
C^\ast =\Hom _R (C, R^{\fibr} )^{\cofr}\quad \textnormal{in}\quad \MODD 
(F\sm R).
\]
Here $R^{\fibr}$ denotes the fibrant replacement of $R$ in $R \MMODD R$
which comes from the functorial factorization, in particular,
the weak equivalence $R\weq R^{\fibr}$ is also a cofibration (we will need 
this in Proposition~\ref{prop1}). 
Similarly, $\Hom _R (C, R^{\fibr} )^{\cofr}$ is the cofibrant replacement 
of $\Hom _R (C, R^{\fibr} )$ in 
$\MODD (F\sm R)$.

\begin{prop}\label{prop1}
Suppose that $R$ is cofibrant and a monoid in $\mCC$, 
$C$ is a cofibrant and fibrant left $R$-module
which is compact in $\mD(R)$, and let $F$ and $C^\ast$ be defined
as above. Then there are triangulated functors
\[
\xymatrix@M=.7pc{
  \mD(R)   \ar[rr]^{j^{\ast}}
  && \mD(F^{\op}) \ar@/_7mm/[ll]_{j_!}
  \ar@/^7mm/[ll]^{j_{\ast}}
}
\]
given by $j_! =-\sm ^L_F C$, $j^\ast =\RHom _R(C,-)$, $j_\ast 
=\RHom_F(C^\ast ,-)$ such that
\begin{enumerate}
\item[\textnormal{(i)}] $j^\ast \iso C^\ast \sm ^L_R -$,
\item[\textnormal{(ii)}] $(j_!,j^\ast,j_\ast)$ is an adjoint triple,
\item[\textnormal{(iii)}] $j_!$ and $j_\ast$ are fully faithful.
\end{enumerate}
\end{prop}

\begin{proof}
Note first that 
we have a Quillen pair
\[
\xymatrix@M=.7pc{R \MMOD \ar@<.5ex>[rr]^-{C^\ast \sm_R -}
&&\MODD F \ar@<.5ex>[ll]^-{\Hom_F(C^\ast ,-)}}
\]
since $C^\ast$ is cofibrant as a right $(F\sm R)$-module.

Ad (i). 
The compactness of $C$ implies that $j^\ast=\RHom_R(C,-)$ preserves
coproducts. (This can be proved in the same manner as Theorem~\ref{thm1}(ii).)
Hence it 
suffices to give a natural transformation between the 
functors $j^\ast$ and $C^\ast\sm^L_R-$
which is an isomorphism in $\mD(F^{\op})$ for the generator $R$
(Lemma~\ref{useful}). For $X$ in $R \MMOD$ a map
\begin{equation}\label{map1}
C^\ast \sm _R^L X \to \RHom _R(C,X)
\end{equation}
in $\mD(F^{\op})$ corresponds via adjunction to a map
\begin{equation}\label{map2}
(C^\ast\sm _R^L X)\sm _F^LC\to X
\end{equation}
in $\mD(R)$. For $X^{\cofr}$ the cofibrant replacement
of $X$ in $R\MMOD$ we have isomorphisms
\[
(C^\ast\sm _R^L X)\sm _F^LC\iso (C^\ast\sm _R X^{\cofr})\sm _FC\iso
(C^\ast\sm _FC)\sm _R X^{\cofr}
\]
 in $\mD(R)$. Using the bimodule map $C^\ast=\Hom _R (C, R^{\fibr} )^{\cofr}
\to \Hom _R (C,R^{\fibr})$ (which is a weak equivalence,
but we do not need that here) and the evaluation 
\[
\Hom _R (C,R^{\fibr})\sm_F 
C\to R^{\fibr}
\]
we get a map
\[
(C^\ast\sm _FC)\sm _R X^{\cofr}\to R^{\fibr}\sm _R X^{\cofr}.
\]
The functor $-\sm _R X^{\cofr}: R \MMODD R \to R \MMOD $ 
is left Quillen
because $R$ being cofibrant in $\mCC$ implies that $X^{\cofr}$ is cofibrant
in $\mCC$ by Corollary
\ref{cor2}. Hence smashing with $X^{\cofr}$ preserves the trivial cofibration
$R\to R^{\fibr}$ and we get
\[
R^{\fibr}\sm _R X^{\cofr}\iso R\sm _R X^{\cofr} \iso X^{\cofr}\iso X
\]
in $\mD(R)$.  
%comes from the fact
%that smashing with the cofibrant $R$-module $X^{\cofr}$ preserves
%weak equivalences. 
%Note that we do not really need this latter fact. If one takes the
%fibrant replacement coming from the functorial factorization in
%$R\MMODD R$ we have a trivial cofibration $R\stackrel{\sim}{\cof} 
%R^{\fibr}$ in $R\MMODD R$.
Altogether this defines a natural map (\ref{map2}) and via
adjunction the desired map (\ref{map1}). Setting $X=R$ gives us the isomorphism
$C^\ast\sm _R^L R\iso C^\ast\iso \RHom _R(C,R)$ in $\mD(F^{\op})$.
% see own notes 10.10.05 !!!

Ad (ii). 
We know already that $(j_!,j^\ast)$ and $(C^\ast\sm^L_R-,j_\ast)$ are 
adjoint pairs. Hence 
it follows from part
(i) of this proposition that $(j_!,j^\ast,j_\ast)$ is an adjoint triple.
 
%Since $R$ was assumed to be cofibrant Corollary 
%\ref{cor2} implies that $C^\ast$ is also cofibrant as a right $F$-module,
%so we have a Quillen pair
%\[
%\xymatrix@M=.7pc{R \MMOD \ar@<.5ex>[rr]^-{C^\ast \sm_R -}
%&&\MODD F .\ar@<.5ex>[ll]^-{\Hom_F(C^\ast ,-)}}
%\]
%Using (i) we can conclude that $(j^\ast ,j_\ast)$ is also an adjoint pair.

Ad (iii). In the homotopy category, 
$C$ is compact and in particular self-compact. Furthermore,
$C$ is cofibrant and fibrant by assumption. Hence $j_!$ is
fully faithful by Theorem~\ref{thm1}. The fact that the right 
adjoint $j_\ast$ of $j^\ast$ is also fully faithful can be  
deduced formally from Proposition~\ref{myprop}. Alternatively,
we check that
the counit of the adjoint pair $(j^\ast , j_\ast)$ is an isomorphism.
Let $X$ be  a right $F$-module. We show that $X$
and $j^\ast j_\ast X =\RHom _R(C, \RHom_F(C^\ast , X))$ 
represent the same contravariant functor on
$\mD(F^{\op})$. For $Z\in\MODD F$ we have natural isomorphisms
\begin{eqnarray*}
\left[Z,\RHom _R\left(C, \RHom_F(C^\ast , X)\right)\right]^{\Ho F^{\op}} 
&\iso 
& \left[Z\sm _F^LC, \RHom_F(C^\ast ,X)\right]^{\mD(R)} \\ 
& \iso & \left[(C^\ast \sm _R^L(Z\sm _F^L C),X\right]^{\mD(F^{\op})} \\ 
& \iso & \left[\RHom _R(C, Z\sm _F^L C),X\right]^{\mD(F^{\op})} \\ 
& \iso & \left[Z,X\right]^{\mD(F^{\op})}.
\end{eqnarray*}
The first and second are adjunction isomorphisms, the third uses part (i) of
this proposition
and the last is induced by the unit $Z\to\RHom _R(C, Z\sm _F^L C)$
of the adjoint pair $(j_!, j^\ast)$, which is an isomorphism because $j_!$
is faithful.
This gives an isomorphism $j^\ast j_\ast X\stackrel{\iso}{\to} X$ 
which is indeed the counit.
\end{proof}

\begin{lemma}\label{lemma4}
Suppose we are in the situation of Proposition~\ref{prop1}. Then
\[
C^\perp = \ker j^\ast .
\]
\end{lemma}

\begin{proof}
Let $X\in C^\perp$. We want to show $X\in\ker j^\ast$, that is,
$\RHom _R(C,X)\iso 0$ in $\mD(F^{\op})$. 
Let us again use the characterization for generators of triangulated categories
with coproducts given in \cite[Lemma~2.2.1]{SS03}. 
Since $F$ is a generator for $\mD(F^{\op})$ it suffices to show
that $\RHom _R(C,X) \in F^\perp$. 
One has isomorphisms

\begin{eqnarray*}
\bigl[\Sigma ^n F, \RHom_ R (C,X)\bigr]^{\mD(F^{\op})}
& \iso & 
[\Sigma ^nF\sm ^L_FC,X]^{\mD(R)}\\
& \iso & \bigl[\Sigma^n(F\sm ^L_FC),X\bigr]^{\mD(R)}\\
& \iso & [\Sigma^nC,X]^{\mD(R)}\\
& \iso & 0,
\end{eqnarray*}
where the last one is because of our assumption, $X\in C^\perp$.
This shows $C^\perp \subset \ker j^\ast$.

If, on the other hand, $X\in\ker j^\ast$ we have
\begin{eqnarray*}
[\Sigma^nC,X]^{\mD(R)} 
& \iso & [C, \Sigma^{-n}X]^{\mD(R)}\\
& \iso & [F\sm ^L_FC, \Sigma^{-n}X]^{\mD(R)}\\
& \iso & \bigl[F,\RHom _R(C,\Sigma^{-n}X)\bigr]^{\mD(F^{\op})}\\
& \iso & \bigl[F,\Sigma^{-n}\RHom _R(C,X)\bigr]^{\mD(F^{\op})}\\
& \iso & 0,
\end{eqnarray*}
where the last isomorphism is induced 
by $\RHom _R(C,X)\iso 0$ in $\mD(F^{\op})$. Hence $X\in C^\perp$ and the
proof is complete.

\end{proof}

\subsection{The main theorem}
We are now able to prove the main theorem which gives a necessary
and sufficient criterion for the existence of a certain recollement.

\begin{thm}\label{mainthm}
Let $R$ be a monoid in the reasonable stable model category $\mCC$ 
and let $B$ and $C$ 
be left $R$-modules. Then the following are equivalent.
\begin{enumerate}
\item[\textnormal{(i)}] There is a recollement
\[
\xymatrix@M=.7pc{\mD(S)\ar[rr]^{i_{\ast}}
  &&\mD(R) \ar@/_7mm/[ll]_{i^{\ast}}
  \ar@/^7mm/[ll]^{i^!}\ar[rr]^{j^{\ast}}
  && \mD(T) \ar@/_7mm/[ll]_{j_!}
  \ar@/^7mm/[ll]^{j_{\ast}}
}
\]
where $S$ and $T$ are monoids in $\mCC$ such that $i_\ast (S)\iso B$
and $j_! (T)\iso C$.
\item[\textnormal{(ii)}] In the derived 
category $\mD(R)$, the module $B$ is self-compact,
$C$ is compact, 
$B^\perp\nolinebreak\cap\nolinebreak C^\perp =0$, and $B\in C^\perp$.
\end{enumerate}
\end{thm}

\begin{proof}
(i) $\Rightarrow$ (ii)
For this implication, 
the proof of \cite[Theorem~3.4, (i)~$\Rightarrow$~(ii)]{Jorgensen} can
be translated literally. The details are as follows.
By definition, the triangulated functor $i_\ast$ is fully faithful and,
as a left adjoint, preserves coproducts. Since $S$ is compact
in $\mD(S)$ by Lemma~\ref{cpgen}, we can conclude that $B\iso i_\ast (S)$ 
is self-compact in $\mD(R)$ (cf.~Example~\ref{compactness}(2)).
%implies
%$\langle B\rangle \subset \essim i_\ast$.  Recall that the restriction
%of $i_\ast i^\ast$ to the essential image of $i_\ast$ is natural
%isomorphic to the identity (Remark~\ref{rem} (3)). We thus get
%isomorphisms
%\[
%\mD(R)(B,-)\mid_{\langle B\rangle}\iso  
%\mD(R)(i_\ast(S),i_\ast i^\ast (-))\iso
%\mD(S)(S, i^\ast(-))\mid_{\langle B\rangle}.
%\]
%The latter functor preserves coproducts because $i^\ast$ does (as a left
%adjoint) and because
%$S$ is compact in $\mD(S)$. 
%PROOF?!
%This proves the self-compactness of $B$.
To see the compactness of $C\iso j_!(T)$ use the adjunction isomorphism
\[
[j_!(T),-]^{\mD(R)}\iso [T,j^\ast (-)]^{\mD(T)},
\]
the fact that $j^\ast$ preserves coproducts (as a left adjoint), and
the compactness of $T$ in $\mD(T)$.

For $X\in B^\perp \cap C^\perp$ we have
%\begin{eqnarray*}
\[
0  \iso  [\Sigma ^nB ,X]^{\mD(R)} \iso [\Sigma ^ni_\ast (S) ,X]^{\mD(R)}
\iso [\Sigma ^nS,i^!X]^{\mD(S)}
\]
%\end{eqnarray*}
and
%\begin{eqnarray*}
\[
0 \iso [\Sigma ^nC ,X]^{\mD(R)} \iso [\Sigma ^nj_! (T) ,X]^{\mD(R)}
\iso [\Sigma ^nT,j^\ast X]^{\mD(T)}
\]
%\end{eqnarray*}
for each $n\in\mZ$, and this implies by \cite[Lemma~2.2.1]{SS03}
$i^! X=0$ and $j^\ast X=0$.
Using the exact triangle in Definition~\ref{recollement}(iv)(b)
we can conclude that $X=0$ and thus we get $B^\perp \cap C^\perp =0$.

It remains to show that $B\in C^\perp$. For each $n\in\mZ$ we have
%\begin{eqnarray*}
\[
[\Sigma ^n C,B]^{\mD(R)} \iso [\Sigma ^n j_!(T),i_\ast (S)]^{\mD(R)}
\iso [\Sigma ^n T,j^\ast i_\ast (S)]^{\mD(T)} = 0\, ,
\]
%\end{eqnarray*}
where the last equation holds because of $j^\ast i_\ast =0$.

(ii) $\Rightarrow$ (i)
We want to use Proposition~\ref{prop1} and therefore need a monoid in $\mCC$
which is cofibrant in $\mCC$.
We can 
take a cofibrant replacement $R^{\cofr}$ of our monoid
$R$ in the model category of monoids in $\mCC$ using the
model structure of \cite[Theorem~4.1(3)]{SS00}. The same theorem 
tells us that $R^{\cofr}$ is cofibrant in $\mCC$. Schwede and Shipley have 
shown that whenever the smashing condition is satisfied
(see Definition~\ref{reasonable}; we need the smashing condition only for
this application)
a weak equivalence between monoids (in our case the weak equivalence 
$R^{\cofr}\stackrel{\sim}{\to}R$ ) induces a Quillen equivalence
between the module categories and in particular a triangulated equivalence
between the homotopy categories \cite[Theorem~4.3]{SS00}. 
The properties `compactness' and `self-compactness' as well as the 
`generating condition' in part (ii) of our theorem are of course preserved
by triangulated equivalences.
Hence we
can without loss of generality assume that $R$ is cofibrant in $\mCC$
(Remark~\ref{rem}(5)).
We can also assume that $C$ is cofibrant and fibrant in $R \MMOD$ (otherwise
we take a cofibrant and fibrant replacement). 

Now the proof of \cite[Theorem~3.4, (ii) $\Rightarrow$ (i)]{Jorgensen}
can be imitated.
Let $E=\Hom _R(B,B)$ and $F=\Hom _R(C,C)$. 
We apply Proposition~\ref{prop1} and get the right part 
\begin{equation}\label{rightpart}
\xymatrix@M=.7pc{
  \mD(R)   \ar[rr]^{j^{\ast}}
  && \mD(F^{\op}) \ar@/_7mm/[ll]_{j_!}
  \ar@/^7mm/[ll]^{j_{\ast}}
}
\end{equation}
of
a recollement. Let $\iota$ denote the inclusion of $C^\perp$ in
$\mD(R)$. By Lemma~\ref{lemma4}, $C^\perp$ is just
the kernel of $j^\ast$. Hence we can 
use Proposition~\ref{myprop} to complete
the diagram (\ref{rightpart}) into a recollement
\begin{equation}\label{Obelix}
\xymatrix@M=.7pc{C^\perp \ar[rr]^{\iota}
  &&\mD(R) \ar@{.>}@/_7mm/[ll]
  \ar@{.>}@/^7mm/[ll]\ar[rr]^{j^{\ast}}
  && \mD(F^{\op}) \ar@/_7mm/[ll]_{j_!}
  \ar@/^7mm/[ll]^{j_{\ast}}
}.
\end{equation}
Our next claim is
\begin{equation}\label{<B>=Csenkrecht}
C^\perp = \langle B\rangle .
\end{equation}
Since $B\in C^\perp$ by assumption, $\langle B\rangle$ is contained
in the localizing subcategory $C^\perp$.
Now let $X$ be in $C^\perp$. We apply Theorem~\ref{thm1} 
to the case $\mMM = R\MMOD$ and consider the counit $\varepsilon_X:
i_\ast i^! X\to X$ of the adjunction 
\[
\xymatrix@M=.7pc{\mD(E^{\op})\ar@<.5ex>[rr]^-{i_\ast}
  &&\mD(R)\ar@<.5ex>[ll]^-{i^!}}
.
\]
 As for any adjunction, we have a commutative diagram
\[
\xymatrix@M=0.7pc{i^!X \ar[r]^-{\eta_{i^!X}}_-\iso \ar[dr]_=
& i^! i_\ast i^! X \ar[d]^-{i^!(\varepsilon_X)} 
\\ & i^!X\,.}
\]
Here the unit $\eta_{i^!X}$ is an isomorphism since the left
adjoint $i_\ast$ is fully faithful by Theorem~\ref{thm1}(iii). 
Hence $i^!(\varepsilon_X)$ is also
an isomorphism. If we extend $\varepsilon_X$ to an exact triangle
\[
i_\ast i^! X\stackrel {\varepsilon_X}{\to} X\to Y\to \Sigma i_\ast i^! X
\]
and apply the triangulated functor $i^!$ we can conclude
that $\RHom _R(B,Y)=i^! Y=0$. This implies $Y\in B^\perp$ because
\[
[\Sigma^nB,Y]^{\mD(R)}\iso[E\tensor^L_EB,\Sigma^{-n}Y]^{\mD(R)}
\iso[E,\Sigma^{-n}\RHom_R(B,Y)]^{\mD(E^{\op})}=0.
\]
Moreover, 
the first term
$i_\ast i^! X$ of the exact triangle is in the essential
image of $i_\ast$, which is by Theorem~\ref{thm1}(iii) equal to 
$\langle B\rangle$. But we have
already shown that $\langle B\rangle\subset C^\perp$, so $i_\ast i^! X$
is in $C^\perp$, as the middle term $X$ of the exact triangle is by
assumption. It follows that $Y$ is also in $C^\perp$ and hence in
$C^\perp\cap B^\perp=0$. But $Y=0$ implies that $\varepsilon_X:
i_\ast i^! X\to X$ is an isomorphism. Thus $X$ is in $\essim i_\ast$, which
is the same as $\langle B \rangle$ by Theorem~\ref{thm1}(iii).
This proves equation (\ref{<B>=Csenkrecht}).

Since $i_\ast :\mD(E^{\op})\to \mD(R)$ is fully faithful it restricts 
to an equivalence of triangulated categories
\[
i_\ast :\mD(E^{\op}) \stackrel{\simeq}{\to} \essim i_\ast = C^\perp .
\]
Composing this equivalence with the recollement (\ref{Obelix}) yields a
recollement
\[
\xymatrix@M=.7pc{\mD(E^{\op})\ar[rr]^{i_{\ast}}
  &&\mD(R) \ar@/_7mm/[ll]_{i^{\ast}}
  \ar@/^7mm/[ll]^{i^!}\ar[rr]^{j^{\ast}}
  && \mD(F^{\op}) \ar@/_7mm/[ll]_{j_!}
  \ar@/^7mm/[ll]^{j_{\ast}}
}
\]
with $i_\ast (E)=E\tensor^L_EB\iso B$ and $j^!(F)=F\tensor_F^LC\iso C$.
This completes the proof.
\end{proof}

\begin{ex}
Let $\mCC$ be the model category of chain complexes over $\mZ$, 
$R=\mZ$, $B=\mZ[\frac{1}{2}]$ (i.e., the integers with $2$ 
inverted, cf.~Examples~\ref{compactness}(2)), and $C=\mZ/2$. 
Then it is verified
in \cite[Example~3.5]{Jorgensen} that
$B$ is self-compact, $C$ is compact, $B^\perp\cap C^\perp =0$, and 
$B\in C^\perp$, so that Theorem~\ref{mainthm} applies and yields a recollement
for $\mD(\mZ)$.
\end{ex}

\begin{rem}
Let us consider two special cases.
\begin{enumerate}[(1)]
\item If $\mCC$ is the model category of 
chain complexes (over some commutative ground
ring) Theorem~\ref{mainthm} is just the same as 
J{\o}rgensen's \cite[Theorem~3.4]{Jorgensen}.
\item Assume $B=0$ in the theorem. Then statement (i) reads as follows:
There is a triangulated equivalence
\[
\xymatrix@M=.7pc{\mD(R)\ar@<-.5ex>[r]_-{j^\ast}&\mD(T)\ar@<-.5ex>[l]_-{j_!}}
\]
where $T$ is a monoid in $\mCC$ such that $j_!(T)\iso C$.
Namely, $i_\ast(S)\iso 0$ means we are in the degenerate case where $S$ is 
zero, which implies that $\mD(S)$ is zero. Then, by Remark~\ref{rem}(6),
the recollement in (i) is the same as a triangulated equivalence $j^\ast$ with 
inverse $j_!$.

Consider the second statement of Theorem~\ref{mainthm}.
 Of course, the trivial module is self-compact and
$0^\perp$ is the whole category $\mD(R)$. Since a compact $C$ in $\mD(R)$
is a generator
if and only if $C^\perp =0$, statement (ii) reads as: $C$ is a
compact generator for the derived category $\mD(R)$.
\end{enumerate}
\end{rem}

\newpage
\part{Topological well generated categories}
We will characterize the topological well generated triangulated
categories.
Our main result, Theorem~\ref{characterization}, states that 
a topological triangulated category is well generated
if and only if it is a localization of the derived category of a spectral
category (alias ring 
spectrum with several objects)
such that the acyclics are generated
by a set. 
This is a topological version of Porta's characterization of 
the algebraic well generated
categories \cite[Theorem~5.2]{Porta}, 
as appropriate localizations of
the derived categories of DG categories, that is,
DG algebras with several objects.

In Section~3 we will recall some definitions from triangulated
category theory (such as
$\alpha$-small, $\alpha$-perfect, $\alpha$-compact, well generated),
which can all be found in \cite{Neeman}, although sometimes stated in a
different way. We make frequent use of the results in Neeman's book and 
we have decided to use Neeman's original definition of
well generated categories although it is not as easily stated as Krause's
characterization \cite{Krause}. 
%It would be interesting to think about
%a proof of the classification theorem 
%using Krause's characterization (which, as far as 
%I know, Porta takes as a definition when he classifies
%the algebraic well generated categories).
We will write down a proof for the fact that the class
of well generated triangulated categories is stable under forming
subcategories that are generated by a set of objects 
and localizations whose acyclics
are generated by a set of objects, cf.~Proposition~\ref{Neeman's}.

Section~4 starts with giving the definitions of (symmetric) ring
spectra with several objects (`spectral categories') and module spectra
over such. The derived category of a spectral category is the
homotopy category of its modules.
We describe a Quillen pair defined by Schwede 
and Shipley in \cite{SS03} between a spectral model category $\mKK$
and the category of modules over some `endomorphism' spectral category 
$\mEE$ (depending
on the choice of a set $\mGG$ 
of certain objects in $\mKK$). The induced triangulated adjoint
functors (Lemma~\ref{summary}) form the basis for the proof of 
the characterization theorem.
We will define an appropriate homology functor from the derived
category of $\mEE$ into the abelian category of $\mGG$-modules which
reflects isomorphisms, i.e., the isomorphisms in the derived category of $\mEE$
are exactly the quasi-isomorphisms with respect to this homology functor.
In Section~4.2, we give a brief sketch of the proof for the characterization
theorem. The details are presented in Section~4.3.

The last section gives a lift of one implication of 
the main theorem from Section~4
to the level of model categories.
Using Hirschhorn's existence
theorem for Bousfield localizations \cite[Theorem~4.1.1]{Hirschhorn}, 
we show that a spectral model
category which has a well generated homotopy category is Quillen equivalent
(via a single Quillen functor) to a Bousfield localization of a model
category of modules over some spectral category (Theorem~\ref{lift}). 

%Recall that these
%categories of modules carry the cofibrantly generated model structure 
%of \cite[Theorem~A.1.1]{SS03}, where weak equivalences and fibrations
%are objectwise.

%As in Part~1, \cite[Theorem~3.9.3]{SS03} will be an essential tool;
%also Krause's \cite[Theorem~C]{Krause}.

\section{Well generated categories}

\subsection{Terminology}
Throughout this section we let $\mTT$ be a triangulated category
with (arbitrary set-indexed) coproducts. Let $\alpha$ be an infinite
cardinal. An $\alpha$-\emph{localizing} subcategory
of $\mTT$ is a triangulated subcategory which is closed under 
$\alpha$-coproducts, i.e., coproducts
of strictly less than $\alpha$ objects. A triangulated subcategory
is \emph{localizing} if it is closed under all coproducts.
For $\alpha\geq\aleph_1$, $\alpha$-localizing 
subcategories have countably infinite coproducts and thus are 
\emph{thick}, i.e., 
closed under direct summands \cite[Remark~3.2.7]{Neeman}.
We call a set $\mSS$ of objects in $\mTT$ a
\emph{weak generating} set for $\mTT$ if it is, 
up to isomorphism, closed under 
(de-)sus\-pen\-sions and any object $T\in\mTT$ is zero if and only if 
$\mTT (S,T)=0$ for all $S\in\mSS$. 
If $\mSS$ is a weak
generating set for $\mTT$, then a map $X\to Y$ in $\mTT$ is an isomorphism
if and only if the induced map $\mTT(S,X)\to\mTT(S,Y)$ is an isomorphism
for all $S\in\mSS$.
To see this, consider the cofiber $Z$ of $X\to Y$, which is zero if and only
if $X\to Y$ is an isomorphism, and use that $\mTT(S,-):\mTT\to\Ab$ is 
homological, i.e., it maps the triangle $X\to Y\to Z\to \Sigma X$ to
a long exact sequence of abelian groups.
By $\langle\mSS\rangle$, resp. $\alpha\loc\langle\mSS\rangle$, 
we denote the smallest localizing, resp. $\alpha$-localizing,
subcategory of $\mTT$
which contains a given set of objects $\mSS$. If $\mTT=\langle\mSS\rangle$
then $\mSS$ is called a \emph{generating} set for $\mTT$.
A generating set closed under (de-)sus\-pen\-sions 
is also a weak generating set for $\mTT$. 
(The converse holds
if $\mSS$ is $\aleph_1$-perfect \cite[Proposition~8.4.1]{Neeman}. For
example, any
set of compact generators is $\aleph_1$-perfect. We will give the definition
of $\alpha$-perfect below.)

Let $\alpha$ be an infinite cardinal. An object $T\in\mTT$ is
$\alpha$-\emph{small} if any map $T\to\coprod_{i\in I}X_i$
into an arbitrary coproduct in $\mTT$ factors through some sub-coproduct
\[
T\to\coprod_{i\in I'}X_i\incl\coprod_{i\in I}
X_i
\]
with  $|I'|<\alpha$, i.e., the cardinality of $I'$ is
strictly smaller than $\alpha$.
 
A class $\mSS$ of objects in $\mTT$ is called $\alpha$-\emph{perfect}
(for an infinite cardinal $\alpha$) if it satisfies the following.
\begin{enumerate}[(i)]
\item $0\in \mSS$
\item Any map $S\to\coprod_{i\in I}T_i$ in $\mTT$
with $S\in\mSS$ and $|I|<\alpha$ factors as
\begin{align}\label{composite}
S\to\coprod_{i\in I}S_i\overset{\coprod f_i}{\to}
\coprod_{i\in I}T_i \tag{$\ast$}
\end{align}
with $S_i\in\mSS$ and  maps $f_i : S_i
\to T_i$ in $\mTT$.
\item If a composite such as (\ref{composite}) vanishes, every map
$f_i$ can be factored as
$
S_i\overset{g_i}{\to}
S'_i
\overset{h_i}{\to}T_i
$
with $S'_i\in\mSS$ such that the composite $S\to\coprod_{i\in I}
S_i\to\coprod_{i\in I}S'_i$ already vanishes.
\end{enumerate}

We will now give the definitions of `$\alpha$-compactly 
generated' and `well generated'. The original 
definition is from Neeman's book 
\cite[Definition~8.1.6 and Remark~8.1.7]{Neeman}. 

\begin{defn}\label{alphacompact}
Let $\alpha$ be
an infinite cardinal which is regular (i.e., $\alpha$ cannot 
be written as the sum of less
than $\alpha$ cardinals, all strictly smaller than $\alpha$). 
A set of objects in a
triangulated category with coproducts is called
an $\alpha$-\emph{compact generating set} if it is a 
weak generating set which is $\alpha$-perfect and contains only
$\alpha$-small objects. 

An $\alpha$-\emph{compactly generated category}
is a triangulated category with coproducts which admits an $\alpha$-compact
generating set.
A triangulated category which is $\beta$-compactly generated for some
infinite regular cardinal $\beta$ is called \emph{well generated}.
\end{defn}

Let us recall some basic statements concerning smallness and compactness.
The proofs can be found in \cite[Chapters~3 and~4]{Neeman}.
By $\mTT^{(\alpha)}$
we denote the triangulated subcategory of $\alpha$-small objects in $\mTT$.
It is a thick triangulated subcategory, which is $\alpha$-localizing if 
$\alpha$ is regular.
%\cite[Lemmas~4.1.6 and 4.1.5]{Neeman}
For $\alpha\leq\beta$, we have $\mTT^{(\alpha)}\subset\mTT^{(\beta)}$.
An object $T$ is $\aleph_0$-small if and only if it is compact, that is, if
and only if
the covariant Hom-functor $\mTT(T,-):\mTT\to\Ab$ commutes with coproducts.

Every triangulated subcategory 
$\mSS$ of $\mTT$ contains a unique maximal $\alpha$-perfect class, denoted 
by $\mSS_\alpha$, which is a thick triangulated subcategory. 
We set $\mTT^\alpha=(\mTT^{(\alpha)})_\alpha$ and call the objects
of this thick triangulated subcategory $\alpha$-\emph{compact}. Hence
$\alpha$-compact objects are in particular $\alpha$-small. If $\alpha$ is
regular, the $\alpha$-compact objects form an $\alpha$-localizing 
subcategory. %\cite[Lemma~4.2.5]{Neeman}
For $\alpha\leq\beta$, 
we have $\mTT^\alpha\subset\mTT^\beta$. %\cite[Lemma~4.2.3]{Neeman}
Note that any class of objects in a triangulated category with coproducts
is $\aleph_0$-perfect (use that 
finite coproducts are also products). Hence $\aleph_0$-compact is the same
as $\aleph_0$-small, which
is the same as compact.
 
\begin{rem}\label{Wildschwein}
All objects of an $\alpha$-compact generating set $\mSS$ are $\alpha$-compact. 
Such an $\mSS$ is not only a weak generating set but
also a generating set in the sense that $\mTT=\langle\mSS\rangle$. 
This is because 
$\mSS$ is in particular $\aleph_1$-perfect by \cite[Lemma~4.2.1]{Neeman}.
Note also that $\mTT$ is $\alpha$-compactly generated if and only
if the subcategory $\mTT^\alpha$ 
of $\alpha$-compact objects has a small skeleton which is 
a weak generating set \cite[Remark~8.4.3]{Neeman}. Such a skeleton is then an
$\alpha$-compact generating set for $\mTT$. If $\mTT$ is well generated
then $\mTT^\beta$ is essentially small for all infinite $\beta$.
\end{rem}

The following 
characterization of well generated triangulated categories, which is
due to Krause \cite{Krause},
is easier stated than Neeman's original definition. 
A triangulated
category with coproducts is well generated 
if and only if there is a weak generating set
$\mSS$ consisting of $\alpha$-small objects for some cardinal $\alpha$ 
such that
the following holds: given any set-indexed
family of maps $X_i\to Y_i$, $i\in I$, with the
induced maps
$\mTT(S,X_i)\to \mTT(S,Y_i)$ being surjective for all $S\in\mSS$, then the
induced map
$\mTT(S,\coprod_{i\in I}X_i)\to\mTT(S,\coprod_{i\in I}Y_i)$ is also surjective.
Note that $\alpha$ can be chosen to be regular by enlarging it if necessary.
One important property of well generated categories $\mTT$
is that they satisfy
Brown representability (cf.~\cite[Proposition~8.4.2]{Neeman}). This
means, every homological functor $\mTT^{\op}\to\Ab$ which maps coproducts
to products is
naturally isomorphic to $\mTT(-,X)$ for some $X$.

%\vspace{-.1cm}
\subsection{Subcategories and localizations of well generated categories}
Of course, all compactly generated ($=\aleph_0$-compactly generated)
triangulated categories are well generated. 
An example of a well generated category which is not compactly generated
is the 
derived category of sheaves on a non-compact, connected manifold 
of dimension $\geq 1$. This is discussed in \cite{Neeman-manifold}.
A whole class of examples comes from the following proposition.
For our definition of localization and
the comparison to Neeman's definition see Definition~\ref{loc}
and Remark~\ref{Keller/Neeman}(2).

\begin{prop}\label{Neeman's}
Let $\mTT$ be a well generated triangulated category and $\mTT'$ a 
localizing subcategory
which is generated by a set of objects. 
Then $\mTT'$ is well
generated.

Furthermore, the quotient $\mTT / \mTT'$ is a localization 
of $\mTT$ (and has in particular honest Hom-sets), and $\mTT / \mTT'$ is
also well
generated.
\end{prop}

This proposition does not appear in Neeman's book in this general form, 
so we will give a proof for it. Using the 
tools given in \cite{Neeman}, this will not be
hard although somewhat technical. Recall from Section~1
that the essential image 
$\essim F$ of a triangulated
functor $F:\mTT\to\mTT'$ is the full subcategory consisting of all
objects which are, up to isomorphism, in the image of $F$. It is a
triangulated subcategory if $F$ is full, and a localizing subcategory if,
in addition, $F$ is a coproduct preserving functor between triangulated
categories with coproducts.

\begin{proof}[Proof of Proposition~\ref{Neeman's}]
Let $\mSS'$ be a set with $\mTT'=\langle\mSS'\rangle$. Since
$\mTT$ is well generated, we have $\mTT=\bigcup_\alpha\mTT^\alpha$,
where $\alpha$ runs through all (infinite) regular cardinals 
\cite[Proposition~8.4.2]{Neeman}.
Hence 
\[
\coprod_{S'\in\mSS'}S'\in\mTT^{\alpha_1}
\]
for some regular $\alpha_1$. But $\mTT^{\alpha_1}$ is thick, so we have
$\mSS'\subset\mTT^{\alpha_1}$. Since $\mTT$ is well generated, there exists
an $\alpha_2$-compact generating
set $\mSS$ for some  regular $\alpha_2$.  
If we put $\alpha=\max (\aleph_1 , \alpha_1 , \alpha_2 )$, we have
\[
\mSS\subset\mTT^\alpha,\quad \mSS'\subset\mTT'\cap\mTT^\alpha,\quad\mTT'=
\langle \mSS'\rangle,\quad \textrm{and}\quad\mTT=\langle \mSS\rangle .
\]
That means 
we are in the situation assumed in \cite[Theorem~4.4.9]{Neeman}. This
theorem tells us that for all regular $\beta\geq\alpha$,
\[
\beta \textrm{-loc}\langle\mSS'\rangle =(\mTT')^\beta ,\quad 
\beta \textrm{-loc}\langle\mSS\rangle =\mTT^\beta ,
\]
and the canonical functor 
$\mTT^\beta /(\mTT')^\beta\to\mTT /\mTT'$ factors over an equivalence
\[
\mTT^\beta / (\mTT')^\beta \overset{\simeq}{\to} (\mTT / \mTT')^\beta .
\]
(We need $\alpha\geq\aleph_1$ only for this last equivalence.)

By \cite[Proposition~3.2.5]{Neeman}, 
$\beta\loc\langle\mSS'\rangle$ is essentially small
for all infinite $\beta$. Consequently, $(\mTT')^\beta$ 
is also essentially small for all infinite $\beta$: namely,
for all regular $\beta \geq\alpha$ we have 
$(\mTT')^\beta=\beta\loc\langle\mSS'\rangle$ and for an
arbitrary infinite $\beta$ there exists a regular $\beta'\geq\alpha$
with $\beta\leq\beta'$ and hence $(\mTT')^\beta\subset(\mTT')^{\beta'}$.
Let $\Sk(\mTT')^\beta$ denote a skeleton of $(\mTT')^\beta$.
From $\mSS' \subset \beta \textrm{-loc}\langle\mSS'\rangle =(\mTT')^\beta$
for $\beta\geq\alpha$ we can deduce 
$\mTT' = \langle \mSS'\rangle \subset \langle\Sk(\mTT')^\beta
\rangle$,
hence $\mTT'=\langle\Sk(\mTT')^\beta\rangle$.
This
shows $\mTT'$ is $\beta$-compactly generated for 
$\beta\geq\alpha$ (cf.~Remark~\ref{Wildschwein}) and 
thus well generated.

In particular,  $\mTT'$ satisfies Brown representability. 
By \cite[Proposition~9.1.19]{Neeman}, this implies that $\mTT /\mTT'$
is a localization of $\mTT$ (and hence has honest Hom-sets 
\cite[Remark~9.1.17]{Neeman}).

By Remark~\ref{Wildschwein}, since $\mTT$ is well generated, 
$\mTT^\beta$ is essentially small -- and so is the quotient
$(\mTT / \mTT')^\beta \simeq \mTT^\beta / (\mTT')^\beta$ for any regular
$\beta\geq\alpha$ and
hence for all infinite $\beta$. We now want to show that 
$\Sk(\mTT / \mTT')^\beta$
generates $\mTT / \mTT'$ for $\beta\geq\alpha$. Let $q:\mTT\to\mTT /\mTT'$
denote the canonical triangulated functor into the quotient. Since the
functor $\mTT^\beta / (\mTT')^\beta \to (\mTT / \mTT')^\beta$
is an equivalence, every object of $\Sk(\mTT /\mTT')^\beta$ 
is, up to isomorphism,
of the form $q(X)$ for some $X\in\Sk\mTT^\beta$, so that we have $\langle 
\Sk(\mTT /\mTT')^\beta\rangle = \langle q(\Sk\mTT^\beta)\rangle$. 
Moreover, since
$\mTT /\mTT'$ is a localization, the functor $q$ preserves coproducts
and we can apply Lemma~\ref{essim}. Hence we have
$\langle q(\Sk\mTT^\beta)\rangle\supset\essim 
q\!\mid _{\langle \Sk\mTT^\beta\rangle}$. 
Furthermore, $\langle\Sk\mTT^\beta\rangle=\mTT$ 
because $\mTT$ is in particular $\beta$-compactly generated, and hence
$\essim q\!\mid _{\langle\Sk\mTT^\beta\rangle}=\essim q\!\mid _\mTT=\mTT /\mTT'$.
Altogether, what we get is $\langle
\Sk(\mTT /\mTT')^\beta\rangle \supset
\mTT /\mTT'$ and thus $\langle 
\Sk(\mTT /\mTT')^\beta\rangle =
\mTT /\mTT'$. 
This shows that $\mTT /\mTT'$ is well generated.
\end{proof}

As an immediate consequence of Proposition~\ref{Neeman's}, appropriate
localizations 
of the homotopy category of (symmetric) spectra, which is compactly generated,
are well generated and hence satisfy Brown representability. An example of
a triangulated category which is not well generated is the opposite category
of a non-trivial compactly generated category. This (and other `non-examples')
can be found in \cite[Appendix~E]{Neeman}. The older result of Boardman
\cite{MR0268887}
which says that the stable homotopy category is not self-dual 
can be
regarded as a consequence of this fact.

The following lemma will be useful in Section~4.

\begin{lemma}\label{Methusalix}
Let $\mTT$ be a well generated triangulated category and $\mWW$
a set of maps in $\mTT$ which is, up to isomorphism, closed under
(de-)sus\-pen\-sions, let $\mSS$ be a set consisting of one cofiber
for each map in $\mWW$, and let $\mWW\loc$ consist of
all $X\in\mTT$ for which the induced map
$f^\ast:\mTT(B,X)\to\mTT(A,X)$ is an 
isomorphism for all $f:A\to B$ in $\mWW$.

Then there exists a localization of $\mTT$ with acyclics $\langle\mSS\rangle$
and locals $\mWW\loc$.

\end{lemma}

\begin{proof}
By Proposition~\ref{Neeman's}, the quotient $\mTT / \langle\mSS\rangle$
is a localization and it has $\langle\mSS\rangle$ as subcategory of acyclics.
It remains to show that $\mWW\loc$ is the subcategory of local objects.
By Lemma~\ref{mylemma}(b)(iv) it suffices to check that 
$\langle\mSS\rangle^\perp = \mWW\loc$. Actually, we will show that
\begin{equation}\label{Cleopatra}
\mSS^\perp=\langle\mSS\rangle^\perp=\mWW\loc .
\end{equation}
The inclusion $\mSS^\perp\supset\langle\mSS\rangle^\perp$ follows
immediately from $\mSS\subset\langle\mSS\rangle$. Let us show
$\langle\mSS\rangle^\perp\supset\mWW\loc$: Given $X\in\mWW\loc$ and $C\in\mSS$,
there
exists an exact triangle
\[
A\overset{f}{\to} B\to C\to \Sigma A
\]
with $f\in\mWW$. We apply $\mTT(-,X)$ to the triangle
and get a long exact sequence
\[
\cdots\ot\mTT(A,X)\overset{\,f^\ast}{\ot}\mTT(B,X)\ot\mTT(C,X)
\ot\mTT(\Sigma A,X)\overset{(\Sigma f)^\ast}{\ot}\mTT(\Sigma B,X)\ot\cdots .
\]
Since $X$ is in $\mWW\loc$, the maps $(\Sigma^n f)^\ast$ are isomorphisms
for all integers $n$. Hence $\mTT(\Sigma^n C,X)=0$ and thus $\mSS
\subset\isotope{\perp}{}{(\mWW\loc)}$. Since $\isotope{\perp}{}{(\mWW\loc)}$
is a localizing subcategory, it follows that $\langle\mSS\rangle
\subset\isotope{\perp}{}{(\mWW\loc)}$, which is equivalent to
$\langle\mSS\rangle^\perp\supset\mWW\loc$.

It remains to check $\mWW\loc\supset\mSS^\perp$. Given $X\in
\mSS^\perp$ and $f\in \mWW$, we let $C$ be the cofiber of $f:A\to B$ so
that we obtain a long exact sequence
\[
\cdots\ot\mTT(\Sigma^{-1}C,X)\ot\mTT(A,X)\overset{\,f^\ast}{\ot}
\mTT(B,X)\ot\mTT(C,X)
\ot\cdots .
\]
Now $C$ and $\Sigma^{-1}C$ are in $\mSS$ and $X$ is in $\mSS^\perp$, 
hence we have $\mTT(C,X)=\mTT(\Sigma^{-1}C,X)=0$. This implies
$f^\ast$ is an isomorphism, which shows that $X$ is in $\mWW\loc$.
\end{proof}

%A triangulated category
%is called well generated if it is $\alpha$-compactly generated 
%for some regular
%cardinal $\alpha$ (see \cite[Definition~8.1.6 and Remark~8.1.7]{Neeman} 
%for the original definition or \cite[Theorem~A]{Krause} for a
%nice characterization).
%Note $\aleph_0$-compactly generated is just 
%compactly generated, and $\alpha$-compactly generated is milder than 
%compactly generated if $\alpha > \aleph_0$.

%\input{Section4}
\section{Classification of topological well generated categories}
\label{classification}

\subsection{Spectral model categories versus model categories of modules}

In this section, we will consider
`spectral categories' (or `symmetric ring spectra with several
objects') and modules over such.
From now on, 
when we say `(ring) spectrum' we will
always mean `symmetric (ring) spectrum'.
Spectral categories are $\mCC$-categories as discussed in Section~A.2,
where $\mCC$ is now the closed symmetric monoidal model category 
$(\Sp,\sm,\mS)$ of
symmetric spectra. 
%Let us recall the following definition, \cite[Definition]{SS03}:
This means, a spectral category $\mRR$ consists of a set of objects
and  for any two objects $R$ and $R'$ in $\mRR$ there is
a Hom-spectrum $\mRR (R,R')$ together with an identity `element'
$\mS\to\mRR (R,R)$ for each $R$ in $\mRR$ and composition morphisms
\[
\mRR (R',R'') \sm \mRR (R,R') \to \mRR(R,R'')
\]
which are associative and unital with respect to the identity elements
\cite[Section~3.3]{SS03}.

A \emph{right module} over a spectral category is a spectral functor
\[
X:\mRR^{\op}\to\Sp .
\]
This means, $X$ is a family of spectra $X(R)$, $R\in\mRR$, 
together with maps of spectra
\[
\mRR(R,R')\to\Hom_{\Sp}\left(X(R'),X(R)\right)
\]
which are compatible with composition and identities.
By adjunction,
these maps correspond to a right action of $\mRR$ on X, i.e., maps of spectra
\[
X(R')\sm\mRR(R,R')\to X(R)
\]
which are associative and unital.
The category $\MODD\mRR$ over a spectral category $\mRR$ has as objects 
right $\mRR$-modules, and a morphism $X\to Y$ of $\mRR$-modules is family
of maps of  spectra $X(R)\to Y(R)$ which are compatible with the action
of $\mRR$.
%The \emph{free} module $F_R$ with respect to $R\in\mRR$ is defined
%as $$F_R=\mRR(-, R):\mRR^{\op}\to\Sp .$$
Note that a spectral category which consists only of one object is the
same as a ring spectrum and the modules are just ordinary modules
as considered in Part 1 of this paper.

As in the `one object version', the category $\MODD \mRR$ 
is a spectral (and then by Lemma~\ref{Gutemine} 
also a stable) model category where
maps are weak equivalences 
(resp.~fibrations) if and only if they are objectwise weak equivalences 
(resp.~fibrations) of symmetric spectra in the stable model structure
\cite[Theorem~A.1.1]{SS03}.
The homotopy category of $\MODD \mRR$ will be denoted by $\mD (\mRR^{\op})$
and we call it the \emph{derived category} of $\mRR$ because
it is the topological analog of the derived category of a DG category. 
The \emph{free} modules $F_R=\mRR(-, R):\mRR^{\op}\to\Sp $, for $R\in\mRR$, 
form a set
of compact generators for $\mD(\mRR^{\op})$ \cite[Theorem~A.1.1]{SS03}.
%Note that we
%have to label the $\mRR$ in $\mD (\mRR^{\op})$ with an
%`$\op\,$' to ensure consistency with the notation
%in Part 1, where one object versions of such derived categories have been
%considered.

Let $\mKK$ be a spectral model category, that means the model category
$\mKK$ is enriched, tensored and cotensored over $\Sp$ and the
tensor functor satisfies the pushout product axiom 
\cite[Definition 4.2.1]{Hovey}. Recall that we considered $\mCC$-model 
categories for a more general $\mCC$ in Section~2.2. 
%We denote the tensor functor as $\tensor :\Sp \times \mKK \to \mKK$
%to distinguish it from the internal smash functor $\sm$ of spectra.
Let $\mGG$ be a set of cofibrant and fibrant
objects in $\mKK$ and let $\mEE$ be
the full spectral subcategory of $\mKK$ with objects $\mGG$. In
\cite[Section~3.9]{SS03} Schwede and Shipley define a spectral Quillen
adjunction
\[
\xymatrix@M=.7pc{\mKK\ar@<-.5ex>[rr]_-{\Hom(\mGG,-)}
&&\MODD \mEE ,\ar@<-.5ex>[ll]_-{-\sm_{\mEE}\mGG}}
\]
where for $A\in\mKK$, the value of the right adjoint
is given by $\Hom(\mGG,A):\mEE^{\op}\to\Sp$, $G\mapsto\Hom_\mKK(G,A)$.
The left adjoint is defined by an appropriate coequalizer 
\cite[Theorem~3.9.3]{SS03}. (We considered a one object version of this
Quillen pair in Section~2.3 over a more general category $\mCC$.) 
We have $\Hom(\mGG,G)=F_G$ for $G\in\mGG$, which follows
immediately from the definition of $\Hom(\mGG,-)$.
The counit $\epsilon_G:\Hom(\mGG,G)\sm_\mEE\mGG\to G$ is an isomorphism for
each $G\in\mGG$. The reason is, that for $A\in\mKK$ the induced map
of spectra
\[
\epsilon_G^\ast:\Hom_\mKK(G,A)\to\Hom_\mKK\left(\Hom(\mGG,G)\sm_\mEE\mGG,A\right)
\]
is the composition of the map
\begin{align*}
\Hom_\mKK(G,A)\to & \Hom_{\MODD\mEE}\left(\Hom(\mGG,G),\Hom(\mGG,A)\right)\\
= & \Hom_{\MODD\mEE}\left(\Hom_\mKK(-,G),\Hom_\mKK(-,A)\right),
\end{align*}
which is an isomorphism by the enriched Yoneda lemma, with the adjunction
isomorphism
\[
 \Hom_{\MODD\mEE}\left(\Hom(\mGG,G),\Hom(\mGG,A)\right)
\iso \Hom_\mKK\left(\Hom(\mGG,G)\sm_\mEE\mGG,A\right).
\]
This shows that $\epsilon_G$ is an isomorphism. In particular, we obtain
an isomorphism 
\[
F_G\sm_\mEE\mGG\iso G.
\]

Let us denote the derived adjunction by
\[
\xymatrix@M=.7pc{\Ho\mKK\ar@<-.5ex>[r]_-{F}
&\DE.\ar@<-.5ex>[l]_-{J}}
\]
Since both $\mKK$ and $\MODD \mEE$ are stable, $\Ho\mKK$ and
$\DE$ are triangulated categories with coproducts and
$J$ and $F$ are triangulated functors.
Since $G\in\mGG$ is cofibrant we have $F(G)\iso F_G$ on the homotopy level.

We want to see that the counit $\varepsilon:JF(G)\to G$ is also
an isomorphism for objects $G$ of $\mGG$. The following is in general true
for Quillen pairs
\[
\xymatrix@M=.7pc{\mCC\ar@<-.5ex>[r]_-{U}
&\mDD .\ar@<-.5ex>[l]_-{V}}
\]
If $C$ is a fibrant object of $\mCC$ such that $U(C)$ is cofibrant
in $\mDD$ and the counit 
\[
\epsilon_C:VU(C)\to C
\]
is an isomorphism in $\mCC$, 
then the counit $\varepsilon_C:(LV)(RU)(C)\to C$ of the derived adjunction
is an isomorphism in $\Ho\mCC$.
In our case, $G$ is fibrant in $\mKK$ by assumption. Thus, in order
show that the counit $\varepsilon:JF(G)\to G$ is an isomorphism, it
suffices to prove that
$F(G)\iso F_G$ is cofibrant in $\MODD \mEE$. This can be seen
by analyzing the proof of  
\cite[Theorem~A.1.1]{SS03}. 
%Alternatively, if we accept 
%that $\MODD\mEE$
%has a model structure with weak equivalences and fibrations objectwise
%in $\Sp$, the existence of a lift in the diagram
%\[
%\xymatrix@M=.7pc{\ast\ar[r]\ar[d]&M\ar@{->>}[d]^\sim \\
%  F(G)\ar[r]\ar@{.>}[ur]&N}
%\]
%in $\MODD\mEE$, where the right vertical map is a trivial fibration,
%is via the Yoneda lemma equivalent to the existence of a lift
%in the diagram
%\[
%\xymatrix@M=.7pc{\ast\ar[r]\ar[d]&M(G)\ar@{->>}[d]^\sim \\
% G\ar[r]\ar@{.>}[ur]&N(G)}
%\]
%in $\mKK$. Here, the right vertical map is a trivial fibration because
%trivial fibrations in $\MODD\mEE$ are objectwise. Since $G$ is cofibrant
%by assumption, the lift does exist.
Alternatively, if we accept 
that $\MODD\mEE$
has a model structure with weak equivalences and fibrations objectwise
in $\Sp$, we can consider the Quillen pair
\[
\xymatrix@M=.7pc{\MODD\mEE\ar@<-.5ex>[rr]_-{\ev_G}
&&\MODD \mS = \Sp \ar@<-.5ex>[ll]_-{-\sm F_G}}
\]
where $\ev_G(X)=X(G)$ and $Y\sm F_G$ is given by
$(Y\sm F_G)(G')=X\sm F_G(G')=X\sm \mEE(G',G)$
with the obvious right action of $\mEE$. Applying the left Quillen
functor $-\sm F_G$ to the (cofibrant) sphere spectrum $\mS$
shows that $F_G$ is a cofibrant module.

The unit $\eta_{F_G}:F_G\to FJ(F_G)$ 
is also an isomorphism for every free module
$F_G$. This follows from $F_G\iso F(G)$ and $F(\varepsilon_G)\eta_{F(G)}
=\id_{F(G)}$, which holds in general for adjunctions.

Let us summarize these facts in the following

\begin{lemma}\label{summary}
Let  $\mKK$ be a spectral model category, $\mGG$ a set of cofibrant
and fibrant objects, and $\mEE$ the full spectral subcategory of $\mKK$
with objects $\mGG$.

Then the derived category $\DE$ of $\mEE$ has the free
modules $F_G$, $G\in\mGG$, as a compact generating set. There is
an adjoint pair of triangulated functors 
\[
\xymatrix@M=.7pc{\Ho\mKK\ar@<-.5ex>[r]_-{F}
&\DE,\ar@<-.5ex>[l]_-{J} }
\]
under which the objects in $\mGG$ correspond to the free modules, that is, 
$F(G)\iso
F_G$ and $J(F_G)\iso G$, such that the counit $\varepsilon_G :JF(G)\to G$
and  the unit $\eta_{F_G} : F_G\to FJ(F_G)$ are isomorphisms for $G\in\mGG$.

\hfill $\square$
\end{lemma}

\begin{rem}\label{Vercingetorix}
As the counit $\varepsilon_G:JF(G)\to G$ is an isomorphism for $G\in\mGG$
the functor $F$ is fully faithful when the source is in $\mGG$, that
is, for $G\in\mGG$ and arbitrary $A\in\Ho\mKK$, the map
\begin{eqnarray*}
[G,A]^{\Ho\mKK}&\to&[F(G),F(A)]^\DE ,\\
 g&\mapsto &F(g) ,
\end{eqnarray*}
is an isomorphism.
Namely, if we compose this map with the adjunction isomorphism
$[F(G), F(A)]^\DE\to[JF(G),A]^{\Ho\mKK}$, we obtain the map
\[
\varepsilon_G^\ast:[G,A]^{\Ho\mKK}\to[JF(G),A]^{\Ho\mKK}
\]
induced
by the colimit $\varepsilon_G$, which is
an isomorphism by Lemma~\ref{summary}.
\end{rem}

Schwede and Shipley have shown that the Quillen pair
\[
\xymatrix@M=.7pc{\mKK\ar@<-.5ex>[rr]_-{\Hom(\mGG,-)}
&&\MODD \mEE ,\ar@<-.5ex>[ll]_-{-\sm_{\mEE}\mGG}}
\]
is in fact a Quillen equivalence if
$\mGG$ is a set of compact generators for $\Ho\mKK$. 
In particular, $F$ and $J$ are then inverse
equivalences. We will study the case where $\mGG$
is not necessarily a compact generating set for $\Ho\mKK$ but
only an appropriate $\alpha$-compact generating set. The goal is to
show that under this assumption $F$ is fully faithful and hence
$\Ho\mKK$ is a localization of $\DE$.

As a triangulated category, $\Ho \mKK$ is in particular an $\Ab$-category
(or a ring with several objects),
i.e., enriched over the closed symmetric monoidal category $\Ab$ of 
abelian groups. 
Regarding the set $\mGG$ of objects in $\mKK$ as a full subcategory
of 
$\Ho \mKK$, we get an $\Ab$-category which we denote for simplicity
again by $\mGG$. 

\begin{defn}
The category $\MODD \mGG$ of right $\mGG$-modules has
as objects the $\Ab$-functors $\mGG^{\op}\to\Ab$ and as morphisms natural 
transformations. We define a functor
\[
H_0:\DE\to\MODD\mGG
\]
by $H_0(X)=[F|_\mGG (-),X]:\mGG^{\op}\to\Ab$. 
\end{defn}

The category $\MODD \mGG$ is abelian and has objectwise defined
coproducts and products. Recall 
that a functor from a triangulated category to an abelian category is called
\emph{homological} if it maps exact triangles to (long) exact sequences
\cite[Definition~1.1.7]{Neeman}.

\begin{lemma}
The functor $H_0:\DE\to\MODD \mGG$ is homological and preserves
coproducts and products. If $\mGG$ is up to isomorphism closed under 
(de-)sus\-pen\-sions,
then $H_0$ reflects isomorphisms, i.e., $f:X\to Y$ is an isomorphism
in $\DE$ if and only if $H_0(f):H_0(X)\to H_0(Y)$ is an isomorphism
in $\MODD\mGG$.
\end{lemma}

\begin{proof}
Consider an exact
triangle $X\to Y \to Z \to \Sigma X$ in $\DE$. The induced
sequence
\[
\cdots\to H_0(X)\to H_0(Y)\to H_0(Z)\to H_0(\Sigma X)\to \cdots
\]
is exact if and only if it is objectwise exact, i.e., 
\[
\cdots\to [F(G),X]\to [F(G),Y]
\to [F(G),Z]\to [F(G),\Sigma X]\to \cdots
\]
is a long exact sequence for each $G\in\mGG$. But this is true
since $[F(G),-]:\DE\to \Ab$ is homological.

Since products in $\MODD\mGG$ are objectwise, it is clear that $H_0$ 
preserves them. Now consider a 
family $(X_i)_{i\in I}$ of objects in $\DE$. Since
$F(G)$ is compact, the canonical map
\[
\bigoplus_{i\in I} [F(G),X_i] \to \Bigl[F(G),\coprod_{i\in I}X_i\Bigr]
\]
is an isomorphism for every $G\in\mGG$. Hence 
$\coprod_{i\in I}H_0(X_i)\to H_0\left(\coprod_{i\in I}X_i\right)$ 
is an isomorphism in $\MODD\mGG$. 

Now let $\mGG$ be closed under (de-)sus\-pen\-sions (up to isomorphism).
Then the set 
\[F(\mGG)=\{F(G)\,|\,G\in\mGG\}
\]
is in particular
a weak generating set for $\DE$ (as a consequence
of Lemma~\ref{summary}). Given a map $f:X\to Y$ in $\DE$ with
cofiber $Z$. We then have the following logical equivalences:
\begin{eqnarray*}
f \textrm{ is an isomorphism }
&\Longleftrightarrow& Z\iso 0 \\
&\Longleftrightarrow&[F(G),Z]\iso 0 \textrm{ for all }G\in\mGG \\
&\Longleftrightarrow&[F(G),X]\to [F(G),Y] \textrm{ is an isomorphism for all }
G\in\mGG \\
&\Longleftrightarrow&H_0(X)\to H_0(Y) \textrm{ is an isomorphism in } \MODD\mGG,
\end{eqnarray*}
where the third equivalence uses that 
$[F(G),-]$ is homological.

%By triangulated category theory, $f$ is an isomorphism
%if and only if $Z$ is isomorphic to $0$. Since $F(\mGG)$ is a weak generating
%set, this is equivalent to $[F(G),Z]\iso 0$ for all $G\in\mGG$.
%But $[F(G),-]$ is a homological 
%functor, so $[F(G),Z]\iso 0$ for all $G\in\mGG$
%if and only if $[F(G),X]\to [F(G),Y]$ is an isomorphism for all $G\in\mGG$
%-- which is the same as saying
\end{proof}

%\newpage

\subsection{The characterization theorem and the strategy of proof}
Let us state the main result, Theorem~\ref{characterization}, 
fix some notation, 
and sketch the proof before giving 
the details in Section~4.3. Porta considers those well generated
triangulated categories which are algebraic, that is, triangulated 
equivalent to the stable category of a Frobenius category 
\cite[Section~3.6]{DGCats}. One feature of algebraic categories 
is that they allow certain `derived Hom-functors' into derived categories
of DG categories. In the topological case, we would like to have
such derived Hom-functors into derived categories of spectral categories.
Lemma~\ref{summary} provides such a functor $F$ if the triangulated category
in question is the homotopy category of a spectral model category. This leads
us to the following

\begin{defn}
A triangulated category $\mTT$ is called \emph{topological} if it is 
triangulated equivalent to the homotopy category of a spectral model 
category.
\end{defn}

\begin{exs}\label{spectral}
In particular, any stable model category Quillen equivalent
to a spectral model category has a homotopy category which is topological.
Here are three classes of such model categories.

\begin{enumerate}[(1)]
\item Schwede and Shipley have proved that every simplicial, cofibrantly
generated, proper stable model category is Quillen equivalent to a spectral
model category \cite[Theorem~3.8.2]{SS03}. 
\item By \cite[Theorems~9.1 and~8.11]{Hovey-Spectra}, a simplicial, cellular,
left proper stable model category for which the domains of the generating
cofibrations are cofibrant is Quillen equivalent to a spectral model category.
\item Another class of examples arises from \cite{Dugger}. Using 
\cite[Theorem~8.11]{Hovey-Spectra} one can deduce from
\cite[Propositions~5.5(a) and~5.6(a)]{Dugger} that presentable stable
model categories are Quillen equivalent to spectral model categories.
\end{enumerate}
\end{exs}

By \cite[Theorem~A.1.1]{SS03}, the derived category of a spectral 
category is compactly
generated. Hence
Proposition~\ref{Neeman's} yields one implication %((ii) $\Rightarrow$ (i))
of the following characterization theorem.

\begin{thm}\label{characterization}
Let $\mTT$ be a topological triangulated category. Then the following are
equivalent.
\begin{enumerate}
\item[\textnormal{(i)}]$\mTT$ is well generated.
\item[\textnormal{(ii)}]$\mTT$ is triangulated equivalent to a localization
of the derived category of a spectral category where the acyclics are
generated by a set.

\end{enumerate}

\end{thm}

The implication (i) $\Rightarrow$ (ii)
is more involved and the proof is given in Section~4.3.
Let us from now on and for the rest of this paper use the following

\begin{notation}\label{notation}
We assume that $\mKK$ is a spectral model category having
a well generated homotopy category $\Ho\mKK$. Then there exists a
regular infinite cardinal $\alpha$ such that the full
subcategory $(\Ho\mKK)^\alpha$ of $\alpha$-compact objects in $\Ho\mKK$
has a small skeleton $\mGG$ which is an $\alpha$-compact generating
set for $\Ho\mKK$ (see Remark~\ref{Wildschwein}).
We fix such a cardinal $\alpha$ and 
such a generating set $\mGG$, which is then, up to isomorphism,
closed under (de-)sus\-pen\-sions, triangles, and $\alpha$-coproducts (that
is, coproducts of strictly less than $\alpha$ objects).
Moreover, by choosing cofibrant and fibrant 
replacements, we can assume all objects in $\mGG$ are cofibrant and fibrant
in $\mKK$. 
We let $\mEE$ be the full spectral subcategory of $\mKK$ with objects
$\mGG$.
As above, by slight abuse of notation, we regard $\mGG$ not
only as a set but also as an $\Ab$-category with objects $\mGG$. Since
$\mGG$ is in particular closed under finite coproducts 
and contains a zero object, it is also an additive category.
\end{notation}

%\newpage

Consider the following diagram.
\begin{equation}\label{diagram}
\xymatrix@M=.7pc{\Ho\mKK\ar@{.>}@<-.5ex>[dr]_{\tF}
\ar@<0ex>@/^1pc/[drr]_{F}\ar@<0ex>@/_1pc/[ddr]_\Phi & & \\
  &\DAE\ar@<-1ex>@{^{(}->}[r]_R
\ar[d]^{\tH} \ar@{.>}@<-.5ex>[ul]_{\tJ}
& \DE\ar@{.>}@<-1ex>[l]_L
\ar[d]^{H_0} \ar@<-1ex>@/_1pc/[ull] _J \\
  & \MODDalpha\mGG \ar@{^{(}->}[r]_r&\MODD\mGG }
\end{equation}

Here $\MODDalpha\mGG$ is a suitable subcategory of $\MODD\mGG$
(see Definition~\ref{Yellowsubmarine})
and $\DAE$ is the corresponding subcategory of
$\DE$ of those objects whose homology lies in $\MODDalpha\mGG$
(see Definition~\ref{Appendix}).
The pair $(J,F)$ is the adjoint pair from Lemma~\ref{summary}. 
It restricts
to an adjoint pair $(\tJ,\tF)$. 

If we can show that
\begin{enumerate} 
\item $\DAE$ is a localization of
$\DE$, i.e., there exists a left adjoint $L$ for the inclusion
$R$, and
\item $\tJ$ and $\tF$ are inverse equivalences
of triangulated categories,
\end{enumerate}
then it follows that $\Ho\mKK$ is a localization of $\DE$.

For the proof of (2) we will consider the unit
\[
X\to\tF\tJ(X)\quad\textrm{and the counit}\quad\tJ\tF(A)\to A
\]
 of
the adjunction and show that they are isomorphisms for all objects
$X\in\DAE$ and $A\in\Ho\mKK$. This is easy to see for the
free modules $F_G$ (which lie in fact in the subcategory $\DAE$ and
form a set of generators for it) and for the generators $G\in\mGG$
of $\Ho\mKK$.
Then it suffices to prove that both $\tJ$ and $\tF$ preserve coproducts.
Of course, in the case of the left adjoint $\tJ$ this is true. 
The problem is to show
that $\tF$ also preserves coproducts. In his paper
\cite{Krause}, Krause defines a functor $\Phi:\Ho\mKK\to\MODDalpha\mGG$
which preserves coproducts and is isomorphic to $\tH\tF$ (where
$\tH$ is the restriction of $H_0$). This helps us to reduce
the question whether $\tF$ preserves coproducts to the following:
Does $\tH$ preserve coproducts of objects which are in the essential image
of $\tF$?

%It will be essential that $\mGG$ is a generating set for $\Ho\mKK$
%as in Notation~\ref{notation}.

\subsection{Proof of the characterization theorem}
Recall that $\mKK$, $\alpha$, $\mGG$, and $\mEE$ are
as in Notation~\ref{notation}.
Let us first define the categories $\MODDalpha\mGG$ and $\DAE$ 
occurring in the diagram (\ref{diagram}). We will then prove a series of
lemmas and finally the remaining part
of Theorem~\ref{characterization}.

\begin{defn}\label{Yellowsubmarine}
The category $\MODDalpha\mGG$ is defined
as the full subcategory of $\MODD\mGG$ with
objects those functors $\mGG^{\op}\to\Ab$ which send $\alpha$-coproducts 
in $\mGG$ to products in $\Ab$. The functor 
\[
\Phi:\Ho\mKK\to\MODDalpha\mGG
\]
is defined by $A\mapsto [-,A]|_\mGG$.

\end{defn}

Note that $\Phi$ indeed takes values in $\MODDalpha\mGG$: 
the functor $[-,A]$ maps even arbitrary coproducts
in $\Ho\mKK$ to products in $\MODDalpha\mGG$.
Neeman denotes the category $\MODDalpha\mGG$
by $\mathcal{E}x(\mGG^{\op},\mathcal{A}b)$ 
\cite[Defintion~6.1.3]{Neeman}, Krause uses 
$\mathrm{Prod}_\alpha (\mGG^{\op},\mathrm{Ab})$ \cite{Krause}.
It is an abelian subcategory of $\MODD\mGG$, which is closed under
products \cite[Lemma~6.1.4 and Lemma~6.1.5]{Neeman}, but
it is not closed under coproducts in general. Nevertheless, $\MODDalpha\mGG$
does have coproducts -- which cannot be objectwise in general, 
since they have
to be different from those in $\MODD\mGG$. An explicit description of the
coproducts in $\MODDalpha\mGG$
can be found in Neeman's book \cite[Section~6.1]{Neeman} (the 
definition together with the proof of the universal property takes twelve
and a half pages). Krause shows that $\MODDalpha\mGG$ is equivalent
to the category of coherent functors $(\mathrm{Add}\,\mGG)^{\op}\to\Ab$,
where $\mathrm{Add}\,\mGG$ denotes the closure of $\mGG$ in $\Ho\mKK$
under all coproducts
and direct summands 
\cite[Lemma~2]{Krause}. In the category of these coherent functors, 
coproducts have a nicer description \cite[Lemma~1]{Krause-coherent}.
However, we do not need to know what the coproducts in $\MODDalpha\mGG$
look like -- the only thing we will need is the fact that 
the functor $\Phi:\Ho\mKK\to\MODDalpha\mGG$
preserves coproducts \cite[Theorem~C]{Krause}.

%\pagebreak
\begin{defn}\label{Appendix}
The category $\DAE$ is the 
full subcategory of $\DE$ 
having as objects those $X$ for which $H_0(X)$ is in 
$\MODDalpha \mGG$. The functor 
\[
\tH:\DAE\to\MODDalpha\mGG
\]
is the restriction of $H_0$.
\end{defn}

\begin{rem}\label{Amnesix}
Note that $\tH$ reflects isomorphisms since $H_0$ does. Moreover, 
$\tH$ is homological, that is, it sends exact triangles
to long exact sequences. The reason is the following:
$H_0R$ is homological because
$R$ is a triangulated functor and $H_0$ is homological, so $r\tH=H_0R$ is
homological. 
Since $r$ is the inclusion of an abelian subcategory, 
$\tH$ has to be homological.
But it is not so easy to see that $\tH$ preserves coproducts -- this will be
a consequence of Proposition~\ref{eq}.
\end{rem}

\begin{lemma}\label{aleph_0}
The category $\DAE$ is a triangulated subcategory of $\DE$. It is
colocalizing, i.e., closed under products. If $\alpha=\aleph_0$ then
$\DAE=\DE$.
\end{lemma}

\begin{proof}
Let $(G_i)_{i\in I}$ be a family of objects in $\mGG$ with
$|I|<\alpha$. Assume $X$ is in $\DAE$. Then the canonical map
\[
\Bigl[F\Bigl(\coprod_{i\in I}G_i\Bigr),X\Bigr]\to\prod_{i\in I}[F(G_i),X]
\]
is an isomorphism. We want to show that $\Sigma X$ is also in $\DAE$.
Since both $F$ and coproducts commute with the desuspension $\Sigma^{-1}$,
the canonical map
\[
\Bigl[F\Bigl(\coprod_{i\in I}G_i\Bigr),\Sigma X\Bigr]
\to\prod_{i\in I}[F(G_i),\Sigma X]
\]
is isomorphic to
\[
\Bigl[F\Bigl(\coprod_{i\in I}\Sigma^{-1}(G_i)\Bigr),X\Bigr]
\to\prod_{i\in I}[F\Sigma^{-1}(G_i),X].
\]
But this map is an isomorphism because $\mGG$ is closed under
desuspensions and $X$ is in $\DAE$. Similarly, $\DAE$
is closed under desuspensions.

Consider a triangle $X\to Y\to Z\to \Sigma X$ in $\DE$ such that
$X$ and $Y$ are in $\DAE$. Since $H_0$ is homological and products
of exact sequences in $\Ab$ are exact again, we get a commutative
diagram of abelian groups with long exact columns:
\[\xymatrix@u@M=.7pc{\vdots
&(H_0 \Sigma X)\Bigl(\coprod_{i\in I}G_i\Bigr)\ar[d]^\iso\ar[l]
&(H_0Z)\Bigl(\coprod_{i\in I}G_i\Bigr)\ar[d]\ar[l]
&(H_0Y)\Bigl(\coprod_{i\in I}G_i\Bigr)\ar[d]^\iso\ar[l]
&(H_0X)\Bigl(\coprod_{i\in I}G_i\Bigr)\ar[d]^\iso\ar[l]
&\vdots\ar[l]\\
\vdots
&\prod_{i\in I}(H_0\Sigma X)(G_i)\ar[l]
&\prod_{i\in I}(H_0Z)(G_i)\ar[l]
&\prod_{i\in I}(H_0Y)(G_i)\ar[l]
&\prod_{i\in I}(H_0X)(G_i)\ar[l]
&\vdots\ar[l]
}
\]
We can apply the 5-lemma and get and isomorphism
\[
(H_0Z)\Bigl(\coprod_{i\in I}G_i\Bigr)\overset{\iso}{\to}
\prod_{i\in I}(H_0Z)(G_i).
\]
This shows $Z$ is in $\DAE$ and thus $\DAE$ is closed under triangles.
Since $\MODDalpha\mGG$ is closed under products in $\MODD\mGG$ and 
$H_0$ preserves products, $\DAE$ is also closed under products. 

Now let $\alpha=\aleph_0$. Then $\MODDalpha\mGG$ contains all additive
functors
$\mGG^{\op}\to\Ab$ which map finite coproducts to products. But this
is true for all additive functors and thus $\MODDalpha\mGG=\MODD\mGG$ and
$\DAE=\DE$.
\end{proof}

\begin{lemma}\label{Hund}
The functor $F:\Ho\mKK\to\DE$ factors through $\DAE$. Consequently
we get an adjoint pair of triangulated functors
\[
\xymatrix@M=.7pc{\Ho\mKK\ar@<-.5ex>[r]_-{\tF}
&\DAE.\ar@<-.5ex>[l]_-{\tJ} }
\]
Moreover, the composition $\tH\tF$ is isomorphic to $\Phi$.
\end{lemma}

\begin{proof}
We have to check that for $A\in\Ho\mKK$ the functor $H_0F(A)$ 
sends $\alpha$-coproducts in $\mGG$
to products in $\Ab$. Let $(G_i)_{i\in I}$ be a family in $\mGG$ with
$|I|<\alpha$. Using the adjunction $(J,F)$ and the fact that its counit
$JF(G)\to G$ is an isomorphism for $G\in\mGG$ (see Lemma~\ref{summary})
we can conclude that the canonical map
\[
\Bigr[F\Bigl(\coprod_{i\in I}G_i\Bigr), F(A)\Bigr]^\DE
\to\prod_{i\in I}[F(G_i),F(A)]^\DE
\]
is isomorphic to the map $\bigl[\coprod_{i\in I}G_i,A]^{\Ho\mKK}\to 
\prod_{i\in I}[G_i,A]^{\Ho\mKK}$, which is an isomorphism by the universal
property of the coproduct. Hence $H_0F(A)$ 
maps $\alpha$-coproducts to products.
This yields a functor $\tF:\Ho\mKK\to\DAE$ with $R\tF=F$, which is left
adjoint to $\tJ=JR$.

Using Remark~\ref{Vercingetorix}, we get an isomorphism 
\[
\tH\tF(A)(G)= H_0F(A)(G)=[F(G),F(A)]^\DE\iso[G,A]^{\Ho\mKK}=\Phi(A)(G)
\]
which is natural in $G\in\mGG$ and $A\in\Ho\mKK$.
\end{proof}

\begin{prop}\label{locprop}
The category $\DAE$ is a localization of
$\DE$, i.e., there exists a left adjoint $L$ for the inclusion
$R$. 
Let $\mWW$ denote the set of the canonical maps
\[
\coprod_{i\in I}F(G_i) \to F\left(\coprod_{i\in I}G_i \right)
\]
where $(G_i)_{i\in I}$ runs through all families in $\mGG$ with
$|I|<\alpha$. (Strictly speaking, we allow one and only one set $I$
for each cardinality smaller than $\alpha$ to ensure all the considered
maps really form a \emph{set}.)
Then the acyclics are generated by the set $\mSS$ containing one
cofiber for each map in $\mWW$.
Moreover, the subcategory $\mWW\loc$ (see Lemma~\textnormal{\ref{Methusalix}})
is equal to $\DAE$.
\end{prop}

\begin{proof}
Since $\mGG$ is closed under (de-)sus\-pen\-sions so is $\mWW$. 
We know that $\DE$ is well generated (even
compactly generated by the free modules), so we can apply 
Lemma~\ref{Methusalix} and get a localization of $\DE$ with 
$\mWW\loc$ as the class of local objects and $\langle\mSS\rangle$
as the class of acyclics.
An object $X$ of $\DE$ is in $\DAE$ by definition if and only if
$H_0(X):\mGG^{\op}\to \Ab$ sends $\alpha$-coproducts to products, i.e., 
if and only if the canonical map
\begin{equation}\label{Teutates}
H_0(X)\Bigl(\coprod_{i\in I}G_i\Bigr)\to\prod_{i\in I}H_0(X)(G_i)
\end{equation}
is an isomorphism. Using the definition of $H_0$ and the fact
that $[-,X]:\DE\to\Ab$ maps coproducts to
products we see that the map (\ref{Teutates})
is isomorphic to the map
\[
\Bigl[F\Bigl(\coprod_{i\in I}G_i\Bigr),X\Bigr]\to\Bigl[\coprod_{i\in I} 
F(G_i),X\Bigr],
\]
which is an isomorphism if and only if $X$ is in $\mWW\loc$.
This shows $\DAE=\mWW\loc$. By Remark~\ref{Minerva}, $\DAE=\mWW\loc$
is equivalent to $\DE/\langle\mSS\rangle$ and thus
a localization of $\DE$ with acyclics $\langle\mSS\rangle$.
Let $L$ be the composition
\[
\DE\to\DE/\langle S\rangle\overset{\simeq}{\to}\DAE.
\]
The inclusion $R:\DAE\to\DE$ is then a fully faithful 
right adjoint for $L$.
\end{proof}

As an immediate consequence of Proposition~\ref{locprop} 
we get the
following
\begin{cor}
The category $\DAE$ has coproducts, namely 
$
\coprod_{i \in  I} X_i = L\left(\coprod_{i 
\in  I} RX_i \right)
$
for any indexing set $I$. 
The canonical map $X_j\to\coprod_i X_i$ is the composition
\[
X_j \overset{\iso}{\to}LR(X_j)\to L \left(\coprod_{i\in I}RX_i\right),
\]
where the first map is the inverse of the counit (which is an isomorphism
since the right adjoint $R$ is fully faithful) and the
second map is $L$ applied to the canonical map of the coproduct in $\DE$.
\hfill$\square$
\end{cor}

\begin{rem}
The inclusion of abelian categories $r:\MODDalpha\mGG\to\mGG$
has a left adjoint, too. But since this left adjoint (which is 
discussed in \cite[Section~7.5]{Neeman})
is not exact in general, it
will not be useful for us.
\end{rem}

\begin{lemma}\label{Alaaf}
The functor $\tF:\Ho\mKK\to\DAE$ is isomorphic to the composition
$LF$. Moreover, $\tF$ preserves $\alpha$-coproducts of objects which
lie in $\mGG$.
\end{lemma}

\begin{proof}
Since $R$ is fully faithful we have an isomorphism $LR\iso\id_\DE$,
which is given by the counit of the adjoint pair $(L,R)$. This yields
an isomorphism $LF=LR\tF\iso\tF$.

As a consequence of Proposition~\ref{locprop}, for any family
$(G_i)_{i\in I}$ with $G_i\in\mGG$ and $|I|<\alpha$, the map
\[
L\left(\coprod_{i\in I}F(G_i)\right) \to LF\left(\coprod_{i\in I}G_i \right)
\]
is an isomorphism. As a left adjoint, $L$ commutes with coproducts.
Using the isomorphism $LF\iso\tF$, we see that $\tF$ preserves
$\alpha$-coproducts of objects in $\mGG$.

\end{proof}

\begin{lemma}\label{Heidewitzka}
The functor $L:\DE\to\DAE$ preserves $\alpha$-compact objects.
\end{lemma}

\begin{proof}
Since $\DE$ is well generated (even compactly generated), every object
is $\beta$-compact for some $\beta$ \cite[Proposition~8.4.2]{Neeman}.
By \cite[Lemma~4.4.4]{Neeman} it suffices then to show
that the acyclics of the localization 
\[
\xymatrix@M=.7pc{ \DAE  \ar@<-1ex>@{^{(}->}[r]_R
& \DE \ar@<-1ex>[l]_L}
\]
have a generating set containing only $\alpha$-compact objects.
Let $C$ be a cofiber of a map
\[
\coprod_{i\in I}F(G_i) \to F\left(\coprod_{i\in I}G_i \right)
\]
where $G_i\in\mGG$ and $|I|<\alpha$. The acyclics are generated
by such cofibers $C$ (Proposition~\ref{locprop}).
Recall that all $\alpha$-compact objects form an $\alpha$-localizing 
triangulated subcategory. As a free module,
each $F(G_i)$  is compact (Lemma~\ref{summary})
and hence $\alpha$-compact -- and so is
the $\alpha$-coproduct $\coprod_{i\in I}F(G_i)$.
Since $\coprod_{i\in I}G_i$
is up to isomorphism in $\mGG$, the object
$F\left(\coprod_{i\in I}G_i \right)$ is also $\alpha$-compact and so
is the cofiber $C$.

\end{proof}

\begin{lemma}\label{Schaf}
The functor $\tF:\Ho\mKK\to\DAE$ is fully faithful when the
source is in $\mGG$, that
is, for $G\in\mGG$ and arbitrary $A\in\Ho\mKK$, the map
\begin{eqnarray*}
[G,A]^{\Ho\mKK}&\to&[\tF(G),\tF(A)]^\DAE ,\\
 g&\mapsto &\tF(g) ,
\end{eqnarray*}
is an isomorphism.
\end{lemma}

\begin{proof}
We will again use $\tF=LF$ (Lemma~\ref{Alaaf}). 
By Remark~\ref{Vercingetorix}, $F$ is fully faithful when the source is
in $\mGG$.
Since $R$ is fully faithful,
the left adjoint $L$ is fully faithful on $\essim R$ and in particular
on $\essim F \subset \essim R$. 

%It hence suffices to check that $F$ is fully
%faithful on $\mGG$. For $G\in\mGG$ and $A\in\Ho\mKK$ compose
%the map
%\begin{eqnarray}\label{Samba-Samba}
%[G,A]^{\Ho\mKK}&\to&[F(G),F(A)]^\DE ,\\
% g&\mapsto &F(g) ,\nonumber
%\end{eqnarray}
%with the adjunction isomorphism $[F(G),F(A)]^\DE\to[JF(G),A]^{\Ho\mKK}$.
%By adjoint functor theory, this composition is just the map
%\[
%\varepsilon^\ast_G:[G,A]^{\Ho\mKK}\to[JF(G),X]^{\Ho\mKK}
%\]
%induced by the counit $\varepsilon_G :JF(G)\to G$ which is an isomorphism
%(Lemma~\ref{summary}). Thus the map (\ref{Samba-Samba}) is an isomorphism,
%as was to be shown. 

\end{proof}

\begin{lemma}\label{genset}
The set $\tF(\mGG)=\{\tF(G)\,|\,G\in\mGG\}$ is an $\alpha$-compact
generating set (in the sense of Definition~\textnormal{\ref{alphacompact}}) for
$\DAE$.
\end{lemma}

\begin{proof}
First of all, $\tF(\mGG)$ is a weak generating set for
$\DAE$. Recall what this means: it is closed
under (de-)sus\-pen\-sions up to isomorphism,
and $[\tF(G),X]^\DAE=0$ for all $G\in\mGG$ implies $X=0$. 
This holds because
$F(\mGG)$ is a weak generating set for $\DE$. 
More is true, we have $\langle\tF(\mGG)\rangle =\DAE$.
Namely, using the Lemmas~\ref{Alaaf} and \ref{essim}, we get
\[
\langle\tF(\mGG)\rangle =
\langle LF(\mGG)\rangle \supset \essim L|_{\langle F(\mGG)\rangle} =
\essim L = \DAE 
\]
and thus $\langle\tF(\mGG)\rangle =\DAE$.

By Lemma~\ref{Heidewitzka}, each
$\tF(G)=LF(G)$ is $\alpha$-compact in $\DAE$.
Note that this does not yet imply that $\tF(\mGG)$ is an $\alpha$-compact
generating set for $\DAE$ -- we still have to show it is $\alpha$-perfect.
In fact, we will show that $\tF(\mGG)$, regarded as a subcategory of $\DAE$,
is equivalent to the category $\DAE^\alpha$ of all $\alpha$-compact objects in 
$\DAE$, which is by definition $\alpha$-perfect.

The set $\mGG$ is by assumption (cf.~Notation~\ref{notation}), 
up to isomorphism,
closed under \linebreak 
(de-)sus\-pen\-sions, triangles, and $\alpha$-coproducts.
The same holds for the set $\tF(\mGG)$ in $\DAE$ because
$\tF$ is triangulated, fully faithful on $\mGG$ (Lemma~\ref{Schaf}),
and preserves $\alpha$-coproducts of objects
in $\mGG$ (Lemma~\ref{Alaaf}).
As a consequence, the full subcategory $\overline{\tF(\mGG)}$ of $\DAE$
consisting of all objects
isomorphic to some object in $\tF(\mGG)$ is an $\alpha$-localizing triangulated
subcategory of $\DAE$. This implies that the inclusion of $\tF(\mGG)$ in 
$\overline{\tF(\mGG)}$ gives an equivalence of categories
\begin{equation}\label{Bieraufhawai}
\tF(\mGG)\simeq\overline{\tF(\mGG)}
=\alpha \textrm{-loc}\langle\tF(\mGG)\rangle ,
\end{equation}
where $\alpha \textrm{-loc}\langle\tF(\mGG)\rangle$ is the smallest 
$\alpha$-localizing subcategory containing $\tF(\mGG)$. We can now
apply \cite[Lemma~4.4.5]{Neeman} to get an equality
$\alpha \textrm{-loc}\langle\tF(\mGG)\rangle =\DAE^\alpha$. Together
with (\ref{Bieraufhawai}) this implies that the inclusion 
$\tF(\mGG)\incl\DAE^\alpha$ 
is an equivalence of categories and in particular $\tF(\mGG)$ is an
$\alpha$-compact generating set for $\DAE$.
\end{proof}

\begin{lemma}\label{Hirte}
The functor 
$
\tF:\Ho\mKK\to\DAE
$
preserves $\alpha$-coproducts.
\end{lemma}

\begin{proof}
For a family $(A_i)_{i\in I}$ of
objects in $\Ho\mKK$ with $|I|<\alpha$ let
\[
\gamma : \coprod_{i\in I} \tF(A_i)\to\tF\left(\coprod_{i\in I}A_i\right)
\]
denote the canonical map. Since $\tF(\mGG)$ is a weak generating set
for $\DAE$ (see Lemma~\ref{genset}), 
it suffices to show that for each $G\in\mGG$
the induced map
\[
\gamma_\ast :\left[\tF(G),\coprod_{i\in I} \tF(A_i)\right]
\to\left[\tF(G),\tF\left(\coprod_{i\in I}A_i\right)\right]
\]
is bijective.

The surjectivity can be seen as follows. Any given 
morphism $\tF(G)\to\tF\left(\coprod_{i\in I}A_i\right)$
is by Lemma~\ref{Schaf} of the form $\tF(g)$ for some 
$g:G\to\coprod_{i\in I}A_i$. Since $\mGG$ is an $\alpha$-compact 
generating set for $\Ho\mKK$, this map $g$ factors as
\[
G\overset{h}{\to}\coprod_{i\in I}G_i\overset{\coprod f_i}{\to}
\coprod_{i\in I}A_i .
\]
Consider the following commutative diagram (of solid arrows).
\[
\xymatrix@M=.7pc{
&\tF(G)\ar[rr]^-=\ar@{.>}[dl]_-k\ar@{.>}[dd]|\hole ^(.3){f}
&&\tF(G)\ar[dl]_{\tF(h)}\ar[dd]^{\tF(g)}\\
\coprod_{i\in I} \tF(G_i)\ar[rr]^(.65){\iso}\ar[dr]_{\coprod \tF(f_i)} 
&&\tF\left(\coprod_{i\in I}G_i\right)\ar[dr]_{\tF\left(\coprod f_i\right)} \\
&\coprod_{i\in I} \tF(A_i)\ar[rr]_-\gamma
&&\tF\left(\coprod_{i\in I}A_i\right) 
}
\]

The horizontal arrow in the middle is an isomorphism since we
know already by Lemma~\ref{Alaaf} that $\tF$ preserves 
$\alpha$-coproducts of objects of $\mGG$.
Hence there exists a dotted arrow $k$ such that the whole diagram
commutes. If we define $f$ to be the composition 
$\bigl(\coprod \tF(f_i)\bigr) k$ then $\gamma_\ast (f)=\gamma f =\tF(g)$
and hence $\gamma_\ast$ is surjective.

To prove the injectivity of $\gamma_\ast$ consider a morphism
$f:\tF(G)\to\coprod_{i\in I}\tF(A_i)$ such that $\gamma f=0$. 
By Lemma~\ref{genset}, $\tF(\mGG)$ is in particular an $\alpha$-perfect
generating set for $\DAE$. This implies that $f$ can be factored as
\[
\tF(G)\overset{k}{\to}\coprod_{i\in I}\tF(G_i)
\xymatrix@M=.2pc{\ar[rr]^-{\coprod\limits \tF(f_i)}&&}
\coprod_{i\in I}\tF(A_i) ,
\]
where we used that $\tF$ is full for arrows with source in
$\mGG$ (Lemma~\ref{Schaf}). Composing $k$ with the isomorphism
$\coprod_{i\in I} \tF(G_i)\overset{\iso}{\to}\tF\left(\coprod_{i\in I}G_i\right)$
yields a map 
$\tF(G)\to\tF\left(\coprod_{i\in I}G_i\right)$ which is of the form $\tF(h)$
for some $h:G\to\coprod_{i\in I}G_i$.
We can conclude that $\tF\left(\coprod f_i\right)\tF(h)=\gamma f=0$ and, since
$\tF$ is faithful for morphisms with source in $\mGG$ (Lemma~\ref{Schaf}),
this implies $\left(\coprod f_i\right)h=0$. But the definition of an 
$\alpha$-perfect class allows us then to factor each $f_i:G_i\to A_i$ as
\[
G_i\overset{g_i}{\to}G'_i\overset{h_i}{\to}A_i
\]
with $G'_i\in\mGG$ such that the composition 
$G\overset{h}{\to}\coprod G_i\overset{\coprod g_i}{\to}\coprod G'_i$ 
already vanishes. We finally have a commutative diagram
\[
\xymatrix@M=.7pc{
&\tF(G)\ar[rr]^-=\ar[dl]_-k\ar'[d]'[dd]_f[ddd]
&&\tF(G)\ar[dl]_{\tF(h)}\ar[ddd]^0\\
\coprod_{i\in I} \tF(G_i)\ar[rr]^(.65){\iso}\ar[d]_{\coprod \tF(g_i)}
&&\tF\left(\coprod_{i\in I}G_i\right)\ar[d]_{\tF\left(\coprod g_i\right)}\\
\coprod_{i\in I} \tF(G'_i)\ar[rr]^(.65){\iso}\ar[dr]_{\coprod \tF(h_i)} 
&&\tF\left(\coprod_{i\in I}G'_i\right)\ar[dr]_{\tF\left(\coprod h_i\right)} \\
&\coprod_{i\in I} \tF(A_i)\ar[rr]_-\gamma
&&\tF\left(\coprod_{i\in I}A_i\right) 
}
\]
where the two horizontal arrows in the middle are isomorphisms
by Lemma~\ref{Alaaf}. We have just seen that the 
composition $\tF(\coprod g_i)\tF(h)$
vanishes. As a consequence, the composition 
$\bigl(\coprod \tF(g_i)\bigr)k$
also vanishes and so does $f$. This shows $\gamma_\ast$ is injective.
\end{proof}

\begin{prop}\label{Falbala}
The homological functor
$
\tH:\DAE\to\MODDalpha\mGG
$
preserves coproducts of objects which are in the essential image of 
$\tF:\Ho\mKK\to\DAE$.
\end{prop}

\begin{proof}
Since $r:\MODDalpha\mGG\to\MODD\mGG$ is fully faithful, it suffices to show
that for any set-indexed family $(X_i)_{i\in I}$ of objects in $\essim \tF$
the map
\[
r\left(\coprod_{i\in I}\tH (X_i)\right)\to r\tH\left(\coprod_{i\in I}X_i\right)
\]
is an isomorphism.
This map is the composition of the following isomorphisms, each of which
will be explained below.
\begin{eqnarray}
\label{1}r\left(\coprod_{i \in I}\tilde H_0(X_i)\right) 
& \iso & r \underset{ I '\subset  I, | I '| < \alpha }{\colim}
 \left(
\coprod_{i \in I '}\tilde H_0 (X_i)\right) \\ \label{2}
& \iso & \underset{ I '\subset  I, | I '| < \alpha }{\colim}
 r\left(
\coprod_{i \in I '}\tilde H_0 (X_i)\right) \\ \label{3}
& \iso & \underset{ I '\subset  I, | I '| < \alpha }{\colim}
 r \tilde H_0
\left(\coprod_{i \in I '} X_i\right) \\\label{4}
& = & \underset{ I '\subset  I, | I '| < \alpha }{\colim}
 H_0 R
\left(\coprod_{i \in I '} X_i\right) \\\label{5}
& \iso & H_0  R \left(\coprod_{i \in I } X_i\right) \\\label{6}
& = & r \tilde H_0 \left(\coprod_{i \in I } X_i\right)
\end{eqnarray}

Ad (\ref{1}). It is a general fact from category theory that
a coproduct can be expressed as such a colimit.

Ad (\ref{2}). The inclusion $r:\MODDalpha\mGG\to\MODD\mGG$ preserves
$\alpha$-filtered colimits \cite[Lemma~A.1.3]{Neeman} and the colimit
in question
is indeed $\alpha$-filtered in the sense of \cite[Definition~A.1.2]{Neeman}:
it is a colimit over the category $\mII$ with objects the subsets
$I'$ of $I$ which have cardinality (strictly) less than $\alpha$. Morphisms
are the inclusions between two such subsets. Let $\mJJ$ be a subcategory
of $\mII$
with less than $\alpha$ morphisms. In particular, $\mJJ$ has
less than $\alpha$ objects and we can conclude
\[
\left|\,\bigcup_{J\in\mJJ}J\,\right|\leq\left|\,\coprod_{J\in\mJJ}J\,\right|
=\sum_{J\in\mJJ}|J| < \alpha
\]
where, for the last inequality,  we used that $\alpha$ is regular
and hence any sum of less than $\alpha$ cardinals, all smaller than $\alpha$,
is itself smaller than $\alpha$.
Then $\bigcup_{J\in\mJJ}J$ is an object of $\mII$ admitting
an arrow 
\[
J'\to \bigcup_{J\in\mJJ}J\, , 
\]
for every $J'$ in $\mJJ$. This shows that $\mII$ is
an $\alpha$-filtered category.

Ad (\ref{3}). By Lemma~\ref{Hirte}, $\tF$ preserves $\alpha$-coproducts. Since
$\Phi\iso\tH\tF$ (Lemma~\ref{Hund}) and since 
$\Phi$ preserves arbitrary coproducts
\cite[Theorem~C]{Krause}, the functor $\tH$  
has to preserve $\alpha$-coproducts which are in the essential
image of $\tF$.

Ad (\ref{4}). Here only the the equality $H_0R=r\tH$ is used.

Ad (\ref{5}). Recall that colimits in $\MODD\mGG$ are formed objectwise.
Using the adjunction isomorphism of the pair $(L,R)$
and the isomorphism $LF\iso\tF$ (Lemma~\ref{Alaaf}), we obtain 
for $G\in\mGG$ a natural isomorphism
\begin{eqnarray*}
\underset{ I '\subset  I, | I '| < \alpha }{\colim}
 H_0 R
\left(\coprod_{i \in I '} X_i\right)(G)
& = & \underset{ I '\subset  I, | I '| < \alpha }{\colim}
\left[F(G),R\left(\coprod_{i\in I'}X_i\right)\right]^\DE \\
&\iso &\underset{ I '\subset  I, | I '| < \alpha }{\colim}
\left[\tF(G),\coprod_{i\in I'}X_i\right]^\DAE .
\end{eqnarray*}
By Lemma~\ref{genset}, $\tF(G)$ is in particular $\alpha$-small, that is, 
any map from $\tF(G)$ into a coproduct factors through some sub-coproduct
of less than $\alpha$ objects. Hence the canonical monomorphism
\[
\underset{ I '\subset  I, | I '| < \alpha }{\colim}
\left[\tF(G),\coprod_{i\in I'}X_i\right]^\DAE \to \left[\tF(G),\coprod_{i\in I}
X_i\right]^\DAE
\]
is an isomorphism.
Using again the adjunction isomorphism and $LF\iso\tF$ we get an isomorphism
$\left[\tF(G),\coprod_{i\in I} X_i\right]
\iso H_0R\left(\coprod_{i\in I} X_i\right)(G)$ and hence the isomorphism
(\ref{5}).

Ad (\ref{6}). This is again nothing but $H_0R=r\tH$.

\end{proof}

\begin{cor}\label{Sagnix}
The functor $\tF:\Ho\mKK\to\DAE$ preserves coproducts.
\end{cor}

\begin{proof}
Since $\tH$ reflects isomorphisms (Remark~\ref{Amnesix}),
it suffices to check that for any family $(A_i)_{i\in I}$ of objects in
$\Ho\mKK$ the map
\[
\tH\left(\coprod_{i\in I}\tF(A_i)\right)\to\tH\tF\left(\coprod_{i\in I}A_i\right)
\]
is an isomorphism. But this follows from Proposition~\ref{Falbala} because 
$\tH\tF$ is isomorphic to~$\Phi$ (Lemma~\ref{Hund}) 
and $\Phi:\Ho\mKK\to\MODDalpha\mGG$
preserves coproducts by \cite[Theorem~C]{Krause}.
\end{proof}

\begin{prop}\label{eq}
The adjoint functors 
\[
\xymatrix@M=.7pc{\Ho\mKK\ar@<-.5ex>[r]_-{\tF}
&\DAE.\ar@<-.5ex>[l]_-{\tJ} }
\]
are inverse equivalences of triangulated categories.
\end{prop}

\begin{proof}
Consider for $X\in\DAE$ and $A\in\Ho\mKK$ the unit 
$\eta_X:X \to\tF\tJ X$ and the counit $\varepsilon_A:\tJ\tF(A)\to A$
of the adjunction. It suffices to show that both are isomorphisms
for all objects $A\in\Ho\mKK$ and $X\in\DAE$.
The counit of the restricted adjunction $(\tJ,\tF)$ is the same
as the counit of the adjunction $(J,F)$, which is an isomorphism
for $A\in\mGG$ by Lemma~\ref{summary}. Since both $\tJ$ and $\tF$
preserve coproducts ($\tJ$ as a left adjoint and $\tF$ by 
Corollary~\ref{Sagnix}) and since $\mGG$ generates $\Ho\mKK$,
we can apply Lemma~\ref{useful} and conclude that the counit is an
isomorphisms for all objects.

If we apply the inclusion functor $R$ to the unit of an object
in $\tF(\mGG)$ we get a map
\[
F(G)=R\tF(G)\to R\tF\tJ\tF(G)=FJF(G),
\]
which is (up to isomorphism) the unit of the free module $F_G$ with
respect to the adjoint pair $(J,F)$ and hence an isomorphism by
Lemma~\ref{summary}.
Again, since both $\tJ$ and $\tF$ preserve coproducts and 
$\tF(\mGG)$ is a generating set for $\DAE$ by Lemma~\ref{genset}
we can conclude that the unit of the adjoint pair $(\tJ,\tF)$ is
an isomorphism for all objects.
\end{proof}

Note that, in particular, $F=R\tF$ is fully faithful.
Now we have all what we need to finish the proof of the characterization
theorem.
%it is not hard to prove the characterization theorem, which we state again
%here:

%\begin{thm}\label{char}
%Let $\mTT$ be a topological triangulated category. Then the following are
%equivalent.
%\begin{enumerate}
%\item[\textnormal{(i)}]$\mTT$ is well generated.
%\item[\textnormal{(ii)}]$\mTT$ is triangulated equivalent to a localization
%of the derived category of a spectral category where the acyclics are
%generated by a set.
%\end{enumerate}
%\end{thm}

\begin{proof}[Proof of Theorem~\ref{characterization}, \rm{(i)} 
$\Rightarrow$ \rm{(ii)}]
%Recall that the derived category of a spectral category is well generated --
%even compactly generated 
%\cite[Theorem~A.1.1]{SS03} -- and by Proposition~\ref{Neeman's} 
%localizations of such are well generated again if the acyclics 
%are generated by a set. This shows that (ii) implies~(i).
Let $\mTT$ be a topological triangulated 
category, i.e., $\mTT$ is triangulated
equivalent to the homotopy category of some spectral model category $\mKK$.
If $\mTT$ is well generated then so is $\Ho\mKK$ and we can choose
a regular cardinal $\alpha$ and a generating set $\mGG$ as in 
Notation~\ref{notation} and let $\mEE$ (as before)
be the full spectral subcategory of $\mKK$ with objects $\mGG$.
By Proposition~\ref{eq}, $\Ho\mKK$ is equivalent to $\DAE$, which is
by Proposition~\ref{locprop} 
a localization of $\DE$ with acyclics being generated
by a set. 

\end{proof}

\begin{rem}
The proof of Theorem~\ref{characterization} can also be read as a proof for 
the characterization of the topological \emph{compactly} generated categories
as the derived categories of spectral categories. Namely, compactly generated
means we can choose $\alpha=\aleph_0$ and in this case $\DAE=\DE$, see
Lemma~\ref{aleph_0}.
\end{rem}

\section{A lift to the model category level}

\subsection{Bousfield localizations, properness, and cellularity}

Let us recall some notions from Hirschhorn's book \cite{Hirschhorn}.
If $\mMM$ is a model category and $\mCC$ is
a class of morphisms in $\mMM$, then an object $W$ of $\mMM$ is called 
$\mCC$-\emph{local} if it is fibrant and for every element $f:A\to B$
of $\mCC$ the induced map $f^\ast:\map(B,W)\to\map(A,W)$
is a weak equivalence of simplicial sets \cite[Definition~3.1.4]{Hirschhorn}. 
Here $\map(X,Y)$ denotes a homotopy
function complex between $X$ and $Y$, which is a simplicial set that can in
general be 
obtained by (co-)simplicial resolutions \cite[Section~17.4]{Hirschhorn}.
Given a cofibrant object $X$ and a fibrant object $Y$ in a 
simplicial model category, $\map(X,Y)$ can be chosen
to be the (fibrant) simplicial set given by the enrichment 
\cite[Example~17.1.4]{Hirschhorn}.
Note that model categories of modules over a spectral category are spectral
and thus simplicial \cite[Theorem~A.1.1 and Lemma~3.5.2]{SS03}).
A $\mCC$-\emph{local equivalence} is a map $g:X\to Y$ such that
for every $\mCC$-local object $W$ of $\mMM$
the induced map $g^\ast:\map(Y,W)\to\map(X,W)$ is a weak equivalence
of simplicial sets \cite[Definition~3.1.4]{Hirschhorn}.
Every weak equivalence in $\mMM$ is in particular a $\mCC$-local equivalence
\cite[Proposition~3.1.5]{Hirschhorn}.
For us, a Bousfield localization is what Hirschhorn calls a \emph{left}
Bousfield localization \cite[Definition~3.3.1]{Hirschhorn}:

\begin{defn}\label{Hirloc}
The \emph{Bousfield localization} of a model category $\mMM$ with respect
to a class of maps $\mCC$ is a model category $\mL_\mCC\mMM$ which has the
same underlying category as $\mMM$ such that 
\begin{itemize}
\item the weak equivalences in $\mL_\mCC\mMM$ are the $\mCC$-local equivalences,
\item the cofibrations in $\mL_\mCC\mMM$ are the cofibrations in $\mMM$,
\item and the fibrations in $\mL_\mCC\mMM$ 
are the maps which have the right lifting
property with respect to those maps which are both cofibrations and
weak equivalences in $\mL_\mCC\mMM$.
\end{itemize}
%We refer to this model structure as the $\mL_\mCC$-model structure.
\end{defn}

\begin{rem}\label{locs-are-locs}
Here are some basic properties for a Bousfield localization $\mL_\mCC\mMM$ 
of $\mMM$. 
\begin{enumerate}
\item
Since every weak equivalence in $\mMM$ is a $\mCC$-local equivalence
and the cofibrations in $\mMM$ and $\mL_\mCC\mMM$ are the same,
every fibration in $\mL_\mCC\mMM$ is in particular a fibration in~$\mMM$.
\item
There is a Quillen pair
\[
\xymatrix@M=.7pc{\mL_\mCC\mMM\ar@<-.5ex>[r]_-{Q}&\mMM
\ar@<-.5ex>[l]_-{P}}
\]
where $P$ and $Q$ are the identity functors on underlying categories.
The functor $P$ has the following universal property.
The left derived $P^L:\Ho\mMM\to\Ho\mL_\mCC\mMM$ of $P$ maps
the images in $\Ho\mMM$ of the elements in $\mCC$ to isomorphisms in
$\Ho \mL_\mCC\mMM$ and $P$ is universal with this property, i.e.,
if $F:\mMM\to\mNN$ is an arbitrary left Quillen functor
whose left derived functor takes the images in $\Ho\mMM$ of the 
elements in $\mCC$ to isomorphisms in
$\Ho \mL_\mCC\mNN$, then there exists a unique left Quillen
functor $\bar{F}:\mL_\mCC\mMM\to\mNN$ such that $\bar{F} P=F$ 
\cite[Theorem~3.3.19 and Definition~3.1.1]{Hirschhorn}.
\item The right derived $Q^R:\Ho\mL_\mCC\mMM\to\Ho\mMM$ is fully faithful.
In particular, if $\mMM$ and $\mL_\mCC\mMM$ are stable, $\Ho\mL_\mCC\mMM$
is a localization of $\Ho\mMM$ in the sense of Definition~\ref{loc}.
This can be seen as follows. Let $X$ and $Y$ be $\mL_\mCC$-fibrant objects,
so that $Q^R(X)\iso X$ and $Q^R(Y)\iso Y$. We can further assume that $X$
is cofibrant. Then the right derived $Q^R$ is on morphisms the
map induced by the identity
\[
\mL_\mCC\mMM(X,Y)/\sim\to\mMM(X,Y)/\sim
\]
where $\sim$ denotes the equivalence relation given by (left) homotopy.
This map is clearly surjective. Assume we have $f\sim g$ in $\mMM$
via a cylinder object
\[
\xymatrix@!0@=15mm@M=.7pc{X\amalg X\ar[rr]\ar@{>->}[dr]&&X\\&X'\ar[ur]_-{\sim}&}
\]
in $\mMM$. But this is also a cylinder object in $\mL_\mCC\mMM$ and
thus $f\sim g$ in $\mL_\mCC\mMM$.

\end{enumerate}
\end{rem}

%\begin{lemma}
%Bousfield localizations of stable model categories induce localizations
%of homotopy categories.
%\end{lemma}
%
%\begin{proof}
%Let $\mL_\mCC\mMM$ be a Bousfield localization together with a Quillen pair
%We have first to show that $\mL_\mCC\mMM$ is stable.
%\end{proof}

We want to apply Hirschhorn's existence theorem for
Bousfield localizations \cite[Theorem~4.1.1]{Hirschhorn}. It states
that for any left proper cellular model category $\mMM$ and any
\emph{set} $\mCC$ of morphisms, the Bousfield localization
$\mL_\mCC\mMM$ exists.
Recall that a model category is called \emph{left proper} if pushouts along
cofibrations preserve weak equivalences \cite[Definition~13.1.1]{Hirschhorn}.

\begin{defn}\label{pc}
A spectral category $\mRR$ is called \emph{pointwise cofibrant} if
the symmetric spectrum $\mRR(R,R')$ is cofibrant for all $R$, $R'$ in $\mRR$.
\end{defn}

\begin{lemma}\label{leftproper}
If $\mRR$ is a spectral category which is pointwise cofibrant  
(Definition~\textnormal{\ref{pc}}), then $\MODD\mRR$
is left proper.
\end{lemma}

\begin{proof}
Since weak equivalences in $\MODD\mRR$ are defined objectwise
and the stable model category $\Sp$ of symmetric spectra is left proper
\cite[Theorem~5.5.2]{HSS}, this is
an immediate consequence of Corollary~\ref{objectwise}.
\end{proof}

A \emph{cellular} model category is a certain kind
of cofibrantly generated model category
that allows sets $I$, resp.~$J$, 
of generating cofibrations, resp.~trivial cofibrations, such that
$I$ and $J$ satisfy some finiteness conditions on the domains and codomains
of their elements.
The precise definition of cellular is technical and we will not give it here;
the only place we will use the very definition is the proof of 
Proposition~\ref{cellular}.
The reader is referred to Hirschhorn's 
book \cite[Definition~12.1.1]{Hirschhorn}.
(To be precise, we need a slightly stronger version of the definition
from the book, namely the one Hirschhorn has given in earlier drafts; see
the proof of Proposition~\ref{cellular}.)
Examples of cellular model categories are (pointed or unpointed) 
simplicial sets, topological spaces,
and spectra, as the following lemma states.

\begin{lemma}\label{Honigkuchen}
The stable model category $\Sp$ of symmetric spectra is cellular.
\end{lemma}

\begin{proof}
Hovey has shown that symmetric spectra with the projective model structure
form a cellular model category \cite[Theorem~A.9]{Hovey-Spectra} where
the set of generating cofibrations can be chosen as 
\[
I=\{\mF_{\!n}(\partial\Delta[r]_+)\to \mF_{\!n}(\Delta[r]_+)\,\,|\,\,r,n\geq 0 \}.
\]
Here, $\partial\Delta[r]_+\to \Delta[r]_+$ denotes the inclusion
of the boundary of the simplicial $r$-simplex into the simplicial $r$-simplex
(plus additional basepoints, respectively); these simplicial maps
from a set of generating cofibrations for the model category of 
pointed simplicial sets. The functor $\mF_{\!n}$ from pointed
simplicial sets to symmetric spectra is left adjoint to the $n$-th
evaluation functor.
The category of symmetric spectra with the stable model structure can be 
obtained as a Bousfield localization of the projective model structure
\cite[Section~8]{Hovey-Spectra}. Hence, by \cite[Theorem~4.1.1]{Hirschhorn},
the stable model category of symmetric spectra is also cellular and
the set of generating cofibrations can still be chosen to be $I$.
\end{proof}

\begin{prop}\label{cellular}
If $\mRR$ is a spectral category which is pointwise cofibrant
(Definition~\textnormal{\ref{pc}}), then $\MODD\mRR$
is cellular.
\end{prop}

A one object version of this proposition occurs in a paper by Hovey 
\cite[Proposition~2.9]{Hovey-Monoidal} -- but he does not give the
proof there because the definition of cellular is somewhat technical.
We will give the proof using ideas of an unpublished 
pre-version of Hirschhorn's
book, where enriched diagram categories have been studied.

\begin{proof}[Proof of Proposition~\ref{cellular}]
Let $I$ be a set of generating cofibrations and $J$ a set
of generating trivial cofibrations which make the category of symmetric
spectra into a cellular model category. Recall that $I$ can be chosen
as
\[
I=\{\mF_{\!n}(\partial\Delta[r]_+)\to \mF_{\!n}(\Delta[r]_+)\,\,|\,\,r,n\geq 0 \}.
\]
By \cite[Theorem~A.1.1]{SS03}, $\MODD\mRR$ is
a cofibrantly generated model category. A set of generating cofibrations
is given by 
\[
\tI=\bigcup_{R\in\mRR} I\sm F_R
\]
where $F_R=\mRR(-,R)$ is the free module with
respect to $R$ and $I\sm F_R$ consists of all maps of the form
$f\sm F_R:A\sm F_R\to B\sm F_R$ for some $f:A \to B$ in $I$. 
Similarly, a set of generating trivial cofibrations is
given by
\[
\tJ=\bigcup_{R\in\mRR} J\sm F_R .
\]
We have to show that $\tI$ and $\tJ$ satisfy the following three conditions
\cite[Definition~12.1.1]{Hirschhorn}.

\begin{enumerate}[(i)]
\item For every element $A\to B$ of $I$ and every object $R$ of $\mRR$,
the modules $A\sm F_R$ and $B\sm F_R$ are compact relative to $\tI$ in the
sense of \cite[Definition~10.8.1]{Hirschhorn}.
\item For every element $A\to B$ of $J$ and every object $R$ of $\mRR$,
the module $A\sm F_R$ is cofibrant and small relative to $\tI$ in the sense
of \cite[Definition~10.5.12]{Hirschhorn}.
\item The cofibrations in $\MODD\mRR$ are effective monomorphisms 
\cite[Definition~10.9.1]{Hirschhorn}.
\end{enumerate}
Note that condition (ii) is stronger than in 
\cite[Definition~12.1.1]{Hirschhorn}
(where the domains of the generating trivial
cofibrations are not required to be cofibrant). It occurs in this
stronger form in earlier drafts of Hirschhorn's book,
and as Hirschhorn pointed out to me, the earlier definition of cellular
is the right one for the existence theorem of
Bousfield localizations.

Ad (i). Fix an element $A\to B$ of $I$.
Let $X\to Y$ be a relative $\tI$-cell complex 
\cite[Definition~10.5.8]{Hirschhorn}. This means,
$X\to Y$ is the composition of a $\lambda$-sequence 
\[
X=X_0\to X_1\to X_2\to\cdots\to X_\beta\to\cdots\quad (\beta<\lambda)
\]
(for some ordinal $\lambda$) such that $X_{\beta +1}$ is obtained
from $X_\beta$ by attaching a set of $\tI$-cells, i.e., there are
pushout diagrams
\[
\xymatrix@M=.7pc{\coprod_i A_i\sm F_{R_i}\ar[r]\ar[d]
&\coprod_i B_i\sm F_{R_i}\ar[d]\\
X_\beta\ar[r]&X_{\beta +1}
}
\]
for elements $A_i\to B_i$ of $I$ and objects $R_i$ of $\mRR$. Such a relative 
$\tI$-cell complex is
\emph{presented} if
a particular $\lambda$-sequence and certain gluing maps 
for the pushout diagrams are specified so that one can consider 
subcomplexes of $X\to Y$, see \cite[Section~10.6]{Hirschhorn} for the
details.

We have to show that there exists a cardinal $\gamma$ such that for every
presented $\tI$-cell complex $X\to Y$, every
map $A\sm F_R\to Y$ factors through a subcomplex $X\to Y'$ of size 
at most $\gamma$ (i.e., the set of cells has cardinality at most $\gamma$).
Via the (Quillen) adjunction
\begin{equation}\label{adj}
\xymatrix@M=.7pc{\MODD\mRR\ar@<-.5ex>[rr]_-{\ev_R}
&&\MODD \mS = \Sp \ar@<-.5ex>[ll]_-{-\sm F_R}},
\end{equation}
a map $A\sm F_R\to Y$ corresponds to a map $A\to Y(R)$. Since
colimits are preserved by the evaluation functor (Corollary~\ref{objectwise}),
$X(R)\to Y(R)$
is the composition of the presented $\lambda$-sequence
\[
X(R)=X_0(R)\to X_1(R)\to X_2(R)\to\cdots\to X_\beta(R)\to\cdots
\]
and $X_{\beta +1}(R)$ is obtained
from $X_\beta(R)$ by attaching $I'$-cells, where
\[
I'=\bigcup_{R'\in\mRR}I\sm F_{R'}(R).
\]
But this means that $X(R)\to Y(R)$ is a presented $I'$-cell complex.
If we can show that $A$ is compact relative to $I'$, we can conclude
that there exists a cardinal $\gamma$ (which does not
depend on $X\to Y$ and $A\sm F_R\to Y$) such that $A\to Y(R)$ factors 
through a subcomplex of size at most $\gamma$. 
Via adjunction, this corresponds to a factorization
of $A\sm F_R\to Y$ through a subcomplex of the same size.

It remains to prove that $A$ is compact relative to $I'$ in $\Sp$. 
For that purpose we want to apply \cite[Proposition~11.4.9]{Hirschhorn}.
A stable cofibration of symmetric spectra is in particular a level cofibration
\cite[Corollary~5.1.5]{HSS} and hence a monomorphism.
Since $F_{R'}(R)=\mRR(R,R')$ was assumed to be cofibrant, all
elements in $I'$ are cofibrations in $\Sp$. The domains of the elements
of $I'$, which are of the form $\mF_{\!n}(\partial\Delta[r]_+)\sm F_{R'}(R)$
\cite[Definition~3.3.2]{HSS}, are cofibrant since $F_{R'}(R)$ is cofibrant
by assumption and every spectrum of the form $\mF_{\!n}(K)$ for a pointed
simplicial set
$K$ is cofibrant \cite[Proposition~3.4.2]{HSS}. By 
\cite[Corollary~12.3.4]{Hirschhorn}, 
it follows that the domains of the elements
of $I'$ are compact (relative to $I$).
Since $\Sp$ is cellular (Lemma~\ref{Honigkuchen}) and $A$ is
the domain of an element of $I$, the object $A$ is compact (relative
to $I$). Now we have all we need to apply 
\cite[Proposition~11.4.9]{Hirschhorn}, which tells us that $A$ is also compact
relative to $I'$.
Hence we have shown that $A\sm F_R$ is compact relative to $\tI$. The same
proof shows that $B\sm F_R$ is compact relative to $\tI$.

Ad (ii). Let $A\to B$ be an element of $J$ and $R$ an object of $\mRR$.
Since $\Sp$ together with $J$ as a set of 
generating trivial cofibrations is cellular, 
$A$ is cofibrant. Applying the left
Quillen functor $-\sm F_R:\Sp\to\MODD\mRR$ shows that $A\sm F_R$ is cofibrant.
For the smallness, we 
need to show that there exists a cardinal $\kappa$ such that for all
$\lambda\geq\kappa$ and all $\lambda$-sequences
\[
X_0\to X_1\to X_2\to\cdots\to X_\beta\to\cdots
\]
in $\MODD\mRR$ where the maps $X_\beta\to X_{\beta +1}$ are relative
$\tI$-cell complexes, the map $\colim_{\beta<\lambda}
\MODD\mRR(A\sm F_R, X_\beta)\to\MODD\mRR(A\sm F_R,\colim_{\beta<\lambda}X_\beta)$
is an isomorphism of sets \cite[Definition~10.4.1]{Hirschhorn}.
Using again the adjunction (\ref{adj}) and the fact that $\ev_R$ preserves
colimits (Corollary~\ref{objectwise}), we can conclude
that it suffices to show $A$ is small relative to relative $I'$-cell
complexes in $\Sp$. Since $\Sp$ is cellular, $A$ (which is
the domain of an element of $J$) is small relative to $I$ and thus, by
\cite[Proposition~11.2.3]{Hirschhorn}, relative
to the class of all cofibrations. But the elements of $I'$ are cofibrations
because $F_{R'}(R)=\mRR(R,R')$ is cofibrant by assumption. Hence all
relative $I'$-cell complexes are cofibrations. This shows that $A$ is
small relative to relative $I'$-cell complexes.

Ad (iii). We have to show that every cofibration $f:A\to B$
in $\MODD\mRR$ is the equalizer of the canonical pair of maps
$B\rightrightarrows B\amalg_AB$.
We can choose an equalizer $E$ of this pair and get an induced map $g$.
\[
\xymatrix@M=.7pc{E\ar[r]&B\ar@<.5ex>[r]\ar@<-.5ex>[r] & B\amalg_AB\\
A\ar[ur]_f\ar@{.>}[u]^g}.
\]
By Corollary~\ref{objectwise}, $\ev_R(f)$ is a cofibration in $\Sp$
and thus an effective monomorphism by Lemma~\ref{Honigkuchen}.
Since for every $R$ in $\mRR$ the map $\ev_R$ 
preserves limits, $\ev_R(g)$ is
an isomorphism for each $R$. Hence $g$ is an isomorphism.
\end{proof}
 
\subsection{Well generated stable model categories}
By a well generated stable model category we mean a stable model category
whose homotopy category is well generated as a triangulated category 
(Definition~\ref{alphacompact}). One implication 
of Theorem~\ref{characterization} says that a topological
well generated triangulated category is equivalent to a localization of the
derived category of a spectral category. Theorem~\ref{lift} lifts
this result to the level of model categories. Roughly speaking,
well generated spectral model categories are localizations of categories
of modules.
Recall from Examples~\ref{spectral} that there are, up to Quillen equivalence,
rather large classes of spectral model categories.

\begin{lemma}\label{local-objects}
Let $\mKK$ be a left proper spectral model category and 
$\mCC$ a set of maps in $\mKK$
such that the localization
$\mL_\mCC\mKK$ exists. Assume the domains and codomains of the maps
in $\mCC$ are cofibrant and the image $\mWW$ of $\mCC$ 
in $\Ho\mKK$ is, up to isomorphism,
 closed under (de-)sus\-pen\-sions.

Then the following are equivalent for an object $X$ in $\mKK$.
\begin{enumerate}
\item[\textnormal{(i)}] $X$ is $\mL_\mCC$-fibrant (i.e., fibrant 
in $\mL_\mCC\mKK$).
\item[\textnormal{(ii)}] $X$ is $\mCC$-local.
\item[\textnormal{(iii)}] $X$ is fibrant (in the original model structure)
and, considered as an object in $\Ho\mKK$, 
it lies in $\mWW\loc$ (cf.~Lemma~\textnormal{\ref{Methusalix}}).
\end{enumerate}
\end{lemma}

\begin{proof}
(i) $\Leftrightarrow$ (ii) This is  
\cite[Propostion~3.4.1]{Hirschhorn}.

(ii) $\Leftrightarrow$ (iii)
By definition, $X$ is $\mCC$-local if and only if
it is fibrant (in the original model structure) and for every
element $f:A\to B$ of $\mCC$ the induced map $f^\ast:\map(B,X)\to\map(A,X)$
is a weak equivalence of simplicial sets. Now the simplicial
enrichment of $\mKK$ is induced by the spectral enrichment, 
that is, $f^\ast:\map(B,X)\to\map(A,X)$ is the level zero map 
of a map of spectra, which are fibrant (i.e., $\Omega$-spectra)
since $A$ and $B$ are cofibrant
and $X$ is fibrant. 
Hence $f^\ast$ is a map of loop spaces and thus
a weak equivalence if and only if for the distinguished basepoint,
$\pi_n(f^\ast):\pi_n\map(B,X)\to\pi_n\map(A,X)$
is an isomorphism for all $n\geq 0$.
Since $A$ and $B$ are cofibrant and $X$ is fibrant, the map
$\pi_n(f^\ast)$ is naturally
isomorphic to the map
$f^\ast:[\Sigma^nB,X]\to [\Sigma^nA,X]$ 
of morphism groups in $\Ho\mKK$.
Since $\mWW$ is closed under (de-)sus\-pen\-sions, the $\mL_\mCC$-fibrant
objects in $\mKK$ are exactly the fibrant objects $X$
such that the induced map $f^\ast:[B,X]\to[A,X]$ 
is an isomorphism for all $f:A\to B$
in $\mWW$. In other words, the $\mL_\mCC$-fibrant objects are precisely
the fibrant objects which, considered
as objects in $\Ho\mKK$, lie in $\mWW\loc$.
\end{proof}

\begin{cor}\label{Oxford}
If in the situation of Lemma~\textnormal{\ref{local-objects}}
\begin{align}\label{Stern}
\xymatrix@M=.7pc{\Ho\mL_\mCC\mKK\ar@<-.5ex>[r]_-{Q^R}&\Ho\mKK
\ar@<-.5ex>[l]_-{P^L}}\tag{$\ast$}
\end{align}
denotes the derived adjunction, then $\essim Q^R=\mWW\loc$, and the 
model category $\mL_\mCC\mKK$
is stable, such that $\Ho\mL_\mCC\mKK$ 
is a localization of $\Ho\mKK$ in the sense of
Definition~\textnormal{\ref{loc}} via \textnormal{(\ref{Stern})}.
\end{cor}

\begin{proof}
The essential image
of $Q^R$ contains all objects which are, up to isomorphism in $\Ho\mKK$,
$\mL_\mCC$-fibrant. By Lemma~\ref{local-objects}, the $\mL_\mCC$-fibrant
objects are the fibrant objects which are in $\mWW\loc$.
Since both $\essim Q^R$ and $\mWW\loc$ are replete subcategories of $\Ho\mKK$
(i.e., they contain
the whole isomorphism class of any of their objects), it follows that they
coincide.

In particular, $\essim Q^R$ is a (colocalizing) triangulated subcategory
of $\Ho\mKK$ and hence closed under (de-)suspensions. This implies
that $\Ho\mL_\mCC\mKK$ (which is isomorphic to $\essim Q^R$ since
$Q^R$ is fully faithful, see Remark~\ref{locs-are-locs}(3)) is also stable.
The details are as follows. Let $\Sigma$, resp.~$\Omega$, denote the 
suspension, resp.~desuspension, in $\Ho\mKK$, and $\Sigma'$, resp.~$\Omega'$,
the suspension, resp.~desuspension, in $\Ho\mL_\mCC\mMM$. We have to show
that $\Sigma'$ and $\Omega'$ are inverse equivalences. As a left derived,
$P^L$ commutes with suspension, while $Q^R$ commutes with desuspension
\cite[Proposition~6.4.1]{Hovey}. Let us first check that $Q^R$ also commutes
with suspension. Since $\essim Q^R$ is a triangulated subcategory
of $\Ho\mKK$, there is an isomorphism  $\Sigma Q^R X\iso Q^RY$ for some 
$Y$ in $\Ho\mL_\mCC\mKK$. Using that the counit of the adjunction
$(P^L, Q^R)$ is an isomorphism ($Q^R$ is fully faithful) we get
an induced isomorphism
\[
Y\iso P^LQ^RY\iso P^L\Sigma Q^R X\iso\Sigma'P^LQ^RX\iso \Sigma'X
\]
so that $\Sigma Q^RX$ is naturally isomorphic to $Q^R\Sigma'X$.
Now we have isomorphisms
\[
\Sigma'\Omega'\iso P^LQ^R\Sigma'\Omega'\iso P^L\Sigma Q^R \Omega'\iso
P^L\Sigma\Omega Q^R\iso P^L Q^R\iso\id
\]
and 
\[
\Omega' \Sigma'\iso P^LQ^R\Omega' \Sigma'\iso P^L\Omega Q^R\Sigma'\iso
P^L\Omega \Sigma Q^R\iso P^LQ^R\iso \id
\]
which show that $\Omega'$ and $\Sigma'$ are inverse equivalences.
\end{proof}

\begin{lemma}\label{letztes-Lemma}
Let $\mKK$ be a left proper spectral model category and $\mCC$ a set
of maps in $\mKK$, closed under (de-)sus\-pen\-sions in $\Ho\mKK$
(up to isomorphism).
If the localization $\mL_\mCC\mKK$ exists, it is a left proper spectral
model category.
\end{lemma}

\begin{proof}
Bousfield localizations of left proper model categories are always left proper
\cite[Proposition~3.4.4]{Hirschhorn}.
For the proof of the spectral part, we 
can without loss of generality assume that the domains and codomains
of the maps in $\mCC$ are cofibrant (otherwise we could cofibrantly
replace them, this would, up to isomorphism of model categories, have no effect
on $\mL_\mCC\mKK$.)
Recall that the underlying categories of $\mL_\mCC\mKK$ and $\mKK$ are
the same.
Since $\mKK$ is spectral, it has a tensor, a cotensor, and an enriched
Hom-functor. We show that the same functors make $\mL_\mCC\mKK$ into
a spectral category. It suffices to verify the pushout product axiom
\cite[Definition~4.2.1]{Hovey}.
Let $f:A\cof B$ be a cofibration in $\Sp$ and $g:X\cof Y$ a
cofibration in $\mL_\mCC\mKK$.  Since $\mKK$ and $\mL_\mCC\mKK$ have
the same cofibrations and since the pushout product axiom holds
for $\mKK$, the map 
\[
f\,\Box\, g:(A\sm Y)\amalg_{A\sm X}(B\sm X)\to B\sm Y
\]
is a cofibration, which is trivial if $f$ is.
Now assume the cofibration $g$ is trivial in $\mL_\mCC\mKK$.
We have to show that $f\,\Box\, g$ is also trivial.
By Corollary~\ref{Oxford}, $\mL_\mCC\mKK$ is stable. So using
the general fact \cite[Theorem~1.2.10(iv)]{Hovey} that a map is a 
weak equivalence if and only if its image in the homotopy category is an
isomorphism, it suffices
to show that the cofiber of $f\,\Box\, g$ is trivial in $\Ho\mL_\mCC\mKK$.
Let $C$ be the cofiber of $f$ and $Z$ the cofiber of $g$. As cofibers
of cofibrations both $C$ and $Z$ are cofibrant; $Z$ is also trivial 
(in $\Ho\mL_\mCC\mKK$) because $g$ was assumed to be trivial. 
The cofiber
of $f\,\Box\, g$ is just $C\sm Z$. By what we have already shown, 
the left adjoint $-\sm Z:\Sp\to\mL_\mCC\mKK$
preserves (trivial) cofibrations, hence it is a left Quillen
functor and the kernel of its left derived $-\sm^LZ:\Ho\Sp\to\Ho\mL_\mCC\mKK$
is a localizing triangulated subcategory of $\Ho\Sp$.
It contains the sphere spectrum $\mS$ because $Z$ is trivial in
$\Ho\mL_\mCC\mKK$. Since the sphere spectrum is a generator for $\Ho\Sp$, 
the kernel of $-\sm^LZ$ is the whole stable homotopy category $\Ho\Sp$.
In particular, since $C$ is cofibrant, we have isomorphisms
$C\sm Z\iso C\sm^LZ\iso 0$ in $\Ho\mL_\mCC\mKK$. This shows that
the cofiber of $f\,\Box\, g$ is trivial and finishes the proof.
\end{proof}

\begin{thm}\label{lift}
Every well generated spectral model category admits a right Quillen
equivalence to a Bousfield localization of the model category
of modules over some spectral category.
\end{thm}

\begin{proof}
Let $\mKK$ be a well generated spectral model category, we can thus
assume we are in the situation of Notation~\ref{notation}, where we
fixed a sufficiently nice generating set $\mGG$ for $\Ho\mKK$ and
we defined $\mEE$ to be the full spectral subcategory of $\mKK$ having
$\mGG$ as set of objects. Recall
that there is a spectral Quillen adjunction
\cite[Section~3.9]{SS03} 
\[
\xymatrix@M=.7pc{\mKK\ar@<-.5ex>[rr]_-{\Hom(\mGG,-)}
&&\MODD \mEE ,\ar@<-.5ex>[ll]_-{-\sm_{\mEE}\mGG}}
\]
which we already considered in Section~4.1.
The spectral category $\mEE$ will not be pointwise cofibrant  
(Definition~\ref{pc}) in general. But by 
\cite[Proposition~6.3]{SS03'},
the spectral categories with a fixed set of objects form a model category
with pointwise weak equivalences and pointwise fibrations and 
where cofibrant
objects are in particular pointwise cofibrant. So we can chose a
cofibrant replacement $\mEE^{\cofr}$ of $\mEE$ with the same set of 
objects $\mGG$, such that $\mEE^{\cofr}$ is pointwise cofibrant.
By \cite[Theorem~7.2]{SS03'}, the corresponding module categories 
are related by a Quillen equivalence
(given by extension and restriction of scalars), which we denote by
\[
\xymatrix@M=.7pc{\MODD\mEE\ar@<-.5ex>[r]_-{U}
&\MODD\mEE^{\cofr} .\ar@<-.5ex>[l]_-{V}}
\]
We define a set $\mCC$ 
of maps in $\MODD\mEE^{\cofr}$ by modifying the set $\mWW$ of the maps
\[
\coprod_{i\in I}F(G_i) \to F\Bigl(\coprod_{i\in I}G_i \Bigr)
\]
in $\DE$, which we used in Proposition~\ref{locprop} to get the localization
$\DAE$ of $\DE$. We let $\mCC$ be the set of maps
\[
\Bigl(U\Bigl(\Bigl(\coprod_{i\in I}\Hom(\mGG,G_i)^{\cofr}\Bigr)^{\fibr}\Bigr)
\Bigr)^{\cofr}
\to \Bigl(U\Bigl( \Hom\Bigl(\mGG,\coprod_{i\in I}G_i \Bigr)^{\fibr}\Bigr)
\Bigr)^{\cofr}
\]
where $(G_i)_{i\in I}$ runs through all families in $\mGG$ with
$|I|<\alpha$ and we allow one and 
only one set $I$
for each cardinality smaller than $\alpha$. The decorations `$\cofr$'
and `$\fibr$' denote cofibrant and fibrant replacements. 
Let $\mWW'$ denote the image in $\mD((\mEE^{\cofr})^{\op})$
of the maps in $\mCC$. Then $\mWW\subset\DE$ and 
$\mWW'\subset\mD((\mEE^{\cofr})^{\op})$ correspond via the inverse
triangulated equivalences $V^L$ and $U^R$.

By Lemma~\ref{leftproper} and Proposition~\ref{cellular}, $\MODD\mEE^{\cofr}$
is left proper cellular and we can apply Hirschhorn's existence theorem
for Bousfield localizations \cite[Theorem~4.1.1]{Hirschhorn}, so that
we obtain a Bousfield localization 
$\mL_\mCC\MODD\mEE^{\cofr}$ (Definition~\ref{Hirloc}) together with a
Quillen functor pair
\[
\xymatrix@M=.7pc{\mL_\mCC\MODD\mEE^{\cofr}\ar@<-.5ex>[r]_-{Q}&\MODD\mEE^{\cofr}
\ar@<-.5ex>[l]_-{P}}
\]
where $P$ and $Q$ are the identity functors on underlying categories
and the right derived $Q^R$ is fully faithful (Remark~\ref{locs-are-locs}(3)).
Consider
the following two diagrams of solid arrows. 
\begin{equation*}%\label{bigdiagram}
\xymatrix@M=.7pc{\mKK \ar@<0ex>@/^1pc/[dr]_{\Hom(\mGG,-)} 
   \ar@{.>}@<-1.5ex>@/_2pc/[dd]_{\overline{\Hom(\mGG,-)}} & \\
   & \MODD\mEE\ar@<-.5ex>[d]_{U} \ar@<-1ex>@/_1pc/[ul]_{-\sm_\mEE\mGG} \\
\mL_\mCC\MODD\mEE^{\cofr} \ar@<-.5ex>[r]_-{Q} 
   \ar@{.>}@<0ex>@/^2pc/[uu]_{\overline{-\sm_\mEE\mGG}}
   &\MODD\mEE^{\cofr}\ar@<-.5ex>[l]_-{P}\ar@<-.5ex>[u]_{V}}
\quad\xymatrix@M=.7pc{\Ho\mKK\ar@<-.5ex>[d]_{\tF}
   \ar@<0ex>@/^1pc/[dr]_{F} \ar@{.>}@<-4.5ex>@/_2pc/[dd]_{\bar{F}} & \\
\DAE\ar@<-1ex>@{^{(}->}[r]_R \ar@<-.5ex>[u]_{\tJ} 
   & \DE\ar@<-1ex>[l]_L\ar@<-.5ex>[d]_{U^R} \ar@<-1ex>@/_1pc/[ul]_J \\
\Ho (\mL_\mCC\MODD\mEE^{\cofr}) \ar@<-.5ex>[r]_-{Q^R}\ar@{.>}[u]_H
   \ar@{.>}@<3ex>@/^2pc/[uu]_{\bar{J}}
   &\mD((\mEE^{\cofr})^{\op} )\ar@<-.5ex>[l]_-{P^L}\ar@<-.5ex>[u]_{V^L}}
\end{equation*}

The left one is a diagram
of model categories, the right one a diagram of homotopy categories,
resp.~triangulated
categories. The right diagram contains all derived functors from the left
one, but there are also functors which are only defined on the
triangulated level since $\DAE$ has only been defined as a triangulated
category (Definition \ref{Appendix}) and not as the homotopy
category of any model category. 
The adjoint pairs $(J,F)$, $(L,R)$ and $(\tJ,\tF)$
have been studied in Section~4; $\tJ$ and $\tF$ are inverse triangulated
equivalences (Proposition~\ref{eq}).

If we apply the composition $J V^L$ of left derived functors
to a map in $\mCC$ we
get the map
$
J(\coprod_{i\in I}F(G_i)) \to JF(\coprod_{i\in I}G_i)
$.
Since $J=\tJ L$, this map is an isomorphism by Proposition~\ref{locprop}.
Hence the universal property of the Bousfield localization yields a
Quillen functor pair $(\overline{-\sm_\mEE\mGG},\overline{\Hom(\mGG,-)})$
such that $(\overline{-\sm_\mEE\mGG})\circ P=(-\sm_\mEE\mGG)\circ V$ (and hence
$Q\circ(\overline{\Hom(\mGG,-)})=U\circ\Hom(\mGG,-)$). 

Our goal is to
show that $\overline{\Hom(\mGG,-)}$ is a Quillen equivalence.
It suffices to check that $\bar{J}$ is a triangulated equivalence.
By Corollary~\ref{Oxford}, the essential image of $Q^R$ 
is the same as $\mWW'\loc$. 
Since $\DAE$ is the same
as $\mWW\loc$ (Proposition~\ref{locprop}) and 
$\mWW\subset\DE$ corresponds to
$\mWW'\subset\mD((\mEE^{\cofr})^{\op})$ via the equivalences $V^L$
and $U^R$, we get an induced equivalence of categories
$H:\Ho(\mL_\mCC\MODD\mEE^{\cofr})\to\mD((\mEE^{\cofr})^{\op})$ such that
$RH\iso V^L Q^R$. 
Recall that the right adjoint $Q^R$ is fully faithful 
and the counit is hence an isomorphism $P^LQ^R\iso\id$.
Moreover, the equality 
$(\overline{-\sm_\mEE\mGG})\circ P=(-\sm_\mEE\mGG)\circ V$
gives us an isomorphism $\bar{J}P^L\iso JV^L$. Using these, we obtain
an isomorphism
\[
\tJ H\iso JRH\iso JV^LQ^R\iso\bar{J}P^LQ^R\iso\bar{J}
\]
and thus $\bar{J}$ is an equivalence since $\tJ$ and $H$ are equivalences.
This shows that the Quillen pair
$(\overline{-\sm_\mEE\mGG},\overline{\Hom(\mGG,-)})$ 
is indeed a Quillen equivalence.

%We have the following logical equivalences:
%\begin{eqnarray}\nonumber
%\nonumber&&X \textrm{ is }\mL_\mCC \textrm{-fibrant}\\
%& \Longleftrightarrow& X \textrm{ is }\mCC\textrm{-local } \\
%& \Longleftrightarrow&X \textrm{ is fibrant and for every map } f:A\to B 
%\textrm{ of } \mCC , \\
%\nonumber&&\textrm{the induced map }
%f^\ast:\map(B,X)\to\map(A,X)\\
%\nonumber&&\textrm{is a weak equivalence of simplicial sets}\\
%&\Longleftrightarrow&X\textrm{ is fibrant and for every map } f:A\to B 
%\textrm{ of } \mCC ,\textrm{ the induced  map}\\
%\nonumber&&\pi_n(f^\ast):\pi_n\map(B,X)\to\pi_n\map(A,X)\\
%\nonumber&&\textrm{of homotopy groups is an isomorphism for }n\geq 0\\
%&\Longleftrightarrow&X\textrm{ is fibrant and for every map } f:A\to B 
%\textrm{ of } \mCC \textrm{, the induced map }\\ 
%\nonumber&&f^\ast:[B,X]^{\mD((\mEE^{\cofr})^{\op})}
%\to[A,X]^{\mD((\mEE^{\cofr})^{\op})}
%\textrm{ is an isomorphism }\\
%&\Longleftrightarrow&X\textrm{ is fibrant and for every map } g:A'\to B' 
%\textrm{ of } \mCC \textrm{, the induced map }\\ 
%\nonumber&&g^\ast:[B',V(X)]^{\mD(\mEE^{\op})}
%\to[A',V(X)]^{\mD((\mEE^{\op})}
%\textrm{ is an isomorphism }\\
%&\Longleftrightarrow&V(X)\in\mWW\loc =\DAE
%\end{eqnarray}

\end{proof}

\begin{cor}
Let $\mKK$ be a model category. Then the following are equivalent.
\begin{enumerate}
\item[\textnormal{(i)}] $\mKK$ is Quillen equivalent to 
a well generated spectral model category.
\item[\textnormal{(ii)}] $\mKK$ is Quillen equivalent to a 
Bousfield localization $\mL_\mCC\MODD\mRR$ for 
some pointwise cofibrant spectral category $\mRR$ and 
some set $\mCC$ of morphisms in $\MODD\mRR$
whose image in $\mD(\mRR^{\op})=\Ho\MODD\mRR$ 
is, up to isomorphism, closed under (de-)sus\-pen\-sions.
\end{enumerate}
\end{cor}

\begin{proof}
(i) $\Rightarrow$ (ii) This is Theorem~\ref{lift}. Its proof shows 
that the image of $\mCC$ 
in the homotopy category is indeed closed, up to isomorphism, under 
(de-)sus\-pen\-sions. 

(ii) $\Rightarrow$ (i) 
We may assume that the domains and codomains of the elements
of $\mCC$ are cofibrant (otherwise we could replace them cofibrantly, this
would not have any effect on the localization). Let $\mWW$ denote
the image of $\mCC$ in $\mD(\mRR)$.
By Lemma~\ref{letztes-Lemma}, $\mL_\mCC\MODD\mRR$ is spectral. Its homotopy
category is a localization of the compactly
generated triangulated category $\mD(R^{\op})=\Ho\MODD\mRR$ by 
Corollary~\ref{Oxford}. Here, the acyclics are generated by the set
containing one cofiber for each map in $\mWW$ (cf.~Lemma~\ref{Methusalix}).
Hence, by Theorem~\ref{characterization}, the homotopy category
of $\mL_\mCC\MODD\mRR$ is well generated.

%Let $\mRR$ be a spectral category and $\mCC$ a set of morphisms in
%$\MODD\mRR$, so that the localization $\mL_\mCC\MODD\mRR$ exists. 
%We claim that the localization $\mL_\mCC\MODD\mRR$ is well
%generated, that is, its homotopy category is well generated as a triangulated
%category. Note that since $\mKK$ is spectral (and hence stable by 
%Lemma~\ref{Gutemine}),
%$\mL_\mCC\MODD\mRR$ is stable and thus its homotopy category 
%is triangulated. By Proposition~\ref{Neeman's} and 
%Lemma~\ref{mylemma}(b)(iii) it suffices to show that
%$\Ho\mL_\mCC\MODD\mRR$ is a localization of $\Ho\MODD\mRR$ whose
%acyclics are generated by a set. The functor $Q^R:\Ho\mL_\mCC\MODD\mRR
%\to\Ho\MODD\mRR$ is fully faithful (Remark~\ref{locs-are-locs}(3)), that
%means, we have indeed a localization of triangulated categories.
%We may assume that the domains and codomains of the elements
%of $\mCC$ are cofibrant (otherwise we could replace them cofibrantly, this
%would not have any effect on the localization). Let $\mWW$ denote the set
%containing all iterated (de-)sus\-pen\-sions of maps which are in the
%image of $\mCC$ under the canonical functor $\MODD\mRR\to\Ho\MODD\mRR$.
%Then the same arguments as in the proof of Theorem~\ref{lift} show
%that the essential image of $Q^R$, i.e., the subcategory of local objects,
%is the same as $\mWW\loc$. Then the acyclics of the (triangulated)
%localization are generated by the set $\mSS$ consisting of one cofiber
%for each map in $\mWW$ (Lemma~\ref{Methusalix}).
\end{proof}

\newpage
%\part*{}
%\input{Appendix}
\appendix
\section{Module categories}
\subsection{The one object case}
We use the terminology of \cite[Chapter~4]{Hovey}. 
Let $\mCC$ be a closed symmetric monoidal category 
with $\sm$ denoting the monoidal product, $\mS$ the unit, and 
$\Hom$ the internal Hom-functor
(the \emph{set} of morphisms from $X$ to $Y$ will be denoted by $\mCC (X,Y)$).
Let
$R$, 
$S$ and $T$ be monoids therein. We can then consider module categories, 
even bimodule categories like $R \MMODD S$ (which is isomorphic to the category
of $R\sm S^{\op}$-modules). There are bifunctors
\begin{eqnarray}
\sm_S:R \MMODD S \;\times \; S \MMODD T \to R \MMODD T, \nonumber\\
\Hom_R:(R \MMODD S)^{\op}\;\times \;R \MMODD T 
\to S \MMODD T,\label{bifunctors}\\
\Hom_T:(S \MMODD T)^{\op}\;\times \; R \MMODD T \to R \MMODD S,\nonumber
\end{eqnarray}
where the last one should better be denoted by $\Hom_{T^{\op}}$ instead, 
but let us allow ourselves this slight abuse of notation.
The object
$X\sm_SY$ is defined as the coequalizer in $\mCC$ of the diagram
\[
\xymatrix@M=.7pc{X\sm S \sm Y\ar@<.5ex>[r] \ar@<-.5ex>[r] & X\sm Y}
\]
where the upper map uses the right action of $S$ on $X$ and 
the lower the left action of $S$ on $Y$. This gives an object in $\mCC$
which has a left action of $R$ via the left action of $R$ on $X$ and
a right action of $T$ via the right action of $T$ on $Y$.
Similarly, $\Hom_R(X,Y)$ is defined as the equalizer in $\mCC$ of the diagram
\[
\xymatrix@M=.7pc{\Hom (X,Y)\ar@<.5ex>[r] \ar@<-.5ex>[r] & \Hom(R\sm X,Y).}
\]
Here, both maps can be defined via their adjoint maps
\[
\xymatrix@M=.7pc{R\sm X\sm\Hom(X,Y)\ar@<.5ex>[r] \ar@<-.5ex>[r] &Y,}
\]
where the upper map is first using the action of $R$ on $X$ and then the
evaluation map, and the lower is first evaluation and then the action
of $R$ on $Y$. The left $S$-action on $\Hom_R(X,Y)$
comes from the right $S$-action on the contravariant variable $X$, the
right $T$-action comes from the right $T$-action on $Y$.
The three bifunctors (\ref{bifunctors}) give an adjunction of two variables
in the sense of \cite[Definition~4.1.12]{Hovey}. Note that the
forgetful functor and the functor given by smashing with the
free module yield an adjunction
\begin{equation}\label{ev}
\xymatrix@M=.7pc{R\MMOD\ar@<-.5ex>[r]  &\mCC\ar@<-.5ex>[l]_-{R\sm -}.}
\end{equation}

From now on let $\mCC$ be a closed symmetric monoidal \emph{model} category
\cite[Definition~4.2.6]{Hovey} which is cofibrantly generated 
\cite[Definition~2.1.17]{Hovey}, 
has 
only small objects (small in the sense of \cite{SS00}, i.e., 
$\kappa$-small with respect to the whole category for some cardinal $\kappa$),
 and satisfies the monoid axiom 
\cite[Definition~3.3]{SS00}. There are many examples of 
such model categories: simplicial sets, symmetric spectra, stable module
categories, chain complexes ($\mZ$-graded, unbounded, over some 
commutative ground ring), and others \cite[Section~5]{SS00}. We are mainly
interested in  symmetric
spectra and chain complexes.
Then both the module category over a 
fixed monoid
in $\mCC$ and the category of monoids in $\mCC$
have a cofibrantly generated 
model structure where fibrations, resp.~weak equivalences,
are just fibrations, resp.~weak equivalences, in the underlying
category $\mCC$ and cofibrations are determined by the lifting property
with respect to trivial fibrations \cite[Theorem~4.1]{SS00}. 
In particular, the adjunction (\ref{ev}) is a Quillen functor pair
(recall that we use the convention according to which the left adjoint arrow
is drawn above the right adjoint).
If $R$ is a monoid in $\mCC$, the homotopy category of $R \MMOD$ will
be called the \emph{derived} category of $R$ and we denote it by $\mD (R)$.

\begin{lemma}\label{lemma1}
If in \textnormal{(\ref{bifunctors})} $S$ is cofibrant 
in $\mCC$ then $\sm_S$ together with
$\Hom _R$ and $\Hom _T$ is  
a Quillen bifunctor (in the sense of
\cite[Definition~4.2.1]{Hovey}) and hence induces an adjunction of two 
variables on the level of homotopy categories.
\end{lemma}

\begin{proof}
We have to check the \emph{pushout product axiom} 
\cite[Definition~4.2.1]{Hovey}. It suffices to do this for the generating
cofibrations and the generating trivial cofibrations. 
Such a generating cofibration in $R \MMODD S$
is of the form
$R\sm A\sm S \cof R\sm B\sm S$ with $A\cof B$ a cofibration in $\mCC$, 
similarly 
for a generating trivial cofibration. If $S\sm X\sm T \cof S\sm Y\sm T$
is a generating cofibration in $S \MMODD T$, the pushout product map
is isomorphic to the map
\[
R\sm(A\sm S\sm Y\amalg_{A\sm S\sm X} B\sm S\sm X)\sm T 
\to R\sm(B\sm S\sm Y)\sm T.
\]
This is a cofibration in $R \MMODD T$ since $S$ is cofibrant
in $\mCC$ (hence smashing with $S$ preserves cofibrations),
the pushout product 
axiom holds in $\mCC$, and
$R\sm -\sm T\iso$ preserves cofibrations (as a left Quillen functor). 
If one of the above cofibrations
is a trivial one the same proof shows that the pushout product map
is a trivial cofibration.
\end{proof}

In particular, if $S$ is cofibrant, smashing 
over $S$ with a cofibrant bimodule gives a
left Quillen functor between bimodule categories.
If we do not assume $S$ to be cofibrant, the 
functor $\sm_S$ in (\ref{bifunctors}) is not a Quillen bifunctor 
in general. For example, if the unit $\mS$
is cofibrant in $\mCC$, the monoid $S$ is not cofibrant in $\mCC$,
and $R=T=\mS$, then $S$ 
is cofibrant as a right resp.~left $S$-module 
(take $R=\mS$ in the Quillen adjunction $(S\sm_R -,f^\ast)$ in
Lemma~\ref{lemma3}) 
whereas $S\sm_SS\iso S$
is not cofibrant in $\mCC$ by assumption. That is, $\sm_S$ is not
a Quillen bifunctor in this case.

However, in the case where $S$ is not necessarily assumed to be cofibrant,
smashing over $S$ with an $R\textnormal{-}S$-bimodule  
can be a left Quillen functor:

\begin{lemma}\label{lemma2}
Suppose that the unit $\mS$ in $\mCC$ is cofibrant. If moreover
$X\in R \MMODD S$ is cofibrant in $R\MMOD$ then the adjoint pair
\[
\xymatrix@M=.7pc{S\MMODD T\ar@<.5ex>[rr]^-{X\sm_S -}
  &&R\MMODD T \ar@<.5ex>[ll]^-{\Hom_R(X,-)}}
\]
is a Quillen pair.
\end{lemma} 
Recall that we use the convention according to 
which the left adjoint arrow is drawn
above the right adjoint.
The derived adjoint pair will be denoted by $$\xymatrix@M=.7pc{\mD 
(S\sm T^{\op}) 
\ar@<.5ex>[rr]^-{X\sm_S^L -}
&&\mD (R\sm T^{\op}) \ar@<.5ex>[ll]^-{\RHom_R(X,-)}}.$$

\begin{proof}
We show that $\Hom_R(X,-)$ is a right Quillen functor.
In the diagram
\[
\xymatrix@M=.7pc{S\MMODD T \ar[d] &&R\MMODD T \ar[ll]_-{\Hom_R(X,-)}\ar[d]\\
  \mCC && R\MMOD \ar[ll]_{\Hom_R(X,-)}}
\]
the vertical functors are the forgetful functors. They 
preserve and reflect fibrations and trivial fibrations by the definition of
the model structure on module categories. As a consequence of 
Lemma~\ref{lemma1}, $\Hom_R(X,-):R\MMOD \to \mCC$ is a right 
Quillen functor, so it 
preserves both fibrations and trivial fibrations. Now it follows that
\[
\Hom_R(X,-):R\MMODD T \to S\MMODD T
\]
also preserves them. 
\end{proof}

\begin{lemma}\label{lemma3}
Let $f:R\to S$ be a map of monoids in $\mCC$.
The induced functor (restriction of scalars)
\[
f^\ast: S\MMOD\to R\MMOD 
\]
has both a left adjoint $S\sm_R -$ and a right adjoint 
$\Hom_R(S,-)$. Moreover, $(S\sm_R -,f^\ast)$ is
always a Quillen pair, and $\left(f^\ast,\Hom_R(S,-)\right)$
is a Quillen pair whenever the unit $\mS$ is cofibrant in $\mCC$ and
$S$ is cofibrant in $R\MMOD$.
\end{lemma}

\begin{proof}
For the definition of the left adjoint (extension of scalars)
$S\sm_R -:R\MMOD\to S\MMOD$, the $S$-$S$-bimodule $S$ is considered
as an $S$-$R$-bimodule via restriction of scalars along the map
$\id\sm f^{\op}:S\sm R^{\op}\to S\sm S^{\op}$. Since (trivial) 
fibrations are just (trivial) fibrations in $\mCC$, they are
preserved by the restriction of
scalars functor $f^\ast$.

For the definition of the right adjoint $\Hom_R(S,-):R\MMOD\to S\MMOD$,
we consider $S$ as an $R$-$S$-bimodule. Now $f^\ast$ is the same as the
functor $S\sm_S -:S\MMOD\to R\MMOD $, which has $\Hom_R(S,-)$ as
a right adjoint. If $\mS$ is cofibrant in $\mCC$ and $S$ is cofibrant
as an $R$-module, we can apply Lemma~\ref{lemma2} to deduce that
$\left(f^\ast,\Hom_R(S,-)\right)$ is a Quillen pair.
\end{proof}

%\begin{cor}\label{cor1}
%Suppose that $\mS$ is cofibrant in $\mCC$. Given a map of 
%monoids $f:R\to S$ such
%that $S$ is cofibrant as an $R$-module. Then
%\[
%\xymatrix@M=.7pc{S\MMOD\ar@<.5ex>[rr]^-{f^{\ast}}
%  &&R\MMOD \ar@<.5ex>[ll]^-{\Hom_R(S,-)}}
%\]
%is a Quillen pair. $\hfill \square$
%\end{cor}
%Hence in this case, $f^{\ast}$ is both a left and a right Quillen functor.

Consider the special case of the map of monoids $\iota: R\to R\sm S^{\op}$. The
right adjoint of the restriction functor is 
$\Hom_R(R\sm S,-)\iso\Hom_{\mCC}(S,-)$.
\begin{cor}\label{cor2}
Suppose that $\mS$ is cofibrant in $\mCC$. 
Let $R$ and $S$ be monoids such that 
$S$ is cofibrant in $\mCC$. Then we have a Quillen pair
\[
\xymatrix@M=.7pc{R\MMODD S\ar@<.5ex>[rr]^-{\iota^\ast}
  &&R\MMOD \ar@<.5ex>[ll]^-{\Hom(S,-)}.}
\]
In particular, a cofibrant $R$-$S$-bimodule is then also cofibrant as an
$R$-module.
$\hfill \square$
\end{cor}

Many monoidal model categories have the property that smashing with
a cofibrant module $X$ over some monoid $R$ preserves weak equivalences.
%We call this property the \emph{smashing axiom}. 
The functor $-\sm_RX$
induces then a functor between the homotopy categories without being
a Quillen functor in general. Symmetric spectra have this smashing
property \cite[Lemma~5.4.4]{HSS} and there are many other examples 
\cite[Section~5]{SS00} including chain complexes.

\subsection{The several objects case}
Let us now consider the case of monoids with several objects, which
we need for Part~2 of this paper. We use
results from \cite[Sections~6 and 7]{SS03'}; results for the case
of spectral categories can also be found in \cite[Appendix~A]{SS03}.
As in Section~A.1, we fix a closed symmetric monoidal model category
$(\mCC, \sm, \mS)$ which is cofibrantly generated, 
has only small objects, and satisfies
the monoid axiom. The reader is encouraged to think of $\mCC$ as the category
of symmetric spectra (cf.~Example~\ref{Verleihnix}(1)).

A $\mCC$\emph{-category} or a \emph{monoid in} $\mCC$ \emph{with several
objects} \cite[Definition~6.2.1]{Borceux2} is a small category $\mRR$
that is enriched
over $\mCC$. This means, for any two objects $R$ and $R'$ of $\mRR$
there is a $\Hom$-object $\mRR(R,R')$ in $\mCC$
together with an identity `element'
$\mS\to\mRR (R,R)$ for each $R$ in $\mRR$ and composition morphisms
\[
\mRR (R',R'') \sm \mRR (R,R') \to \mRR(R,R'')
\]
which are associative and unital with respect to the identity elements.
When $\mRR$ has only one object, it is the same as a monoid in $\mCC$.
A \emph{morphism} of $\mCC$-categories is a $\mCC$-functor 
\cite[Definition~6.2.3]{Borceux2}.
If $\mCC$ is the category of symmetric spectra, $\mCC$-categories are called
spectral categories; if $\mCC$ is the category of chain complexes
(Example~\ref{Verleihnix}(2)), then $\mCC$-categories are called
DG categories.
We have required the smallness condition (i.e., the objects of
a $\mCC$-category form a set)
since we want to consider module categories over $\mCC$-categories.

A \emph{left module} over a $\mCC$-category $\mRR$ is a $\mCC$-functor
$X:\mRR\to\mCC$, i.e., an object $X(R)$ in $\mCC$ for every $R\in\mRR$
and morphisms in $\mCC$
\[
\mRR(R,R')\to\Hom_{\mCC}\left(X(R),X(R')\right)
\]
which are compatible with composition and identities.
By adjunction,
these maps correspond to a left action of $\mRR$ on X, i.e., maps in $\mCC$
\[
\mRR(R,R')\sm X(R)\to X(R')
\]
which are associative and unital. A \emph{right module} over $\mRR$ is
a left $\mRR^{\op}$-module. A morphism $X\to Y$ of $\mRR$-modules
is a family $X(R)\to Y(R)$ of maps in $\mCC$ compatible with the action
of $\mRR$.
An important point which distinguishes the several objects case 
from the one object case is, that there is 
not just one free $\mRR$-module but 
\emph{one for each object} of $\mRR$, namely
the \emph{free module} $F^\mRR_R$ with respect to $R$. It is defined by
$F^\mRR_R(R')=\mRR(R,R')$ (hence we have $F^{\mRR^{\op}}_R(R')=\mRR(R',R)$ 
for the free right modules). We will sometimes omit the upper index
and just write $F_R$. Note that the enriched Yoneda lemma yields 
an adjoint pair
\begin{equation}\label{ev-sov}
\xymatrix@M=.7pc{\mRR\MMOD\ar@<-.5ex>[rr]_-{\ev_R}
&&\mCC \ar@<-.5ex>[ll]_-{F_R\sm -}}
\end{equation}
where $\ev_R$ is the evaluation functor with
$\ev_R(X)=X(R)$ and $F_R\sm Y$ is given by
\[
(F_R\sm Y)(R')=F_R(R')\sm Y=\mRR(R,R')\sm Y
\]
with the obvious left action of $\mRR$.

Schwede and Shipley \cite[Theorem~6.1]{SS03'} have shown that
the category $\mRR\MMOD$ of $\mRR$-modules is a cofibrantly generated
model category with weak equivalences and fibrations defined objectwise.
As usual, the cofibrations are determined by the lifting property.
A set of generating (trivial) cofibrations is given by all maps
of the form
\[
F_R\sm A\to F_R\sm B
\]
for $A\to B$ a generating (trivial) cofibration in $\mCC$.
Note that the adjunction (\ref{ev-sov}) is indeed a Quillen functor pair.
The homotopy category of $\mRR\MMOD$ is called the \emph{derived}
category of $\mRR$ and we denote it by $\mD(\mRR)$.

The \emph{smash product} $\mRR\sm\mSS$
of two $\mCC$-categories $\mRR$ and $\mSS$
\cite[Section~A.2]{SS03} has as set of objects the product
of the sets of objects of $\mRR$ and $\mSS$. The morphism objects are
given by
\[
{\mRR\sm\mSS}\left((R,S),(R',S')\right)=\mRR(R,R')\sm\mSS(S,S').
\]
This allows us to consider bimodule categories as $\mRR\MMODD\mSS$ with
objects the $\mRR\sm\mSS^{\op}$-modules. Note that a spectral category
$\mRR$ can itself be regarded as an $\mRR\sm\mRR^{\op}$-module in a natural
way. For $\mRR$, $\mSS$ and $\mTT$ spectral categories, there are,
as in the one object case, bifunctors
\begin{eqnarray}
\sm_\mSS:\mRR \MMODD \mSS \;\times \; \mSS \MMODD \mTT \to \mRR \MMODD \mTT, 
\nonumber\\
\Hom_\mRR:(\mRR \MMODD \mSS)^{\op}\;\times \;\mRR \MMODD \mTT 
\to \mSS \MMODD \mTT,\label{bifunctors-sov}\\
\Hom_\mTT:(\mSS \MMODD \mTT)^{\op}\;\times \; \mRR \MMODD \mTT \to \mRR 
\MMODD \mSS,\nonumber
\end{eqnarray}
which form an adjunction of two variables
\cite[Definition~4.1.12]{Hovey}. 
For example, if $X\in\mRR\MMODD\mSS$ and $Y\in\mSS\MMODD\mTT$, then
$(X\sm_\mSS Y)(R,T)$ is the coequalizer of the diagram
\[
\xymatrix@M=.7pc{\coprod\limits_{S,S'\in\mSS}X(R,S)\sm\mSS(S',S)\sm Y(S',T)
\ar@<1.5ex>[r]\ar@<0.5ex>[r]&\coprod\limits_{S''\in\mSS}X(R,S'')\sm Y(S'',T)}
\]
where the upper map uses the right action of $\mSS$ on $X$ and the lower
the left action of $\mSS$ on $Y$.
The left action of $\mRR$ on $X$ and the right action of $\mTT$ on $Y$
yield the $\mRR\sm\mTT^{\op}$-module structure.
Similarly, $\Hom_\mRR$ and $\Hom_\mTT$ (which should be more precisely
denoted by $\Hom_{\mTT^{\op}}$) are given by certain equalizers.

To state the next lemma, which is an analog to Lemma~\ref{lemma1},
we need the following

\begin{defn}\label{pointwise-cofibrant}
A $\mCC$-category is \emph{pointwise cofibrant} if $\mRR(R,R')$ is cofibrant
in $\mCC$ for all $R,R'\in\mRR$.
\end{defn}

\begin{lemma}\label{lemma1-sov}
If in \textnormal{(\ref{bifunctors-sov})} $\mSS$ is pointwise cofibrant,
then $\sm_\mSS$ together with
$\Hom _\mRR$ and $\Hom _\mTT$ is  
a Quillen bifunctor (in the sense of
\cite[Definition~4.2.1]{Hovey}) 
and hence induces an adjunction of two 
variables on the level of homotopy categories.
\end{lemma}

\begin{proof}
The proof is a generalization of the proof of Lemma~\ref{lemma1}. We have to
verify the pushout product axiom 
\cite[Definition~4.2.1]{Hovey} for generating (trivial) cofibrations.
Let $A\cof B$ and $X\cof Y$ be cofibrations in $\mCC$. We have
to consider the pushout product of the maps 
\[
F^\mRR_R\sm A\sm F^{\mSS^{\op}}_S\to F^\mRR_R\sm B\sm F^{\mSS^{\op}}_S
\quad \textrm{and}\quad
F^\mSS_{S'}\sm X\sm F^{\mTT^{\op}}_T\to F^\mSS_{S'}\sm Y\sm F^{\mTT^{\op}}_T
\]
for $R\in\mRR$, $S,S'\in\mSS$ and $T\in\mTT$. By definition,
$F^{\mSS^{\op}}_S\sm_\mSS F^\mSS_{S'}$
is the coequalizer of the diagram
\[
\xymatrix@M=.7pc{\coprod\limits_{S_1,S_2\in\mSS}\mSS(S_1,S)\sm\mSS(S_2,S_1)
\sm \mSS(S',S_2)
\ar@<1.5ex>[r]\ar@<0.5ex>[r]&\coprod\limits_{S_3\in\mSS}\mSS(S_3,S)\sm 
\mSS(S',S_3)}
\]
which is just $\mSS(S',S)$.
Moreover, as a left adjoint, $F^\mRR_{R}\sm - \sm F^{\mTT^{\op}}_T
\iso F^{\mRR\sm\mTT^{\op}}_{(R,T)}$
preserves colimits. Thus the pushout product map is isomorphic
to
\begin{eqnarray*}
F^\mRR_R\sm\left(A\sm\mSS(S',S)\sm Y\underset{A\sm\mSS(S',S)\sm X}{\amalg}
B\sm\mSS(S',S)\sm Y\right)\sm F^{\mTT^{\op}}_T\hspace{2cm}\\
\hspace{2cm}\to
F^\mRR_R\sm\left(B\sm\mSS(S',S)\sm Y\right)\sm F^{\mTT^{\op}}_T
\end{eqnarray*}
But since $\mSS(S',S)$ is cofibrant by assumption and the pushout
product axiom holds in $\mCC$,
this map is just the image of a cofibration in $\mCC$ under
the left Quillen functor
\[
F^\mRR_{R}\sm - \sm F^{\mTT^{\op}}_T:\mCC\to\mRR\MMODD\mTT
\]
 and hence
a cofibration. The same arguments show that the pushout product map
is a trivial cofibration whenever one of the cofibrations $A\cof B$ and 
$X\cof Y$ is trivial.
\end{proof}
Consequently, if $\mSS$ is pointwise cofibrant, smashing over $\mSS$
with a cofibrant bimodule gives a left Quillen functor. 

We now want to prove
an analog of Lemma~\ref{lemma2}. For this purpose we need a further
notion. If $\mRR$ is $\mCC$-category, then $\mS_\mRR$
denotes the `discrete' 
$\mCC$-category associated to $\mRR$, i.e., $\mS_\mRR$ has the same objects as
$\mRR$, and $\mS_\mRR(R,R')$ is the unit $\mS$ if $R=R'$ and the
trivial object $\ast$ otherwise \cite[Section~A.1]{SS03}.
An $\mS_\mRR$-module is simply a family of objects in $\mCC$ 
indexed by the objects 
of $\mRR$, and $\MODD \mS_\mRR$ carries the product model structure, that is,
weak equivalences, cofibrations and fibrations are defined objectwise.
There is a canonical morphism $f:\mS_\mRR\to\mRR$ of $\mCC$-functors
induced by
the unit maps. It induces the restriction of scalars functor
\[
f^\ast:\mRR\MMOD\to\mS_\mRR\MMOD ,\quad X\mapsto X\circ f .
\]

\begin{lemma}\label{lemma2-sov}
Suppose that the unit $\mS$ in $\mCC$ is cofibrant. 
If moreover
$X\in\mRR\MMODD\mSS$ is cofibrant in $\mRR\MMODD\mS_\mSS$ then
the adjoint pair
\[
\xymatrix@M=.7pc{\mSS\MMODD \mTT\ar@<.5ex>[rr]^-{X\sm_\mSS -}
  &&\mRR\MMODD\mTT \ar@<.5ex>[ll]^-{\Hom_\mRR(X,-)}}
\]
is a Quillen pair.
\end{lemma}

\begin{proof}
Note that via restriction of scalars along $\mRR\sm\mS_\mSS
\to\mRR\sm\mSS$, 
the module $X$ can indeed also be considered
as an object of $\mRR\MMODD\mS_\mSS$.
We show that $\Hom_\mRR(X,-)$ is a right Quillen functor.
In the diagram
\[
\xymatrix@M=.7pc{\mSS\MMODD\mTT \ar[d] 
&&\mRR\MMODD\mTT \ar[ll]_-{\Hom_\mRR(X,-)}\ar[d]\\
\mS_\mSS\MMODD\mS_\mTT && \mRR\MMODD\mS_\mTT \ar[ll]_{\Hom_\mRR(X,-)}}
\]
the vertical functors are induced by restriction of scalars. Hence
they preserve and reflect fibrations and trivial fibrations.
Since $\mS$ was assumed to be cofibrant, the discrete $\mCC$-category
$\mS_\mSS$ is pointwise cofibrant. Hence we can apply 
Lemma~\ref{lemma1-sov}, which implies that $\Hom_\mRR(X,-):\mRR\MMODD\mS_\mTT 
\to\mS_\mSS\MMODD\mS_\mTT$ is a right Quillen functor since $X$ is
cofibrant in $\mRR\MMODD\mS_\mSS$, so it preserves fibrations and trivial
fibrations. Consequently, $\Hom_\mRR(X,-):\mRR\MMODD\mTT
\to\mSS\MMODD\mTT$ preserves them as well.
\end{proof}

\begin{lemma}\label{lemma3-sov}
Let $f:\mRR\to\mSS$ be a morphism of $\mCC$-categories.
Then the induced functor (restriction of scalars)
\[
f^\ast:\mSS\MMOD\to\mRR\MMOD ,\quad X\mapsto X\circ f ,
\]
has both a left adjoint $\mSS\sm_\mRR -$ and a right adjoint 
$\Hom_\mRR(\mSS,-)$. Moreover, $(\mSS\sm_\mRR -,f^\ast)$ is
always a Quillen pair, and $\left(f^\ast,\Hom_\mRR(\mSS,-)\right)$
is a Quillen pair whenever the unit $\mS$ is cofibrant in $\mCC$ and
$\mSS$ is cofibrant in $\mRR\MMODD\mS_\mSS$.
\end{lemma}

\begin{proof}
Note that for the definition of the left adjoint (extension
of scalars) $\mSS\sm_\mRR -:\mRR\MMOD\to\mSS\MMOD$,
the $\mSS\sm\mSS^{\op}$-module $\mSS$
is considered as an $\mSS\sm\mRR^{\op}$-module via restriction
of scalars along the map $\id\sm f^{\op}:\mSS\sm\mRR^{\op}\to\mSS\sm\mSS^{\op}$.
Clearly, $f^\ast$ preserves fibrations and trivial fibrations and 
$(\mSS\sm_\mRR -,f^\ast)$ is thus a Quillen pair.

For the definition of the right adjoint $\Hom_\mRR(\mSS,-)$, we consider
$\mSS$ as an $\mRR\sm\mSS^{\op}$-module.
Restriction
of scalars is the same as the functor $\mSS\sm_\mSS -:\mSS\MMOD\to\mRR\MMOD$.
This functor has $\Hom_\mRR(\mSS,-):\mRR\MMOD\to\mSS\MMOD$ as a right
adjoint, which is by Lemma~\ref{lemma2-sov} 
a right Quillen  functor, whenever $\mS$ is cofibrant
in $\mCC$ and $\mSS$ is cofibrant in $\mRR\MMODD\mS_\mSS$.
\end{proof}

The following corollary is used in Section~5.2. Note that if
$\mCC$ is the category of symmetric spectra, the unit $\mS$ is 
indeed cofibrant.

\begin{cor}\label{objectwise}
For each object $R$ in a $\mCC$-category $\mRR$, 
the evaluation functor 
\[
\ev_R:\mRR\MMOD\to\mCC
\]
 with
$\ev_R(X)=X(R)$ preserves fibrations, weak equivalences, limits and colimits. 
Moreover, if the unit $\mS$ is cofibrant in $\mCC$
and $\mRR$ is pointwise cofibrant 
(Definition~\textnormal{\ref{pointwise-cofibrant}}),
$\ev_R$ preserves cofibrations and trivial cofibrations.
\end{cor}

\begin{proof}
The assumption that $\mRR$ is pointwise cofibrant 
implies that $\mRR$, considered as an
bimodule, is cofibrant in $\mS_\mRR\MMODD\mS_\mRR$ (in fact, both conditions
are equivalent).
Hence we can apply 
Lemma~\ref{lemma3-sov} to the canonical morphism of $\mCC$-categories
$f:\mS_\mRR\to\mRR$ and use the fact that in a module category
over a discrete $\mCC$-category everything (limits, colimits, fibrations, 
cofibrations, and weak equivalences) is defined objectwise.
\end{proof}

\newpage
\bibliographystyle{alpha}
\bibliography{literatur}
%\newpage
%\input{Summary}
%\newpage
%\cleardoublepage
%\input{Lebenslauf}
\end{document}